\documentclass[final,leqno]{siamltex}
\usepackage{mathrsfs}
\usepackage{amsfonts}
\usepackage{stmaryrd}
\usepackage{amsmath}
\usepackage{amssymb}
\usepackage{multirow}
\usepackage{cite}
\usepackage{tabularx}
\usepackage{stmaryrd}
\usepackage{booktabs}
\usepackage{cite}
\usepackage{amssymb}
\usepackage{verbatim}
\usepackage{hyperref}
\usepackage[all,cmtip]{xy}
\usepackage{color}
\usepackage{graphicx}
\usepackage{subfigure}
\usepackage{cases}
\DeclareGraphicsExtensions{.eps}

\numberwithin{equation}{section}

\newcommand{\bx}{{\bf x} }

\newcommand{\nn}{\nonumber}

\newcommand{\lag}{\langle}
\newcommand{\rag}{\rangle}

\newcommand{\wh}{\widehat}

\renewcommand{\theequation}{\arabic{section}.\arabic{equation}}
\newcommand{\be}{\begin{equation}}
\newcommand{\ee}{\end{equation}}
\newcommand{\ba}{\begin{array}}
\newcommand{\ea}{\end{array}}
\newcommand{\bea}{\begin{eqnarray}}
\newcommand{\eea}{\end{eqnarray}}
\newcommand{\beas}{\begin{eqnarray*}}
\newcommand{\eeas}{\end{eqnarray*}}
\newtheorem{remark}{Remark}[section]
\newtheorem{example}[theorem]{Example}

\newcommand{\eps}{\varepsilon}
\newcommand{\gm}{\gamma}
\newcommand{\bR}{{\mathbb R}}

\newcommand{\bN}{{\mathbb N}}
\newcommand{\bZ}{{\mathbb Z}}

\newcommand{\fl}{\frac}
\newcommand{\p}{\partial}
\newcommand{\fe}{\mathrm{e}}

\title{A uniformly and optimally accurate method for the Klein-Gordon-Zakharov system in simultaneous high-plasma-frequency and subsonic limit regime\thanks{This work was supported by the Alexander von Humboldt Foundation.}}
\author{Chunmei Su\thanks{Zentrum Mathematik, Technische Universit\"{a}t M\"unchen, 85748 Garching bei M\"unchen, Germany ({\tt sucm13@163.com})} \and
Xiaofei Zhao\thanks{School of Mathematics and Statistics, Wuhan University, 430072 Wuhan, China ({\tt matzhxf@whu.edu.cn})}}

\date{}
\begin{document}

\maketitle

\begin{abstract}
We present a uniformly and optimally accurate numerical method for solving the Klein-Gordon-Zakharov (KGZ) system with two dimensionless parameters $0<\eps\le1$ and $0<\gm\le 1$, which are inversely proportional to the plasma frequency and the acoustic speed, respectively. In the simultaneous high-plasma-frequency and subsonic limit regime, i.e. $\eps<\gm\to 0^+$, the KGZ system collapses to a cubic Schr\"odinger equation, and the solution propagates waves with $O(\eps^2)$-wavelength in time and  meanwhile contains rapid outgoing initial layers with speed $O(1/\gm)$ in space due to the incompatibility of the initial data. By presenting a multiscale decomposition of the KGZ system, we propose a multiscale time integrator Fourier pseduospectral method which is explicit, efficient and uniformly accurate for solving the KGZ system for all $0<\eps<\gm\leq1$. Numerical results are reported to show the efficiency and accuracy of scheme. Finally, the method is applied to investigate the convergence rates of the KGZ system to its limiting models when $\eps<\gm\to 0^+$.
\end{abstract}

\begin{keywords}
Klein-Gordon-Zakharov system, high-plasma-frequency limit, subsonic limit, multiscale decomposition, numerical scheme, uniformly accurate
\end{keywords}
\begin{AMS}
35L70, 65N12, 65N15, 65N35
\end{AMS}

\pagestyle{myheadings}\thispagestyle{plain}
\section{Introduction}\label{sec1}
The Zakharov type models are of paramount importance for studying the Langmuir turbulence in plasma dynamics \cite{Zak1,Zak3,Zak4,app2,Dendy,app3}. As one of them, the Klein-Gordon-Zakharov (KGZ) system was derived from the Euler-Maxwell equations to describe the interaction between Langmuir waves and ion sound waves in the plasma \cite{Colin,KGZ-limit,Zak1,Texier}. We shall consider in this work, the KGZ system ($d=1,2,3$) in its dimensionless form \cite{KGZ,KGZ-limit2,Colin,KGZ-limit}:
\begin{subequations}\label{KGZ}
  \begin{align}
    & \eps^2\partial_{tt}\psi(\bx,t)-\Delta \psi(\bx,t)+\frac{1}{\eps^2}\psi(\bx,t)
    +\psi(\bx,t)\phi(\bx,t)=0,\label{psi}\\
    & \gm^2\partial_{tt}\phi(\bx,t)-\Delta \phi(\bx,t)-\Delta \psi^2(\bx,t)=0,\quad x\in\bR^d,\quad t>0,\label{phi}\\
    &\psi(\bx,0)= \psi_0(\bx),\ \partial_t\psi(\bx,0)=\frac{\psi_1(\bx)}{\eps^2},\  \phi(\bx,0)= \phi_0(\bx),\ \partial_t\phi(\bx,0)=\frac{\phi_1(\bx)}{\gamma},
     \end{align}
\end{subequations}
where
$\psi:=\psi(\bx,t):\bR^d\times[0,\infty)\to\bR$ and $\phi:=\phi(\bx,t):\bR^d\times[0,\infty)\to\bR$ are the unknowns denoting respectively, the fast time scale component of the electric field and the deviation of ion density from a constant equilibrium. Here $0<\eps\leq1$ and $0<\gm\leq1$ are introduced \cite{Colin,KGZ-limit,KGZ} as two dimensionless parameters that are inversely proportional to the plasma frequency and the ion sound speed, respectively, and $\psi_0,\psi_1,\phi_0$ and $\phi_1$ are given real-valued initial functions which are bounded for $\eps,\gamma\in(0,1]$. As is well-known, the energy of the KGZ system (\ref{KGZ}) is conserved as
\begin{align}
E(t)&:=\int_{\bR^d} \left[\eps^2 \left(\partial_t\psi\right)^2 +
\left|\nabla \psi \right|^2 + \frac{1}{\eps^2}\psi^2 +
\frac{\gamma^2}{2} \left|\nabla \varphi\right|^2
+\frac{1}{2}\phi^2 + \phi  \psi^2 \right]d\bx\nonumber\\
&\equiv \int_{\bR^d} \left[\frac{1}{\eps^2}\psi_1^2 +
\left|\nabla \psi_0 \right|^2 + \frac{1}{\eps^2}\psi_0^2 +
\frac{1}{2} \left|\nabla \varphi_0\right|^2
+\frac{1}{2}\phi_0^2 + \phi_0  \psi_0^2 \right]d\bx=E(0),\  t\ge0,
\label{energy}
\end{align}
where $\varphi(\bx,t)$ solves $\Delta \varphi(\bx,t) =
\partial_t\phi(\bx,t)$ with $\displaystyle\lim_{|\bx|\to \infty}\varphi(\bx,t)=0$ and $\varphi_0(\bx)=\Delta^{-1}\phi_1(\bx)$.

In the literature, the KGZ system has been studied in different parameter regimes both analytically and numerically. In the classical regime of (\ref{KGZ}), i.e. $\eps=O(1)$ and $\gm=O(1)$, the well-posedness of the Cauchy problem has been established in \cite{Well2,Well1}, and numerical discretizations equipped with finite difference time domain method \cite{classical1} or finite element method \cite{classical2} or spectral element method \cite{classical3} or exponential wave integrator \cite{XZ} have been considered.
When $\gm=O(1)$ and $\eps\ll 1$, the KGZ system (\ref{KGZ}) is in the \emph{high-plasma-frequency limit regime}, and (\ref{KGZ}) has been proved to converge to the Zakharov system \cite{Colin,Colin2,Masmoudi2008} as $\eps\to0$. The solution of (\ref{KGZ}) in such regime propagates waves with wavelength at $O(\eps^2)$ in time, which causes severe numerical burden in computations, since classical schemes would require step size smaller than the wavelength. To enlarge the step size, a multiscale time integrator with uniform first order accuracy for $\eps\in(0,1]$ was proposed based on a decomposition by frequency in \cite{KGZ1}, and later a class of oscillatory integrators were proposed in \cite{Kath2019} to further overcome the numerical loss of derivative in rigorous error analysis. On the other hand, when $\eps=O(1)$ and $\gm\ll1$ in (\ref{KGZ}), which is known as the \emph{subsonic limit regime}, the KGZ system reduces to the nonlinear Klein-Gordon equation as $\gm\to0$ \cite{Daub}. In this regime, similar to the subsonic limit of Zakharov system  \cite{Added,Ozawa,Masmoudi2008,Weinstein,ZE3}, the solution of (\ref{KGZ}) propagates waves with wavelength at $O(\gm)$ in time and contains outgoing initial layers at speed $O(1/\gm)$ in space. To numerically handle the highly oscillatory behaviours here, an asymptotic consistent formulation was utilised to propose a finite difference method \cite{KGZ2} and a multiscale time integrator \cite{KGZ-subsonic1} with accuracy uniform for $\gm\in(0,1]$. The last but more challenging regime of the KGZ system is the \emph{simultaneous high-plasma-frequency and subsonic limit regime}, i.e. $\eps,\gm\ll1$ in (\ref{KGZ}). As $\eps,\gm\to0$, Masmoudi and Nakanishi showed the convergence of (\ref{KGZ}) to different limit equations under the critical case $\eps<\gamma$ \cite{KGZ-limit} and the super critical case $\eps>\gamma$ \cite{KGZ-limit2}.
In the critical case $\eps<\gamma\to0$, the KGZ system (\ref{KGZ}) converges to a cubic Schr\"{o}dinger equation \cite{KGZ-limit}:
 \be\label{Sch}
  \left\{ \begin{split}
&2i\p_t z_{\rm nls}(\bx,t)-\Delta z_{\rm nls}(\bx,t)-2|z_{\rm nls}(\bx,t)|^2z_{\rm nls}(\bx,t)=0,\quad \bx\in\bR^d,\  t>0,\\
    &z_{\rm nls}(\bx,0)=\fl{1}{2}\left(\psi_0(\bx)-i\psi_1(\bx)\right),
    \end{split}\right.
 \ee
in the sense that
\be\label{app}
\psi\to e^{it/\eps^2}z_{\rm nls}+e^{-it/\eps^2}\overline{z_{\rm nls}},\quad \phi\to -2|z_{\rm nls}|^2+I_{\rm nls},\quad
\eps<\gm\to0,
\ee
where $I_{\rm nls}$ is the free wave defined by
\begin{equation*}\left\{
\begin{split}
&\gm^2\p_{tt}I_{\rm nls}(\bx,t)-\Delta I_{\rm nls}(\bx,t)=0,\quad \bx\in\bR^d,\  t>0,\\
&I_{\rm nls}(\bx,0)=\phi_0(\bx)+2|z_{\rm nls}(\bx,0)|^2=\phi_0(\bx)+\frac{1}{2}[\psi_0^2(x)+\psi_1^2(x)],\\
& \p_t I_{\rm nls}(\bx,0)=\phi_1(\bx)/\gm.
\end{split}\right.
\end{equation*}
The asymptotic behaviour of the solution (\ref{app}) in the limit regime $\eps<\gm\ll1$ indicates that the solution $\psi$ propagates waves with wavelength at $O(\eps^2)$ in time and $\phi$ contains a fast outgoing initial layer with speed at $O(1/\gm)$ in space.
The amplitude of the initial layer is determined by the incompatibility of the given initial data in (\ref{KGZ}), which has a remarkable influence on the behaviour of the solution and the convergence rate in (\ref{app}). To illustrate this, we take an one-dimensional example: $d=1$, $\bx=x$ in (\ref{KGZ}), $\gamma=2\eps$ and
\be\label{psi0}
\psi_0(x)=\mathrm{sech}(x^2),\quad \psi_1(x)=\frac{\fe^{-x^2}}{2},\quad x\in\bR,
\ee
with the following two cases of $\phi_0(x)$ and $\phi_1(x)$:

(i) compatible initial data:
\begin{equation}\label{compatiable}
\begin{split}
&\phi_0(x)=-\frac{1}{2}(\psi_0^2(x)+\psi_1^2(x)),\quad \phi_1(x)=-4\gamma \mathrm{Re}\left(z_{\rm nls}(x,0)\overline{\partial_tz_{\rm nls}}(x,0)\right),
\end{split}
\end{equation}
which perfectly matches with the limit (\ref{app}) in initial position and derivative. Here $\mathrm{Re}(f)$ represents the real part of $f$.

(ii) incompatible initial data:
\be\label{E1}
\phi_0=-\frac{1}{2}(\psi_0^2(x)+\psi_1^2(x))+\rho(x),\quad
\phi_1(x)=-4\gamma \mathrm{Re}\left(z_{\rm nls}(x,0)\overline{\partial_tz_{\rm nls}}(x,0)\right),\ee
 where we add the incompatibility
$$
\rho(x)=g\left(\fl{x+18}{10}\right)g\left(\fl{18-x}{9}\right)\cos(2x+\pi/4),\quad
g(x)=\fl{f(x)}{f(x)+f(1-x)},
$$
with $f(x)=e^{-1/x}\chi_{(0,\infty)}$ and $\chi_\Omega$ being the characteristic function of the domain $\Omega$.  Figure \ref{fig:incomp} displays the profiles of the solutions in the two cases under different $\eps$.
\begin{figure}[t!]
\centerline{Compatible case:}
\begin{minipage}[t]{0.5\linewidth}
\centering
\includegraphics[height=4cm,width=6.75cm]{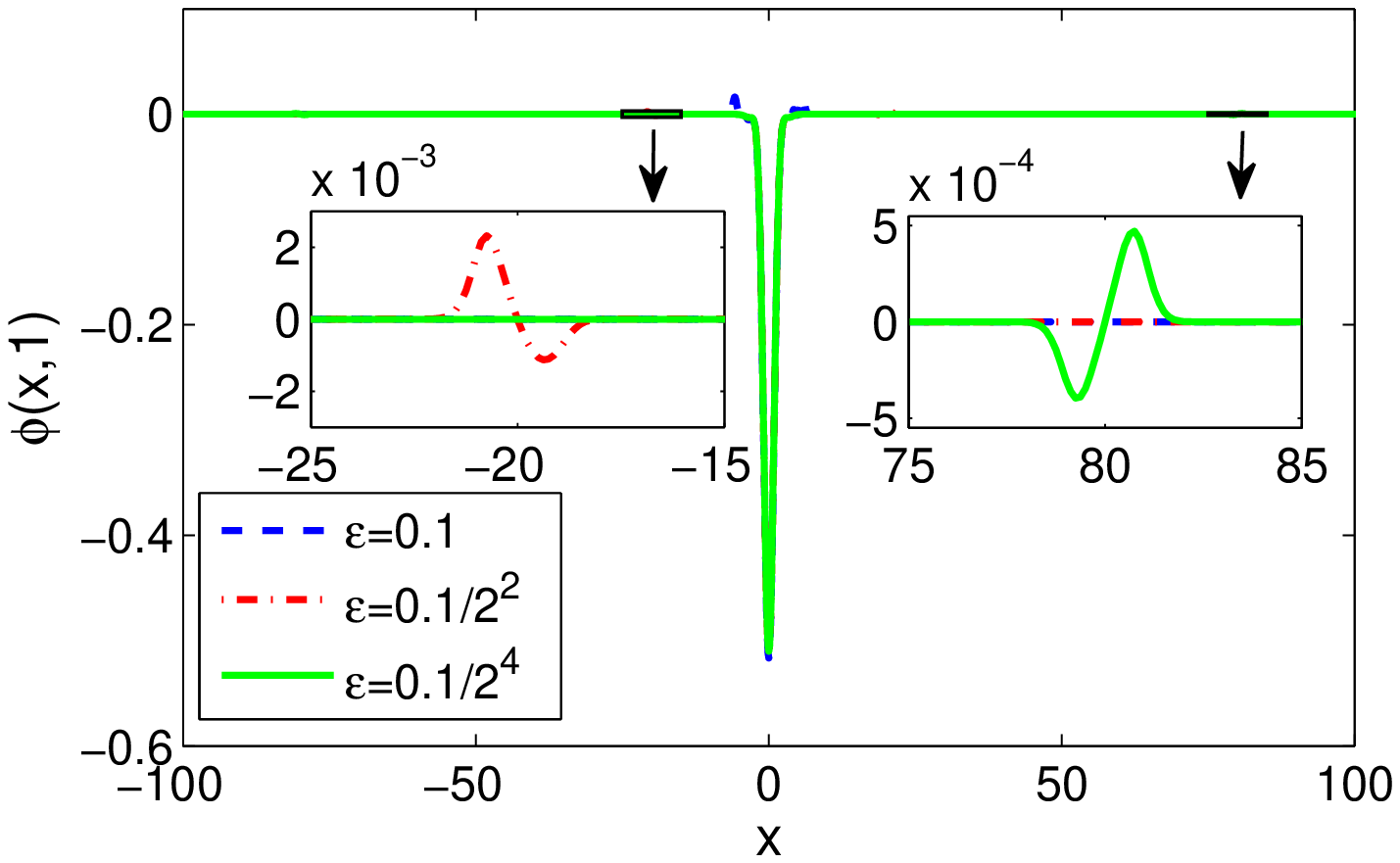}
\end{minipage}%
\hspace{1mm}
\begin{minipage}[t]{0.5\linewidth}
\centering
\includegraphics[height=4cm,width=6.9cm]{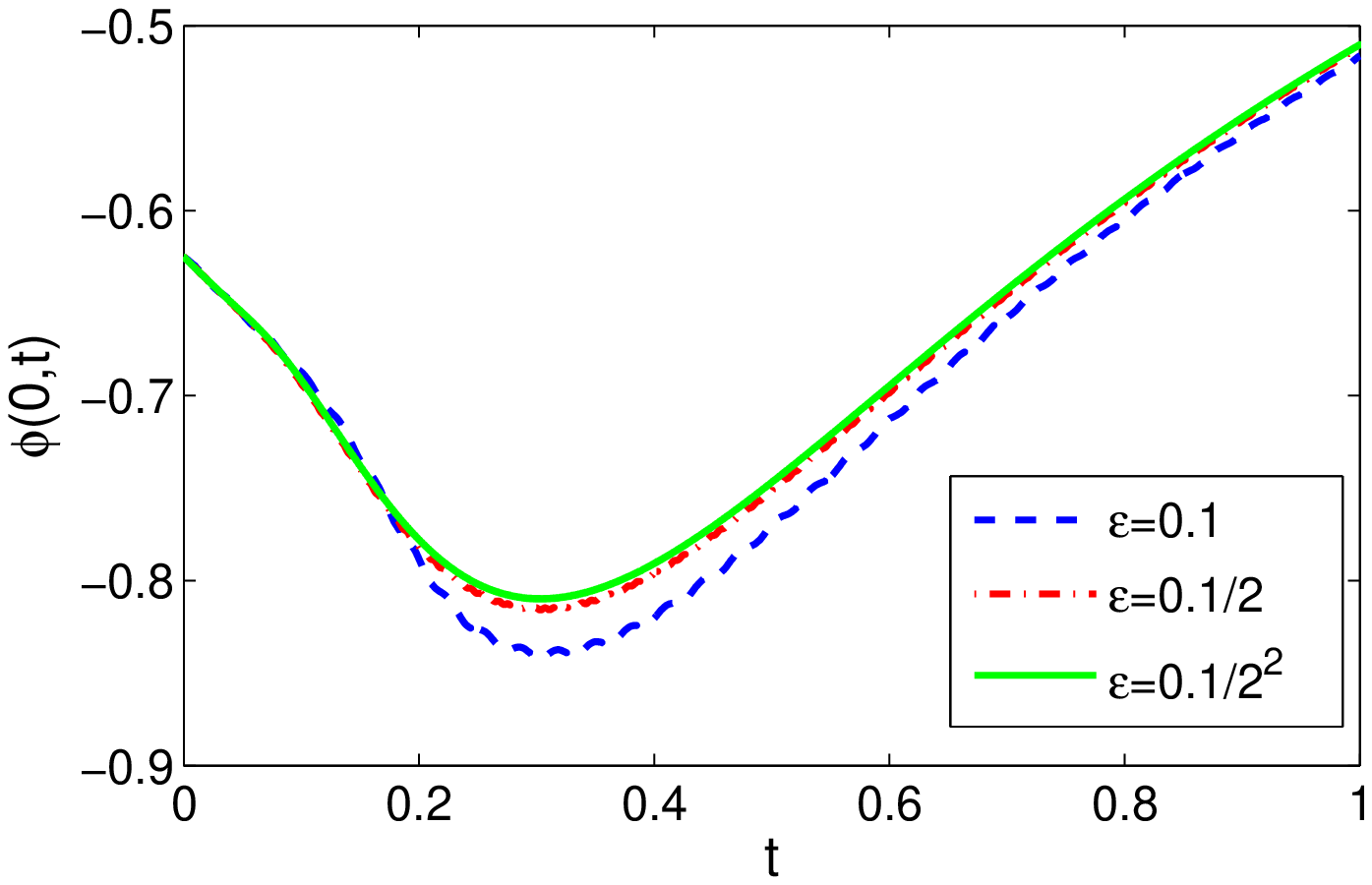}
\end{minipage}
\begin{minipage}[t]{0.5\linewidth}
\centering
\includegraphics[height=4cm,width=6.75cm]{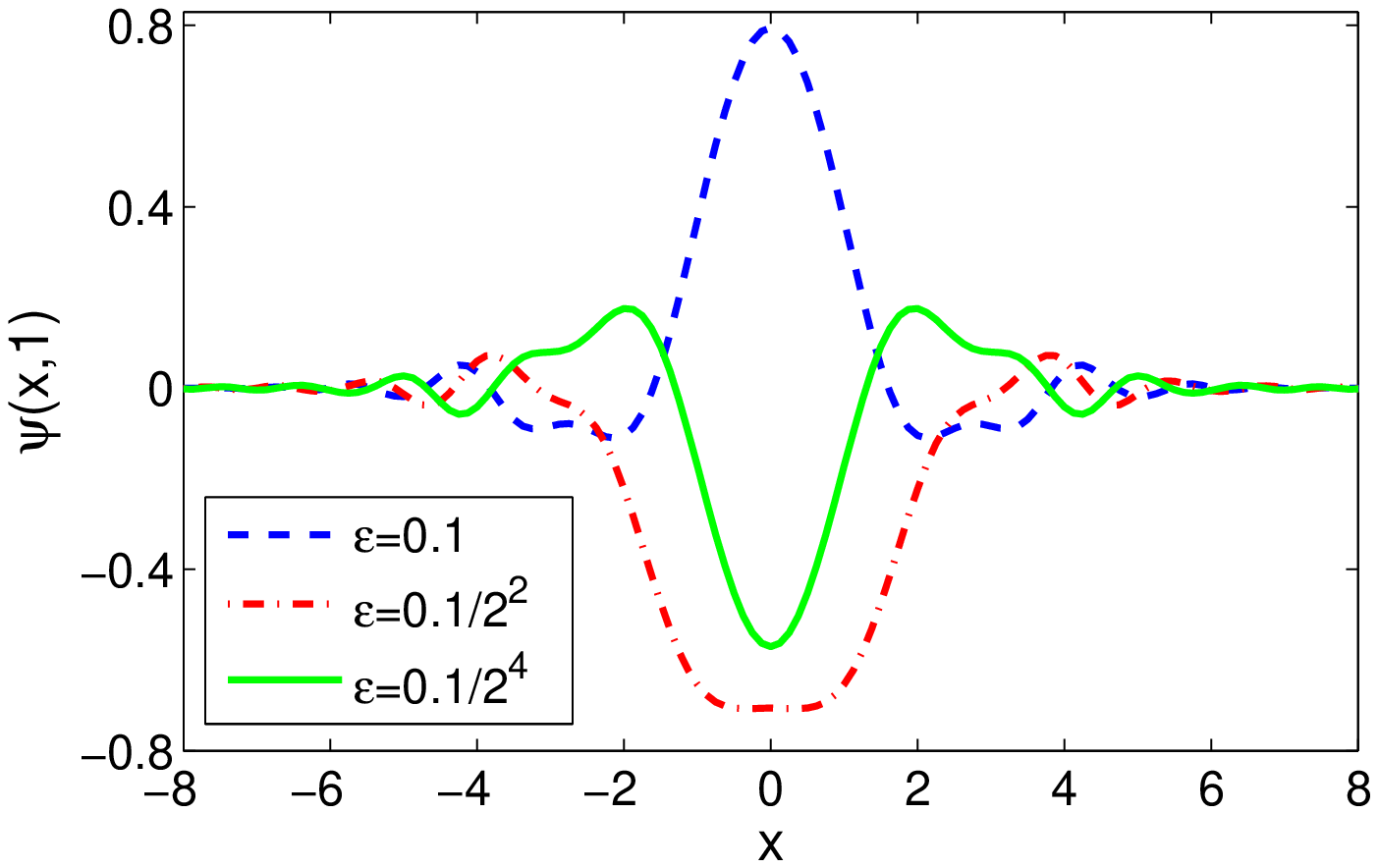}
\end{minipage}%
\hspace{1mm}
\begin{minipage}[t]{0.5\linewidth}
\centering
\includegraphics[height=4cm,width=6.75cm]{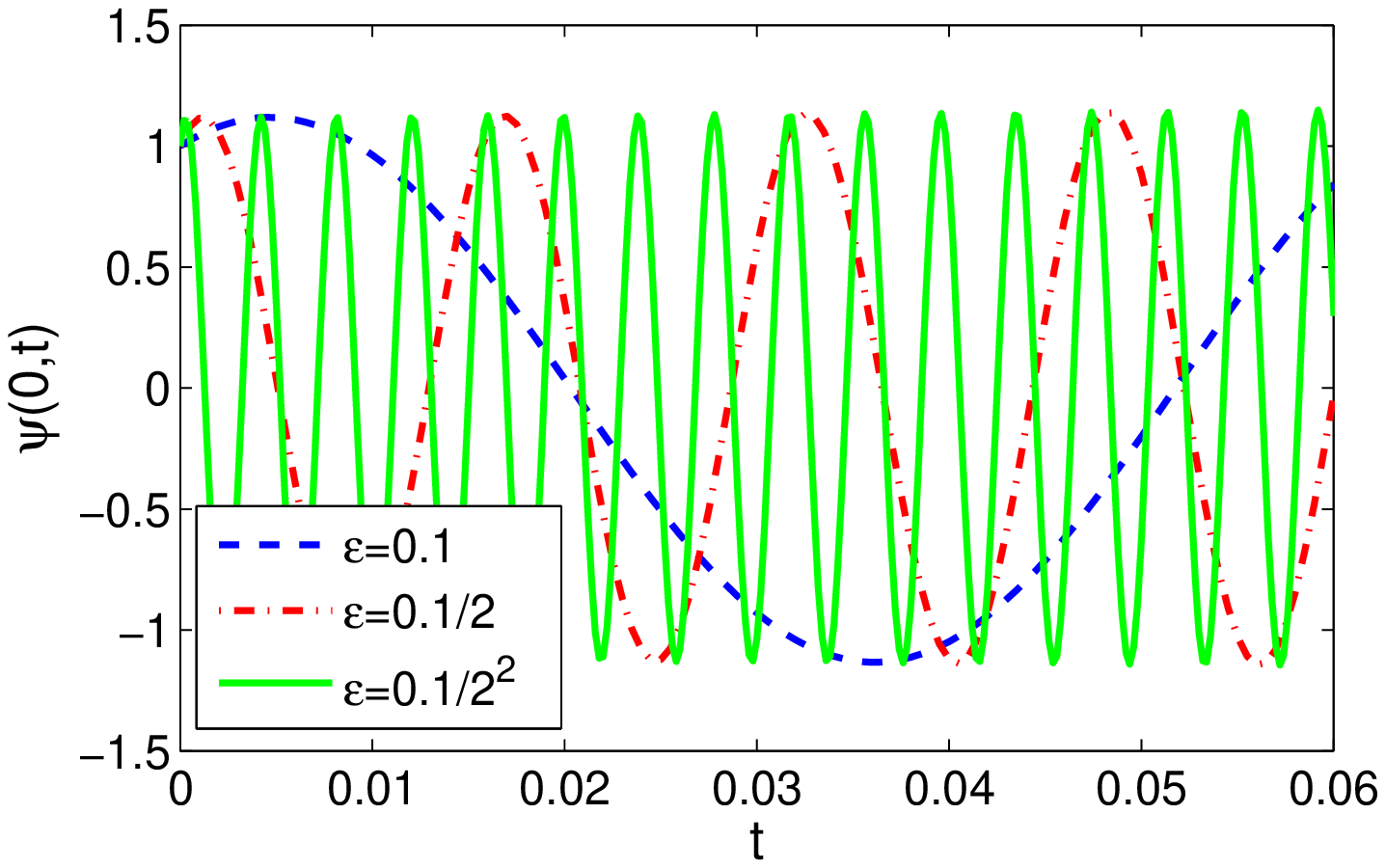}
\end{minipage}
\centerline{Incompatible case:}
\begin{minipage}[t]{0.5\linewidth}
\centering
\includegraphics[height=4cm,width=6.75cm]{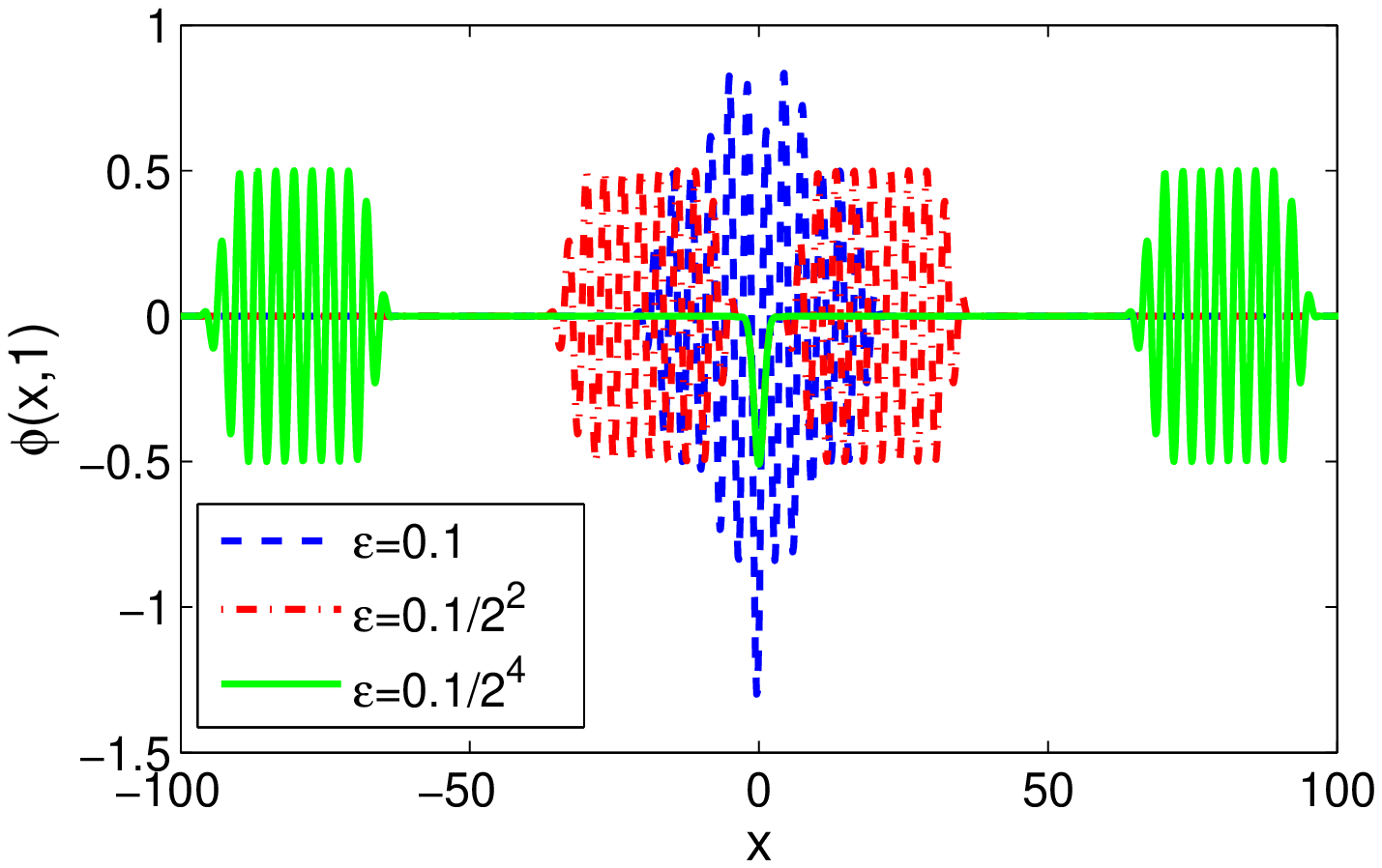}
\end{minipage}%
\hspace{1mm}
\begin{minipage}[t]{0.5\linewidth}
\centering
\includegraphics[height=4cm,width=6.9cm]{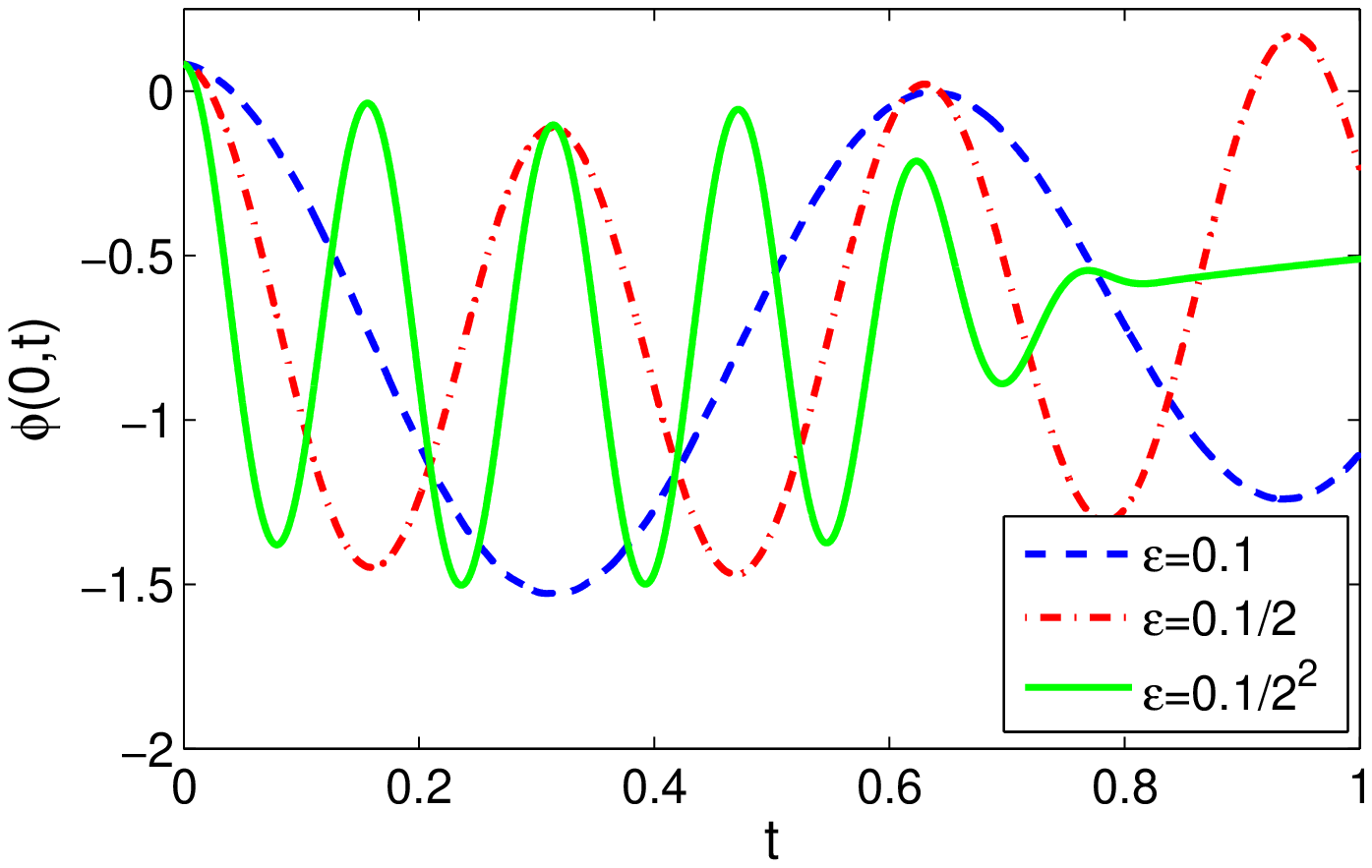}
\end{minipage}
\begin{minipage}[t]{0.5\linewidth}
\centering
\includegraphics[height=4cm,width=6.75cm]{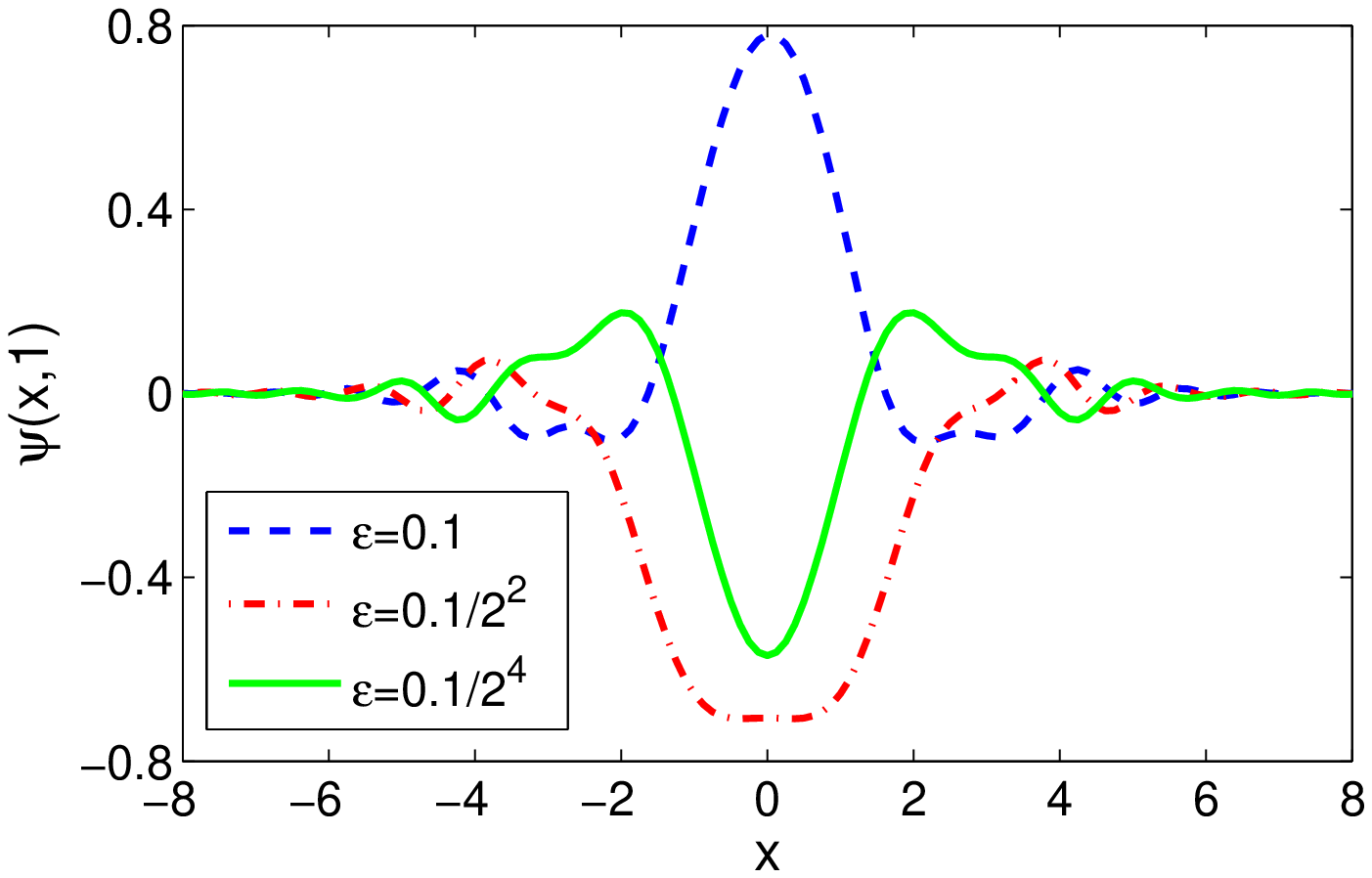}
\end{minipage}%
\hspace{1mm}
\begin{minipage}[t]{0.5\linewidth}
\centering
\includegraphics[height=4cm,width=6.75cm]{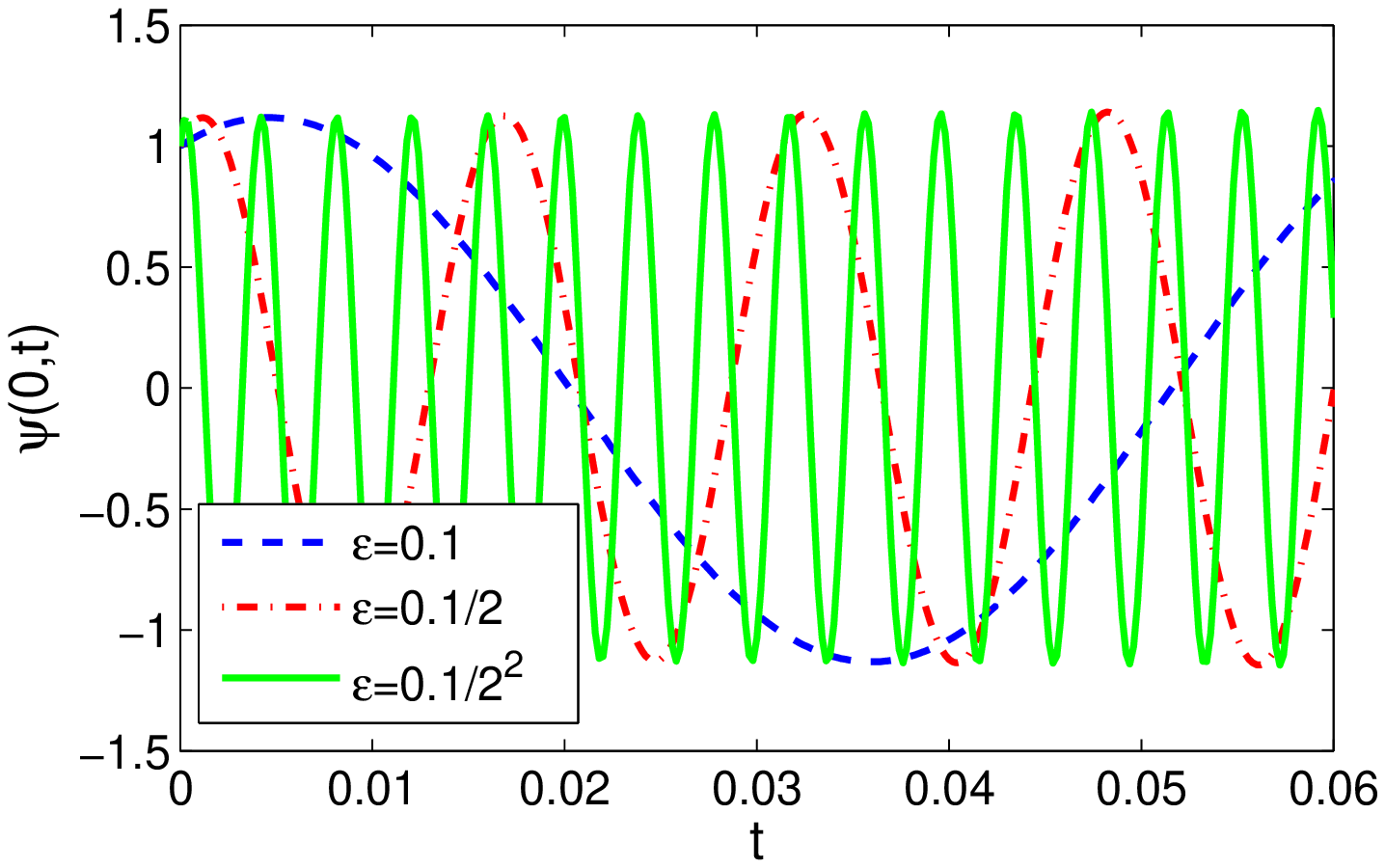}
\end{minipage}
\caption{Solution $\phi(x,1),\,\phi(0,t),\,\psi(x,1),\,\psi(0,t)$ with (\ref{compatiable}) or (\ref{E1}).}\label{fig:incomp}
\end{figure}
It can be seen that when the KGZ system (\ref{KGZ}) starts with initial data that has $O(1)$-incompatibility in the limit regime $\eps <\gm\ll1$, the solution contains both the rapid temporal oscillation and the fast outgoing initial layers of $O(1)$-amplitude.
This complex highly oscillatory behaviour mixes difficulties from high-plasma-frequency limit and subsonic limit, and hence makes the numerical approximation of (\ref{KGZ}) extremely challenging in the regime $\eps<\gm\ll 1$. As has been investigated in \cite{KGZ}, the meshing strategy of the exponential integrator method is $\tau=O(\eps^2)$ in time with $\tau$ denoting the time step.

The aim of this work is to propose an efficient numerical scheme which is uniformly and optimally accurate for solving the KGZ system (\ref{KGZ}) for all $0<\eps<\gm\leq1$ under general (incompatible) initial data. To this purpose, a multiscale decomposition of (\ref{KGZ}) will be derived firstly. For component $\psi$, we adopt the modulated Fourier expansion \cite{Cohen,ICM,APKG,Lubichbook} to explicitly express the oscillations from the high-plasma-frequency limit regime. For the component $\phi$, we use an asymptotic consistent formulation motivated by that of the Zakharov system \cite{ZE2} in the subsonic limit regime, which extracts the initial layer. Based on the decomposed formulation, we propose a multiscale time integrator (MTI) via the time-splitting technique and exponential wave integrators accomplished by Fourier spectral/pseudospectral discretization in space. The proposed MTI scheme is explicit and uniformly accurate with first order convergence rate in time and spectral convergence rate in space for all parameters in the regime $0<\eps<\gm \leq 1$. Extensive numerical evidences are provided to illustrate the accuracy and efficiency of the scheme. Finally, we apply the scheme to study the convergence rates of (\ref{KGZ}) to its limit models when $\eps<\gm \to0^+$.

The rest of the paper is organized as follows.
In Section \ref{sec:2}, we present the multiscale decomposition for the KGZ system.
The uniformly accurate method is derived in Section \ref{sec:3} and numerical results are reported in  Section \ref{sec:4}. Some concluding remarks are drawn in Section \ref{sec:5}.
Throughout the paper, we adopt the standard Sobolev spaces as well as the
corresponding norms \cite{Adams} and denote $A\lesssim B$ to represent that there exists a generic
constant $C>0$ independent of $\eps$, $\gm$, $\tau$ and $h$ such that $|A|\le CB$.

\section{A multiscale decomposition}\label{sec:2}
In this section, we present a multiscale decomposition for the KGZ system (\ref{KGZ}) which is consistent with the limit model (\ref{Sch}) in simultaneous limit regime.

To handle the first equation in the KGZ system, we apply the \emph{modulated Fourier expansion} of $\psi$ in the high-plasma-frequency limit \cite{Cohen,ICM,APKG}:
\begin{equation}\label{mfe}
\psi(\bx,t)=\fe^{it/\eps^2}z(\bx,t)+\fe^{-it/\eps^2}\overline{z}(\bx,t)+r(\bx,t),\quad t\geq0,
\end{equation}
where $z$ is the slow-varying part in terms of $t/\eps^2$ and $r$ denotes the remainder.
Plugging it into (\ref{psi}), we get an equivalent equation as follows:
\begin{align*}
  &\fe^{it/\eps^2}\left[2i\partial_t-\Delta +\phi\right]z+\fe^{-it/\eps^2}\left[-2i\partial_t-\Delta+\phi\right]\overline{z}\\
  &+\eps^2\partial_{tt}r-\Delta r+\frac{r}{\eps^2}+\phi r+\eps^2\fe^{it/\eps^2}\partial_{tt}z+\eps^2\fe^{-it/\eps^2}\partial_{tt}\overline{z}=0.
\end{align*}
Decomposing it into a coupled system for the two unknowns $z$ and $r$, we get
\begin{subequations}
\begin{align}
&2i\partial_tz-\Delta z+\phi z=0,\label{zeq}\\
&\eps^2\partial_{tt}r-\Delta r+\frac{r}{\eps^2}+\phi r+\eps^2\fe^{it/\eps^2}\partial_{tt}z+\eps^2\fe^{-it/\eps^2}\partial_{tt}\overline{z}=0.\label{req}
\end{align}
\end{subequations}
Next, we describe how to set proper initial data for $z$ and $r$.
Based on the expansion and the given initial data, we have
\begin{align*}
&z(\bx,0)+\overline{z}(\bx,0) +r(\bx,0)=\psi_0(\bx),\\
&\frac{i}{\eps^2}\left[z(\bx,0)-\overline{z}(\bx,0)\right]
+\partial_tz(\bx,0)+\partial_t\overline{z}(\bx,0)
+\partial_tr(\bx,0)=\frac{\psi_1(\bx)}{\eps^2}.
\end{align*}
To make it consistent with the limit Schr\"odinger equation (\ref{Sch}), we set the initial data $z(\bx,0)$ the same as that of the limit equation \eqref{Sch}, i.e.,
\begin{equation}z(\bx,0)=\frac{1}{2}\left(\psi_0(\bx)-i\psi_1(\bx)\right)=:z_0(\bx),\end{equation}
which immediately implies that
\[
r(\bx,0)=0,\quad \partial_tr(\bx,0)=-\partial_tz(\bx,0)-\partial_t\overline{z}(\bx,0),
\]
where $\p_t z(\bx,0)$ is given by (\ref{zeq}):
$\p_t z(\bx,0)=-\fl{i}{2}\Delta z_0(\bx)+\fl{i}{2}\phi_0(\bx)z_0(\bx).$

For the density deviation $\phi$, inspired by (\ref{app}) and the asymptotic consistent formulation of the solution of the Zakharov system in the subsonic limit regime \cite{ZE1, ZE2}, we introduce an expansion on $\phi$ as
\be\label{dec-phi}
\phi(\bx,t)=-2|z(\bx,t)|^2+I(\bx,t)+q(\bx,t),\quad t\geq0,
\ee
where
$I(\bx,t)$ represents the fast-outing initial layer caused by the initial incompatibility of the KGZ system, and it is defined by the free wave equation
\be\label{layer}
\left\{\begin{split}
&\gamma^2\partial_{tt} I-\Delta I=0,\\
&I(\bx,0)=\phi_0(\bx)+2|z_0(\bx)|^2,\\
&\partial_t I(\bx,0)=\frac{\phi_1(\bx)}{\gamma}+2\partial_t|z|^2(\bx,0)
=\frac{\phi_1(\bx)}{\gamma}+2\mathrm{Im}\left(\overline{z_0(\bx)}\Delta z_0(\bx)\right),
\end{split}\right.
\ee
where $\mathrm{Im}(f)$ represents the imaginary part of $f$.
Compared to the approximation (\ref{app}), we consider a more detailed decomposition which also involves the second initial layer caused by the initial incompatibility of the time derivative \cite{Ozawa}.
Plugging (\ref{dec-phi}) into (\ref{phi}), we can get the following equation on $q(\bx,t)$:
\[\gamma^2\partial_{tt}q-\Delta q=\Delta (r^2)+2\gamma^2\partial_{tt}|z|^2
+2\mathrm{Re}\left[\fe^{2it/\eps^2}\Delta z^2+2\fe^{it/\eps^2}\Delta (zr)\right].\]

To summarize, by adopting the decomposition (\ref{mfe}) and (\ref{dec-phi}), we equivalently rewrite  the KGZ system (\ref{KGZ}) into the following equations involving the unknowns $z$, $r$ and $q$:
\begin{subequations}\label{KGZ decomp}
\begin{align}
&2i\partial_tz-\Delta z+(-2|z|^2+q+I)z=0,\quad \bx\in\bR^d,\ t>0,\label{eq:z}\\
&\eps^2\partial_{tt}r-\Delta r+\frac{r}{\eps^2}+(-2|z|^2+q+I)r
+\eps^2\fe^{it/\eps^2}\partial_{tt}z+\eps^2\fe^{-it/\eps^2}\partial_{tt}\overline{z}=0,\label{eq:r}\\
&\gamma^2\partial_{tt}q-\Delta q=\Delta (r^2)+2\gamma^2\partial_{tt}|z|^2
+2\mathrm{Re}\left[\fe^{2it/\eps^2}\Delta z^2+2\fe^{it/\eps^2}\Delta (zr)\right],\label{eq:q}
\end{align}
\end{subequations}
with initial data
\be\label{zri}
\begin{aligned}
&z(\bx,0)=\frac{1}{2}[\psi_0(\bx)-i\psi_1(\bx)],\quad q(\bx,0)=0, \quad \partial_t q(\bx,0)=0,\\
&r(\bx,0)=0,\quad \partial_t r(\bx,0)=-\partial_t z(\bx,0)-\partial_t\overline{z}(\bx,0).
\end{aligned}
\ee
Note that the initial layer $I$ is the free wave defined by \eqref{layer}, which can be written explicitly or solved separately and efficiently.

\bigskip

\begin{remark}
We remark that we didn't adopt the multiscale decomposition by frequency from \cite{KG-zhao,KGZ1} for $\psi$, because it would result in a Schr\"odinger equation with a wave operator and a highly oscillatory potential, which is difficult to integrate in a uniformly accurate manner.
\end{remark}

\subsection{Formal estimates}
We give a prior estimate of the decomposition. Firstly, inspired by the oscillation properties of the solution (cf. Fig. \ref{fig:incomp}) and the theoretical results in \cite{KGZ-limit, Masmoudi2008}, we assume that the solution of the KGZ system \eqref{KGZ} and the initial data satisfy:
\begin{equation}
\label{A}\begin{split}
&\|\psi_0\|_{H^{m+6}}+\|\psi_1\|_{H^{m+6}}+\|\phi\|_{L^\infty([0,T]; H^{m+6})}+\gm\|\p_t\phi\|_{L^\infty([0,T]; H^{m+2})}\\
&\quad+\gm^2\|\p_{tt}\phi\|_{L^\infty([0,T]; H^{m})}
\lesssim 1,
\end{split}
\end{equation}
where $0<T<T_{\mathrm{max}}$ with $T_{\mathrm{max}}$ being the maximal common existing time and $m$ is an integer satisfying $m>d/2$ such that the bilinear inequality holds \cite{Adams}
\[\|fg\|_{H^m}\le C_{m,d}\|f\|_{H^m}\|g\|_{H^m}.\]

\begin{proposition}\label{prop}(A prior estimate) Under the assumption (\ref{A}), we have
\begin{align*}
&\|z(t)\|_{H^{m+6}}+\|\p_{t}z(t)\|_{H^{m+4}}+\gm
\|\p_{tt}z(t)\|_{H^{m+2}}+\gm^2\|\p_t^3z(t)\|_{H^{m}}\lesssim 1,\quad t\in [0,T];\\
&\qquad\qquad\|r\|_{L^\infty([0,T];H^m)}\lesssim \eps^2,\quad \|\p_t r\|_{L^\infty([0,T];H^m)}\lesssim 1;\\
&\qquad\qquad\|q\|_{L^\infty([0,T];H^{m-1})}\lesssim \gm,\quad \|\p_t q\|_{L^\infty([0,T];H^{m-2})}\lesssim 1.
\end{align*}
\end{proposition}
\begin{proof}
We omit the space variable for simplicity of notation. It follows from \eqref{eq:z} and Duhamel's formula that
\[z(t)=\fe^{-\fl{it}{2}\Delta}z(0)+\fl{i}{2}\int_0^t \fe^{-\fl{i}{2}(t-s)\Delta}\left[\phi(s)z(s)\right]ds.\]
Noticing $\fe^{is\Delta}$ preserves $H^k$-norm, by applying the Minkovski's inequality, the bilinear inequality, we get
\begin{align*}
\|z(t)\|_{H^{m+6}}&\le \|z(0)\|_{H^{m+6}}+\fl{1}{2}\int_0^t\|\phi(s)z(s)\|_{H^{m+6}}ds\\
&\le \|z(0)\|_{H^{m+6}}+\fl{C_{m,d}}{2}\|\phi\|_{L^\infty([0,T];H^{m+6})}\int_0^t\|z(s)\|_{H^{m+6}}ds.
\end{align*}
Applying the Gronwall's inequality, we obtain
\[\|z\|_{L^\infty([0,T];H^{m+6})}\le \|z(0)\|_{H^{m+6}}\fe^{T C_{m,d}\|\phi\|_{L^\infty([0,T];H^{m+6})}}\lesssim 1,\]
which concludes the boundedness of $z$ by noticing the definition of $z(0)$ (cf. \eqref{zri}) and the assumption \eqref{A}. For $\p_t z$, it follows from \eqref{eq:z} that
\[\|\p_t z(t)\|_{H^{m+4}}\le \fl{1}{2}\|z( t)\|_{H^{m+6}}+\fl{C_{m,d}}{2}\|\phi(t)\|_{H^{m+4}}\|z(t)\|_{H^{m+4}}\lesssim 1,\]
which directly gives the result.
Similarly, we have
\begin{align*}
&\|\p_{tt}z(t)\|_{H^{m+2}}\le \fl{1}{2}\|\p_tz(t)\|_{H^{m+4}}+\fl{C_{m,d}}{2}\left[\|z(t)\|_{H^{m+2}}\|\p_t\phi(t)\|_{H^{m+2}}
\right.\\
&\qquad\qquad\qquad\quad\left.+\|\p_t z(t)\|_{H^{m+2}}\|\phi(t)\|_{H^{m+2}}\right]\lesssim \fl{1}{\gm},\\
&\|\p_{t}^3z(t)\|_{H^{m}}\le \fl{1}{2}\|\p_{tt}z(t)\|_{H^{m+2}}+\fl{C_{m,d}}{2}\left[\|z(t)\|_{H^{m}}
\|\p_{tt}\phi(t)\|_{H^{m}}+\|\p_{tt} z(t)\|_{H^{m}}\|\phi(t)\|_{H^{m}}
\right.\\
&\qquad\qquad\qquad\left.+2\|\p_t z(t)\|_{H^{m}}
\|\p_{t}\phi(t)\|_{H^{m}}\right]\lesssim \fl{1}{\gm^2},
\end{align*}
by noting the assumption (\ref{A}).

Next, we estimate $r$. Duhamel's formula gives
\be\label{rsol}
\begin{split}
r(t)&=\cos(t\lag \nabla \rag_\eps)r(0)+\fl{\sin(t\lag \nabla\rag_\eps)}{\lag\nabla\rag_\eps}\p_t r(0)\\
&\quad-\int_0^t \fl{\sin((t-s)\lag \nabla \rag_\eps)}{\lag \nabla \rag_\eps}\left[\fl{\phi(s)r(s)}{\eps^2}+\fe^{is/\eps^2}\p_{tt}z(s)+\fe^{-is/\eps^2}
\overline{\p_{tt}z(s)}\right]ds,
\end{split}
\ee
where $\lag \nabla\rag_\eps=\fl{1}{\eps^2}\sqrt{1-\eps^2\Delta}$. Noticing that $r(0)=0$,
(cf. \eqref{zri}), we have
\[
r(t)=\fl{\sin(t\lag \nabla\rag_\eps)}{\lag\nabla\rag_\eps}\p_t r(0)-\int_0^t \fl{\sin((t-s)\lag \nabla \rag_\eps)}{\lag \nabla \rag_\eps}\left(\fl{\phi(s)r(s)}{\eps^2}\right)ds+r_1(t)+r_2(t),\]
where
\begin{align*}
r_1(t)&=-\fl{\mathrm{Im}}{\lag\nabla\rag_\eps}\left[\fe^{it\lag\nabla\rag_\eps}
\int_0^t \fe^{is(1/\eps^2-\lag\nabla\rag_\eps)}\p_{tt}z(s)ds\right],\\
r_2(t)&=\fl{\mathrm{Im}}{\lag\nabla\rag_\eps}\left[\fe^{-it\lag\nabla\rag_\eps}
\int_0^t \fe^{is(1/\eps^2+\lag\nabla\rag_\eps)}\p_{tt}z(s)ds\right].
\end{align*}
Integrating the integrals in $r_1(t)$ and $r_2(t)$ by parts in different ways, we get
\begin{align*}
r_1(t)&=-\fl{\mathrm{Im}}{\lag\nabla\rag_\eps}\left[\fe^{it/\eps^2}\p_t z(t)-\fe^{it\lag\nabla\rag_\eps}\p_t z(0)\right]\\
&\quad-\fl{\lag\nabla\rag_\eps-\fl{1}{\eps^2}}{\lag\nabla\rag_\eps}\mathrm{Re}\left[
\fe^{it\lag\nabla\rag_\eps}
\int_0^t \fe^{is(1/\eps^2-\lag\nabla\rag_\eps)}\p_{t}z(s)ds\right],\\
r_2(t)&=\fl{-\eps^2\mathrm{Re}}{\lag \nabla\rag_\eps(1+\eps^2\lag \nabla\rag_\eps)}\left[\fe^{it/\eps^2}\p_{tt}z(t)-\fe^{-it\lag \nabla \rag_\eps}\p_{tt}z(0)\right.\\
&\qquad\qquad\qquad\qquad\quad\,\,\,\left.-\fe^{-it\lag \nabla\rag_\eps} \int_0^t \fe^{is(1/\eps^2+\lag\nabla\rag_\eps)}\p_t^3 z(s)ds\right].
\end{align*}
Noticing for any $s\in\bR$, $k\ge 0$,
\[\Big\|\fl{u}{\lag\nabla\rag_\eps}\Big\|_{H^k}\le \eps^2\|u\|_{H^k},\quad
\|\sin(s\lag \nabla\rag_\eps)u\|_{H^k}\le \|u\|_{H^k},\quad \|\cos(s\lag \nabla\rag_\eps)u\|_{H^k}\le \|u\|_{H^k},\]
and
\[\|\fe^{is\lag\nabla\rag_\eps}u\|_{H^k}=\|u\|_{H^k},\quad\lag\nabla\rag_\eps-\fl{1}{\eps^2}=\fl{-\Delta}{1+\sqrt{1-\eps^2\Delta}},
\quad \Big\|\big(\lag\nabla\rag_\eps-\fl{1}{\eps^2}\big)u\Big\|_{H^k}\le \|u\|_{H^{k+2}},\]
which immediately yields that
\begin{align*}
\|r_1(t)\|_{H^m}&\le 2\eps^2\|\p_t z\|_{L^\infty([0,T];H^m)}+\eps^2 T
\|\p_t z\|_{L^\infty([0,T];H^{m+2})}\lesssim \eps^2,\\
\|r_2(t)\|_{H^m}&\le \eps^4\left[2\|\p_{tt}z\|_{L^\infty([0,T];H^m)}+
T\|\p_t^3z\|_{L^\infty([0,T];H^m)}\right]\lesssim \fl{\eps^4}{\gm^2}\lesssim\eps^2.
\end{align*}
We derive that
\begin{align*}
\|r(t)\|_{H^m}&\le \eps^2\|\p_t r(0)\|_{H^m}+\|r_1(t)\|_{H^m}+
\|r_2(t)\|_{H^m}\\
&\quad+C_{m,d}\|\phi\|_{L^\infty([0,T];H^m)}\int_0^t \|r(s)\|_{H^m}ds,
\end{align*}
which implies that
\begin{align*}
\|r\|_{L^\infty([0,T];H^m)}&\le \fe^{TC_{m,d}\|\phi\|_{L^\infty([0,T];H^m)}}\left[\eps^2\|\p_t r(0)
\|_{H^m}+\|r_1(t)\|_{L^\infty([0,T];H^m)}\right.\\
&\qquad \qquad\qquad\qquad\qquad\,\,\left.+\|r_2(t)\|_{L^\infty([0,T];H^m)}\right]\lesssim \eps^2.
\end{align*}
Differentiating \eqref{rsol} with respect to $t$, we get
\begin{align*}
\p_t r(t)&=\cos(t\lag \nabla\rag_\eps)\p_t r(0)-\fl{1}{\eps^2}
\int_0^t \cos((t-s)\lag \nabla \rag_\eps)\left(\phi(s)r(s)\right)ds+\p_t r_1(t)+\p_t r_2(t),
\end{align*}
with
\begin{align*}
\p_t r_1(t)&=\mathrm{Re}\left(\fe^{it\lag\nabla\rag_\eps}\p_t z(0)
-\fe^{it/\eps^2}\p_t z(t)\right)-
\fl{1}{\lag \nabla\rag_\eps}\mathrm{Im}\left(\fe^{it/\eps^2}\p_{tt} z(t)\right)\\
&\quad+\big(\lag\nabla\rag_\eps-\fl{1}{\eps^2}\big)\mathrm{Im}\Big(\fe^{it\lag\nabla\rag_\eps}
\int_0^t \fe^{is(1/\eps^2-\lag\nabla\rag_\eps)}\p_{t}z(s)ds\Big),\\
\p_t r_2(t)&=\fl{\mathrm{Im}\left[\fe^{it/\eps^2}\p_{tt}z(t)+\eps^2\lag \nabla \rag_\eps
\fe^{-it\lag \nabla \rag_\eps}\big(\p_{tt}z(0)+\int_0^t \fe^{is(1/\eps^2+\lag\nabla\rag_\eps)}\p_t^3 z(s)ds\big)\right]}{\lag \nabla\rag_\eps(1+\eps^2\lag \nabla\rag_\eps)}.
\end{align*}
Thus
\begin{align*}
\|\p_t r(t)\|_{H^m}&\le \|\p_t r(0)\|_{H^m}+\fl{C_{m,d}}{\eps^2}\int_0^t
\|\phi(s)\|_{H^m}\|r(s)\|_{H^m}ds+2\|\p_t z\|_{L^\infty([0,T];H^m)}\\
&+3\eps^2
\|\p_{tt}z\|_{L^\infty([0,T];H^m)}+T\|\p_t z\|_{L^\infty([0,T];H^{m+2})}+T\eps^2
\|\p_t^3z\|_{L^\infty([0,T];H^m)}\\
&\lesssim 1+\fl{\eps^2}{\gm^2}\lesssim 1,
\end{align*}
which completes the proof for the property of $r(t)$.

For $q(t)$, it follows from \eqref{eq:q} that
\begin{align}
q(t)&=2\gm\int_0^t \fl{\sin(\fl{t-s}{\gm}|\nabla|)}{|\nabla|}[\p_{tt}|z|^2(s)]ds\nn\\
&\quad-\fl{|\nabla|}{\gm}\int_0^t \sin\left(\fl{t-s}{\gm}|\nabla|\right)\left[
r^2(s)+2\mathrm{Re}\left(\fe^{2is/\eps^2} z^2(s)+2\fe^{is/\eps^2}z(s)r(s)
\right)\right]ds,\label{qsol}
\end{align}
where $|\nabla|=\sqrt{-\Delta}$. From \eqref{eq:z}, we get
\[\p_t |z|^2=\mathrm{Im}(\overline{z}\Delta z),\quad \p_{tt}|z|^2=\mathrm{Im}(\overline{\p_t z}\Delta z+ \overline{z}\Delta \p_t z),\]
which implies that
\be\label{absz}
\left\|\p_{tt}|z|^2\right\|_{H^m}\lesssim \|\p_t z\|_{H^m}\|z\|_{H^{m+2}}+\|z\|_{H^m}\|\p_t z\|_{H^{m+2}}\lesssim 1.
\ee
Thus
\begin{align}
\|q(t)\|_{H^{m-1}}&\lesssim \gm \|\p_{tt}|z|^2\|_{H^{m-1}}+\|q_1(t)\|_{H^{m-1}}\nn\\
&\quad+\fl{1}{\gm}\left[\|r\|_{L^\infty([0,T];H^{m})}^2+
\|r\|_{L^\infty([0,T];H^{m})}\|z\|_{L^\infty([0,T];H^{m})}\right]\nn\\
&\lesssim \gm+\|q_1(t)\|_{H^{m-1}},\label{q1e}
\end{align}
where
\begin{align*}
q_1(t)&=-\fl{2|\nabla|}{\gm}\mathrm{Re}\left[\int_0^t\sin(\fl{t-s}{\gm}|\nabla|)\left(\fe^{2is/\eps^2} z^2(s)\right)ds\right]=q_2(t)+q_3(t),
\end{align*}
with
\begin{align*}
q_2(t)&=\fl{|\nabla|}{\gm}\mathrm{Im}\left[\fe^{-it|\nabla|/\gm}\int_0^t\fe^{is(2/\eps^2+|\nabla|/\gm)}z^2(s)ds\right],\\
q_3(t)&=-\fl{|\nabla|}{\gm}\mathrm{Im}\left[\fe^{it|\nabla|/\gm}\int_0^t\fe^{is(2/\eps^2-|\nabla|/\gm)}z^2(s)ds\right].
\end{align*}
Integrating $q_2(t)$ by parts, we get
\[q_2(t)=\fl{-\fl{\eps^2}{\gm}|\nabla|}{2+\fl{\eps^2}{\gm}|\nabla|}\mathrm{Re}\Big[\fe^{2it/\eps^2}z^2(t)-
\fe^{-it|\nabla|/\gm}
\big(z^2(0)+2\int_0^t \fe^{is(2/\eps^2+|\nabla|/\gm)}z(s)\p_t z(s)ds\big)\Big],\]
which implies that
\be\label{q2e}
\|q_2(t)\|_{H^{m-1}}\lesssim \fl{\eps^2}{\gm}\left[\|z\|_{L^\infty([0,T];H^m)}^2+
\|z\|_{L^\infty([0,T];H^m)}\|\p_t z\|_{L^\infty([0,T];H^m)}\right]\lesssim \eps.
\ee
For $q_3(t)$, we need to make a more careful investigation since it could involve a resonance.
Taking Fourier transform of $q_3$, we obtain
\begin{align*}
\wh{q_3(t)}(\xi)&=\fl{i|\xi|}{2\gm}\left[\fe^{it|\xi|/\gm}\int_0^t\fe^{is(2/\eps^2-
|\xi|/\gm)}\wh{z^2(s)}(\xi)ds\right.\\
&\qquad\qquad\left.-\fe^{-it|\xi|/\gm}\int_0^t\fe^{is(
|\xi|/\gm-2/\eps^2)}\wh{\overline{z^2(s)}}(\xi)ds\right].
\end{align*}
For $|\xi|\le \gm/\eps^2$, integrating by parts, we get
\begin{align*}
\wh{q_3(t)}(\xi)&=\fl{\fl{\eps^2}{\gm}|\xi|}{4-\fl{2\eps^2}{\gm}|\xi|}\left[
\fe^{\fl{2it}{\eps^2}}\wh{z^2(t)}(\xi)-\fe^{\fl{it|\xi|}{\gm}}\Big(\wh{z^2(0)}(\xi)+2\int_0^t
\fe^{is(\fl{2}{\eps^2}-\fl{|\xi|}{\gm})}\wh{z(s)\p_t z(s)}(\xi)ds\Big)\right.\\
&\quad\left.+\fe^{-2it/\eps^2}\wh{\overline{z^2(t)}}(\xi)-\fe^{-\fl{it|\xi|}{\gm}}
\Big(\wh{\overline{z^2(0)}}(\xi)+2\int_0^t
\fe^{is(|\xi|/\gm-2/\eps^2)}\wh{\overline{z(s)\p_t z(s)}}(\xi)ds\Big)\right],
\end{align*}
which implies that
\begin{align*}
\left|\wh{q_3(t)}(\xi)\right|&\le \fl{\eps^2}{2\gm}|\xi|\left[\left|\wh{z^2(t)}(\xi)\right|+\left|\wh{z^2(0)}(\xi)\right|+
\left|\wh{\overline{z^2(t)}}(\xi)\right|
+\left|\wh{\overline{z^2(0)}}(\xi)\right|\right.\\
&\qquad\qquad\,\,\left.+2\int_0^t \left(\left|\wh{z(s)\p_t z(s)}(\xi)\right|+\left|\wh{\overline{z(s)\p_t z(s)}}(\xi)\right|\right)ds\right],\quad |\xi|\leq \gm/\eps^2.
\end{align*}
For $|\xi|>\gm/\eps^2$, noticing that $\wh{\fl{\p^kf}{\p x_j^k}}(\xi)=(i\xi_j)^k\wh{f}(\xi)$ for $k\in \mathbb{N}$, which implies that
\[\wh{f}(\xi)=-\fl{1}{|\xi|^2}\wh{(\Delta f)}(\xi)=\fl{1}{|\xi|^4}\wh{(\Delta^2 f)}(\xi).\]
Hence for $|\xi|>\gm/\eps^2$, we have
\begin{align*}
\left|\wh{q_3(t)}(\xi)\right|&\le \fl{1}{2\gm |\xi|^3}\left[\int_0^t\left|\wh{\Delta^2 (z^2(s))}(\xi)\right|ds
+\int_0^t\left|\wh{\Delta^2\overline{z^2(s)}}(\xi)\right|ds\right]\\
&\lesssim \fl{\eps^6}{\gm^4}\left[\int_0^t\left|\wh{\Delta^2 (z^2(s))}(\xi)\right|ds
+\int_0^t\left|\wh{\Delta^2\overline{z^2(s)}}(\xi)\right|ds\right].
\end{align*}
Combining the estimates above, we get
\begin{align*}
\|q_3(t)\|_{H^{m-1}}&\lesssim \left\|(1+|\xi|)^{m-1}\wh{q_3(t)}(\xi)\right\|_{L^2}\\
&\lesssim \fl{\eps^2}{\gm}\left[\|z^2(t)\|_{H^{m}}+\|z^2(0)\|_{H^{m}}+
\|z\|_{L^\infty([0,T];H^{m})}\|\p_tz\|_{L^\infty([0,T];H^{m})}\right]\\
&\quad +\fl{\eps^6}{\gm^4}\|z^2\|_{L^\infty([0,T];H^{m+3})}\lesssim \eps,
\end{align*}
which together with \eqref{q1e} and \eqref{q2e} concludes the estimate.

Finally, we give the estimate for $\p_t q$. Differentiating \eqref{qsol} with respect to $t$ and integrating by parts for the term involving $z^2(s)$, we get
\begin{align*}
\p_t q(t)&=2\int_0^t\cos\left(\fl{t-s}{\gm}|\nabla|\right)\left[\p_{tt}|z|^2(s)\right]ds\\
&\quad+\fl{\Delta}{\gm^2}\int_0^t \cos\left(\fl{t-s}{\gm}|\nabla|\right)\left[
r^2(s)+2\mathrm{Re}\left(\fe^{2is/\eps^2} z^2(s)+2\fe^{is/\eps^2}z(s)r(s)
\right)\right]ds\label{q'sol}\\
&\hspace{-1mm}=\int_0^t\cos(\fl{t-s}{\gm}|\nabla|)\left[2\p_{tt}|z|^2(s)+\fl{\Delta}{\gm^2}\left(
r^2(s)+
4\mathrm{Re}\left[\fe^{is/\eps^2}z(s)r(s)\right]\right)\right]ds+q_4(t)\\
&\quad+\fl{\fl{\eps^2}{\gm^2}\Delta}{2+\fl{\eps^2}{\gm}|\nabla|}\mathrm{Im}\Big[\fe^{2it/\eps^2}z^2(t)-
\fe^{-\fl{it|\nabla|}{\gm}}
\big(z^2(0)+2\int_0^t \fe^{is(\fl{2}{\eps^2}+\fl{|\nabla|}{\gm})}z(s)\p_t z(s)ds\big)\Big],
\end{align*}
where
\[
q_4(t)=\fl{\Delta}{\gm^2}\mathrm{Re}\left[\fe^{it|\nabla|/\gm}
\int_0^t\fe^{is(2/\eps^2-|\nabla|/\gm)}z^2(s)ds\right].\]
Applying similar arguments as above, we get
\begin{align*}
\|q_4(t)\|_{H^{m-2}}&\lesssim \fl{\eps^2}{\gm^2}\|z\|_{L^\infty([0,T];H^{m})}
\left(\|z\|_{L^\infty([0,T];H^{m})}+\|\p_tz\|_{L^\infty([0,T];H^{m})}\right)\\
&\quad+\fl{\eps^4}{\gm^4}\|z\|^2_{L^\infty([0,T];H^{m+2})}\lesssim 1.
\end{align*}
Thus
\begin{align*}
\|\p_t q(t)\|_{H^{m-2}}&\lesssim \|\p_{tt}|z|^2\|_{H^{m-2}}+\|q_4(t)\|_{H^{m-2}} +\fl{1}{\gm^2}\|r\|^2_{L^\infty([0,T]; H^{m})}\\
&\quad+\fl{1}{\gm^2}\|r\|_{L^\infty([0,T]; H^{m})}\|z\|_{L^\infty([0,T]; H^{m})}
+\fl{\eps^2}{\gm^2}\|z\|^2_{L^\infty([0,T];H^{m})}\\
&\quad+\fl{\eps^2}{\gm^2}\|z\|_{L^\infty([0,T];H^{m})}
\|\p_tz\|_{L^\infty([0,T];H^{m})}\lesssim 1,
\end{align*}
which completes the proof.
\end{proof}

\subsection{Limit model}
To end  this section, we discuss about the limit models for the KGZ system \eqref{KGZ} in the simultaneous limit regime.

Alternative to the limit model (\ref{Sch}), we present a semi-limit model by the formal estimate results.
Based on the expansion (\ref{mfe}) and (\ref{dec-phi}) and the estimates $\|r\|_{H^{m}}\lesssim \eps^2$ and $\|q\|_{H^{m-1}}\lesssim \gm$ from Proposition \ref{prop}, we formally see that
\begin{equation}\label{app2}
\psi\to\fe^{it/\eps^2}z_{\rm op}+\fe^{-it/\eps^2}\overline{z_{\rm op}},\quad
\phi\to-2|z_{\rm op}|^2+I,\quad \eps<\gm\rightarrow 0^+,
\end{equation}
where by (\ref{KGZ decomp}) $z_{\rm op}:=z_{\rm op}(\bx,t)$ satisfies the following nonlinear Schr\"{o}dinger equation with highly oscillatory potential \cite{NLSOP}
\be\label{Sch2}
\left\{\begin{split}
&2i\partial_t z_{\rm op}(\bx,t)-\Delta z_{\rm op}(\bx,t)+(-2|z_{\rm op}(\bx,t)|^2+I(\bx,t))z_{\rm op}(\bx,t)=0,\quad  t>0,\\
&z_{\rm op}(\bx,0)=z_0(\bx),\quad \bx\in\bR^d,
\end{split}\right.
\ee
and $I(\bx,t)$ is the potential given by the free wave equation (\ref{layer}).

Since the free wave $I(\bx,t)$ quickly travels to far field when $\gm\to0$, its effect on  $z_{\rm op}$ in \eqref{Sch2} vanishes. Therefore, (\ref{Sch2}) can be further reduced to the limit model (\ref{Sch}), which has been rigorously proved in \cite{KGZ-limit}.
Compared to \eqref{Sch}, the semi-limit model \eqref{Sch2} incorporates the impact from the oscillatory potential $I$ to $\psi$ and takes the second initial layer into consideration, which should be more accurate. In Section \ref{sec:4}, we will investigate numerically the convergence rate of the KGZ system (\ref{KGZ}) to the limit models \eqref{Sch} and \eqref{Sch2}.

\section{A uniformly accurate method}\label{sec:3}
In this section, we are going to propose a uniformly accurate (UA) scheme based on (\ref{KGZ decomp}) for solving the KGZ system \eqref{KGZ}. To do so, we consider the one-dimensional case for simplicity of notation, i.e., $d=1$, $\bx=x$ in (\ref{KGZ}), and extensions to high dimensions are straightforward. We truncate the whole space problem (\ref{KGZ}) with $x\in\bR$ onto a bounded interval $x\in\Omega=[-L,L]$ with periodic boundary conditions. The periodic setup has been widely considered for the numerical studies of wave or dispersive type models in the literature \cite{KG-zhao,KGZ,ZE1,Kath2019,UAKG,NUA,APKG,XZ}. Consequently, the decomposed system (\ref{KGZ decomp}) is truncated consistently to
\begin{subequations}\label{KGZ 1d}
\begin{align}
&2i\partial_tz-\partial_{xx} z+(-2|z|^2+q+I)z=0,\quad -L<x<L,\ t>0,\\
&\eps^2\partial_{tt}r-\partial_{xx} r+\frac{r}{\eps^2}+(-2|z|^2+q+I)r
+\eps^2\fe^{it/\eps^2}\partial_{tt}z+\eps^2\fe^{-it/\eps^2}\partial_{tt}\overline{z}=0,\\
&\gamma^2\partial_{tt} I-\partial_{xx} I=0,\\
&\gamma^2\partial_{tt}q-\partial_{xx} q=\partial_{xx} r^2+2\gamma^2\partial_{tt}|z|^2
+2\mathrm{Re}\left[\fe^{2it/\eps^2}\partial_{xx} z^2+2\fe^{it/\eps^2}\partial_{xx} (zr)\right],
\end{align}\end{subequations}
with initial and boundary values
\begin{equation*}
\left\{
\begin{aligned}
&z(x,0)=z_0=\frac{1}{2}[\psi_0-i\psi_1],\quad r(x,0)=q(x,0)=0,\quad I(x,0)=\phi_0+2|z_0|^2;\\
& \partial_tr(x,0)=-2\mathrm{Re}(\partial_tz(x,0)),\quad
\partial_t I(x,0)=\frac{\phi_1}{\gamma}+2\mathrm{Im}(\overline{z_0}\p_{xx} z_0),\quad \partial_t q(x,0)=0;\\
&z(-L,t)=z(L,t),\  r(-L,t)=r(L,t),\  I(-L,t)=I(L,t),\  q(-L,t)=q(L,t),\ t\geq0.
\end{aligned}\right.
\end{equation*}
We shall derive the scheme and meanwhile provide some clues on the UA property of the truncation error.

First of all, we denote $\tau=\Delta t>0$ as the time step for discretizing the time direction and denote $t_n=n\tau,\,n=0,1\ldots$.
For the part $I(\bx,t)$, obviously we have the exact solution from the free wave equation (\ref{layer}), i.e.,
\begin{equation}
  I(x,t)=\sum\limits_{l\in\bZ}\wh{I}_l(t)e^{i\mu_l(x+L)},\quad \widehat{I}_l(t)=\cos(\theta_l t)\widehat{I}_l(0)+\fl{\sin(\theta_l t)}{\theta_l} \widehat{I}_l'(0),\quad t\geq0,
\end{equation}
where $\mu_l=\frac{\pi l}{L}$, $\theta_l=\fl{\mu_l}{\gm}$.

{\bf{Splitting scheme for $z$.}} To obtain $z(x,t)$, we split the equation for $z$ into two subflows:
$$\Phi_k^t:\ 2i\partial_tz-\partial_{xx} z=0\quad\mbox{and}\quad
\Phi_p^t:\ 2i\partial_tz+(-2|z|^2+q+I)z=0.$$
For some $n\geq0$, we apply the Lie-Trotter splitting scheme to get $z(x,t_{n+1})$ as
\begin{equation}\label{z1}
z(x,t_{n+1})\approx \Phi_k^\tau\circ\Phi_p^\tau(z(x,t_n)).
\end{equation}
Note the Lie-Trotter splitting has been identified to offer uniform first order accuracy for integrating a nonlinear Schr\"{o}dinger equation with highly oscillatory potential \cite{NLSOP}.
The flow $\Phi_k^\tau$ can be integrated exactly in the Fourier space. As for $\Phi_p^\tau$, we have
$$\Phi_p^\tau(z(x,t_n))=z(x,t_n)\fe^{\frac{i}{2}\int_0^\tau [-2|z(x,t_n+s)|^2+q(x,t_n+s)+I(x,t_n+s)]ds}.$$
Note that in $\Phi_p^\tau$, $|z(x,t_n+s)|\equiv|z(x,t_n)|$ for $0\leq s\leq\tau$ and
$$J^n(x):=\int_0^\tau I(x,t_n+s)ds=\sum\limits_{l\in\bZ}
\left[\frac{\sin(\theta_l\tau)}{\theta_l}\widehat{I}_l(t_n)+\frac{1-\cos(\theta_l\tau)}{\theta_l^2} \widehat{I}_l'(t_n)\right]e^{i\mu_l(x+L)}.$$
We just approximate $q(x,t_n+s)\approx q(x,t_n)$ to get
\begin{equation}\label{z2}\Phi_p^\tau(z(x,t_n))\approx z(x,t_n)\fe^{\frac{i}{2}\left[-2\tau|z(x,t_n)|^2+\tau q(x,t_n)+J^n(x)\right]}.
\end{equation}
Note that the truncation error here is $O(\tau^2)$, which is uniform for $0<\eps<\gm\leq1$ since $\partial_tq=O(1)$.

{\bf{Exponential integrator for $r$.}} To obtain $r(x,t)$, we firstly write the
equation of $r$ in the Fourier space:
$$\eps^2\widehat{r}_l''(t)+\mu_l^2\widehat{r}_l(t)+\frac{1}{\eps^2}\widehat{r}_l(t)+\widehat{f}_l(t)
+\eps^2\fe^{it/\eps^2}\widehat{z}_l''(t)+\eps^2\fe^{-it/\eps^2}\widehat{\overline{z}}_l''(t)=0,\quad t>0,\ l\in\bZ,$$
where for simplicity we denote
$$f(x,t)=\left(-2|z(x,t)|^2+q(x,t)+I(x,t)\right)r(x,t).$$
For some $n\geq0$, suppose that we know $\widehat{r}_l(t_n)$ and $\widehat{r}_l'(t_n)$. Then we write the solution with the Duhamel's formula:
\begin{align}
\widehat{r}_l(t_{n+1})=&\cos(\omega_l\tau)\widehat{r}_l(t_n)+\frac{\sin(\omega_l\tau)}{\omega_l}
\widehat{r}_l'(t_n)-\int_0^\tau\frac{\sin(\omega_l(\tau-s))}{\eps^2\omega_l}
\left[\widehat{f}_l(t_n+s)\right.\nonumber\\
&\left.+\eps^2\fe^{i(t_n+s)/\eps^2}\widehat{z}_l''(t_n+s)
+\eps^2\fe^{-i(t_n+s)/\eps^2}\widehat{\overline{z}}_l''(t_n+s)\right]ds,\label{r duhamel}
\end{align}
where $\omega_l=\frac{\sqrt{1+\eps^2\mu_l^2}}{\eps^2}$. To get $\widehat{r}_l(t_{n+1})$, we apply proper quadrature rules to the terms in integration. For the one involving $\widehat{f}_l$,  we apply the right rectangle rule to simply have:
$$\int_0^\tau\frac{\sin(\omega_l(\tau-s))}{\eps^2\omega_l}
\widehat{f}_l(t_n+s)ds\approx0.$$
Note that $f=O(r)=O(\eps^2)$ and $\partial_tf=O(1)$ since
$\eps<\gamma$, this quadrature error is uniformly at $O(\tau^2)$. For the other two terms, we take
$$\widehat{z}_l''(t_n+s)\approx \frac{\widehat{z}_l'(t_{n+1})-
\widehat{z}_l'(t_n)}{\tau},\quad \widehat{\overline{z}}_l''(t_n+s)\approx \frac{\widehat{\overline{z}}_l'(t_{n+1})-
\widehat{\overline{z}}_l'(t_n)}{\tau},\quad 0\leq s\leq\tau,$$
where the truncation error is $O(\tau\partial_t^3z)$,
and then we integrate the rest trigonometric parts exactly, which is in the spirit of Gautschi type quadrature \cite{Gautschi,Hochbruck}:
\begin{align*}&\int_0^\tau\frac{\sin(\omega_l(\tau-s))}{\eps^2\omega_l}
\left[\eps^2\fe^{i(t_n+s)/\eps^2}\widehat{z}_l''(t_n+s)
+\eps^2\fe^{-i(t_n+s)/\eps^2}\widehat{\overline{z}}_l''(t_n+s)\right]ds\\
\approx& \fe^{it_n/\eps^2}\sigma_l\big(\widehat{z}_l'(t_{n+1})-
\widehat{z}_l'(t_n)\big)+\fe^{-it_n/\eps^2}\overline{\sigma_l}\big
(\widehat{\overline{z}}_l'(t_{n+1})-
\widehat{\overline{z}}_l'(t_n)\big),
\end{align*}
where
\be\label{sigma}
\begin{split}
  \sigma_l&=\int_0^\tau\frac{\sin(\omega_l(\tau-s))}{\tau\omega_l}
  \fe^{is/\eps^2}ds\\
  &=\frac{\eps^2}{\tau\omega_l(\eps^4\omega_l^2-1)}
  \left[\eps^2\omega_l(\fe^{i\tau/\eps^2}-\cos(\omega_l\tau))
  -i\sin(\omega_l\tau)\right].
\end{split}
\ee
Note that $\partial_t^3z=O(\partial_{tt}I)=O(1/\gamma^2)$, the quadrature error here is thus  $O(\tau^2\eps^2/\gamma^2)$, which is uniformly at $O(\tau^2)$ by noticing $\eps<\gamma$. The approximation for $\widehat{r}_l(t_{n+1})$ in total reads as
\begin{align}
\widehat{r}_l(t_{n+1})&\approx\cos(\omega_l\tau)\widehat{r}_l(t_n)
+\frac{\sin(\omega_l\tau)}{\omega_l}
\widehat{r}_l'(t_n)-\fe^{it_n/\eps^2}\sigma_l\big[\widehat{z}_l'(t_{n+1})-
\widehat{z}_l'(t_n)\big]\nonumber\\
&\quad-\fe^{-it_n/\eps^2}\overline{\sigma_l}\big
[\widehat{\overline{z}}_l'(t_{n+1})-
\widehat{\overline{z}}_l'(t_n)\big].\label{r app}
\end{align}
Besides the UA truncation error, another advantage of the above approximation is that we do not need to compute $\widehat{z}_l''$. Instead, we only need to get $\widehat{z}_l'(t_n)$ which is directly given by the equation of $z$:
$$\partial_tz(x,t_n)=\frac{i}{2}\left[-\partial_{xx}z(x,t_n)
+(-2|z(x,t_n)|^2+q(x,t_n)+I(x,t_n))z(x,t_n)\right],\quad n\geq0.$$
Meanwhile, by the derivative of the Duhamel's formula, we have
\begin{align*}
\widehat{r}_l'(t_{n+1})=&-\omega_l\sin(\omega_l\tau)\widehat{r}_l(t_n)+\cos(\omega_l\tau)
\widehat{r}_l'(t_n)-\int_0^\tau\frac{\cos(\omega_l(\tau-s))}{\eps^2}
\left[\widehat{f}_l(t_n+s)
\right.\\
&\left.+\eps^2\fe^{i(t_n+s)/\eps^2}\widehat{z}_l''(t_n+s)+\eps^2\fe^{-i(t_n+s)/\eps^2}
\widehat{\overline{z}}_l''(t_n+s)\right]ds.
\end{align*}
We approximate the functions in the integration in the similar manner as for $\widehat{r}_l(t_{n+1})$ to get
\begin{align}
\widehat{r}_l'(t_{n+1})\approx&-\omega_l\sin(\omega_l\tau)\widehat{r}_l(t_n)+\cos(\omega_l\tau)
\widehat{r}_l'(t_n)-\frac{\tau}{\eps^2}\widehat{f}_l(t_{n+1})\nonumber\\
&-\fe^{it_{n}/\eps^2}\dot{\sigma}_l\left[\widehat{z}_l'(t_{n+1})
-\widehat{z}_l'(t_{n})\right]
-\fe^{-it_{n}/\eps^2}\overline{\dot{\sigma}}_l\left[\widehat{\overline{z}}_l'(t_{n+1})-
\widehat{\overline{z}}_l'(t_n)\right],\label{dr app}
\end{align}
where
\be\label{dsigma}
\begin{split}
  \dot{\sigma}_l&=\fl{1}{\tau}\int_0^\tau\cos(\omega_l(\tau-s))
  \fe^{is/\eps^2}ds\\
  &=\frac{\eps^2}{\tau(\eps^4\omega_l^2-1)}
  \left[i\fe^{i\tau/\eps^2}-i\cos(\omega_l\tau)+\eps^2\omega_l\sin(\omega_l\tau)\right].
\end{split}
\ee
The approximations (\ref{r app}) and (\ref{dr app}) complete an update of $r(x,t)$ from $t_n$ to $t_{n+1}$ in
the type of the exponential (or trigonometric) integrator \cite{HO}.

{\bf{Exponential integrator for $q$.}} To obtain $q(x,t)$, we begin similarly by writing the equation
of $q$ in the Fourier space:
\begin{align*}
 \gamma^2\widehat{q}_l''+\mu_l^2 \widehat{q}_l=&2\gamma^2\widehat{(|z|^2)}_l''
-\fe^{2it/\eps^2}\mu_l^2 \widehat{(z^2)}_l-2\fe^{it/\eps^2}\mu_l^2\widehat{(zr)}_l
-\fe^{-2it/\eps^2}\mu_l^2\widehat{(\overline{z}^2)}_l \\
&-2\fe^{-it/\eps^2}\mu_l^2 \widehat{(\overline{z}r)}_l
-\mu_l^2 \widehat{(r^2)}_l,\quad t>0,\ l\in\bZ.
\end{align*}
The Duhamel's formula gives
\be\label{q_app}
\widehat{q}_l(t_{n+1})=\cos(\theta_l\tau)\widehat{q}_l(t_{n})+\frac{\sin(\theta_l\tau)}{\theta_l}
\widehat{q}_l'(t_{n})+A_{1,l}^n-A_{2,l}^n-A_{3,l}^n-A_{4,l}^n,
\ee
where
\begin{align*}
  &A_{1,l}^n=\int_0^\tau\frac{2\sin(\theta_l(\tau-s))}{\theta_l}\widehat{(|z|^2)}_l''
(t_n+s)ds,\  A_{2,l}^n=\int_0^\tau \theta_l\sin(\theta_l(\tau-s))\widehat{(r^2)}_l(t_n+s)ds,\\
&A_{3,l}^n=\int_0^\tau\theta_l\sin(\theta_l(\tau-s))\left[\fe^{2i(t_n+s)/\eps^2} \widehat{(z^2)}_l(t_n+s)+\fe^{-2i(t_n+s)/\eps^2}\widehat{(\overline{z}^2)}_l(t_n+s)\right]ds,\\
&A_{4,l}^n=\int_0^\tau 2\theta_l\sin(\theta_l(\tau-s))\left[\fe^{i(t_n+s)/\eps^2}\widehat{(zr)}_l(t_n+s)
+\fe^{-i(t_n+s)/\eps^2}\widehat{(\overline{z}r)}_l(t_n+s)\right]ds.
\end{align*}
Noticing that $\p_{tt}|z|^2=O(1)$ \eqref{absz}, and
\[\p_t^3|z|^2=\mathrm{Im}(\overline{\p_{tt}z}\p_{xx}z+
2\overline{\p_{t}z}\p_{xxt}z+\overline{z}\p_{xxtt}z)=O(1/\gamma),\]
we then are motivated to approximate the trigonometric kernel function with
$s=\tau$ to get $A_{1,l}^n\approx0$ with a uniform truncation error at $O(\tau^2)$.
Similarly, we can get $A_{2,l}^n\approx0$ with a uniform truncation error at $O(\tau^2)$ in view of the fact that $r=(\eps^2)$. For $A_{3,l}^n$, to
make sure that the truncation error is introduced in a uniform manner, we firstly perform an
integration-by-parts to rewrite $A_{3,l}^n$ so that the kernel of the integration part is bounded
as $\eps,\gamma\to0$:
\begin{align*}
A_{3,l}^n&=\fe^{2it_n/\eps^2}\left[\alpha_l(\tau)\widehat{(z^2)}_l(t_{n+1})-\alpha_l(0)\widehat{(z^2)}_l(t_{n})
-\int_0^\tau2\alpha_l(s)\widehat{(z\partial_tz)}_l(t_{n}+s)ds\right]\\
&\quad+\fe^{-2it_n/\eps^2}\left[\overline{\alpha_l}(\tau)\widehat{(\overline{z}^2)}_l(t_{n+1})
-\overline{\alpha_l}(0)\widehat{(\overline{z}^2)}_l(t_{n})
-\int_0^\tau2\overline{\alpha_l}(s)\widehat{(\overline{z\partial_tz})}_l(t_{n}+s)ds\right],
\end{align*}
where
\be\label{alpha}
\begin{split}
\alpha_l(s)&:=\int_0^s\theta_l\sin(\theta_l(\tau-\sigma))e^{2i\sigma/\eps^2}d\sigma
=\frac{\eps^2\theta_l}{4-\eps^4\theta_l^2}\left[
\eps^2\theta_l\cos(\theta_l\tau)+2i\sin(\theta_l\tau)\right.\\
&\qquad\qquad\quad\qquad\qquad\left.-\fe^{2is/\eps^2}\left(\eps^2\theta_l\cos(\theta_l(\tau-s))
+2i\sin(\theta_l(\tau-s))\right)
\right].
\end{split}
\ee
Then we apply the Gautschi type quadrature, i.e., $\widehat{(z\partial_tz)}_l(t_{n}+s)\approx
\widehat{(z\partial_tz)}_l(t_{n})$ and integrate the trigonometric parts exactly to get that
\be\label{A3 app}
\begin{split}
A_{3,l}^n &\approx
\fe^{2it_n/\eps^2}\left[\alpha_l(\tau)\widehat{(z^2)}_l(t_{n+1})
-\kappa_l\widehat{(z\partial_tz)}_l(t_{n})\right]\\
&\quad+\fe^{-2it_n/\eps^2}\left[\overline{\alpha_l}(\tau)\widehat{(\overline{z}^2)}_l(t_{n+1})
-\overline{\kappa_l}\widehat{(\overline{z\partial_tz})}_l(t_{n})\right],
\end{split}
\ee
where
\begin{align}
\kappa_l=\int_0^\tau2\alpha_l(s)ds
&=\frac{2\eps^4\theta_l}{(4-\eps^4\theta_l^2)^2}
\left[4i\eps^2\theta_l\fe^{2i\tau/\eps^2}-4i\eps^2\theta_l\cos(\theta_l\tau)
+(4+\eps^4\theta_l^2)\sin(\theta_l\tau)\right]\nonumber\\
&\quad+\frac{2\tau\eps^2\theta_l}{4-\eps^4\theta_l^2}\left[
\eps^2\theta_l\cos(\theta_l\tau)+2i\sin(\theta_l\tau)
\right].\label{kappa}
\end{align}
Note that $\alpha_l=O(\eps^2/\gamma)$, the local truncation error here is $O(\tau^2\alpha_l\widehat{(\partial_{tt}z)_l})=O(\tau^2\eps^2/\gamma^2)$ which is at the second order uniformly for $0<\eps<\gamma\leq1$.
For $A_{4,l}^n$, we need to take a delicate approximation of $r$ based on the Duhamel's formula (\ref{r duhamel}) as
$$r(x,t_n+s)\approx \cos(s/\eps^2)r(x,t_n)+
\sin(s/\eps^2)r_p^n(x),\quad
r_p^n(x):=\sum_{l\in\bZ}\frac{\widehat{r}_l'(t_n)}{\omega_l}e^{i\mu_l(x+L)},
$$
where for $0\leq s\leq\tau$ the approximation error is at the order of $O(\gamma\tau)$ by noticing that $r=O(\eps^2)$ and $\omega_l=1/\eps^2+O(1)$. More importantly, this approximation separates the temporal highly oscillatory parts in $r$ from the space variable (so as the Fourier modes). Then by taking
$z(x,t_n+s)\approx z(x,t_n)$, we approximate $A_{4,l}^n$ in the Gautschi-type way with a uniform truncation error at $O(\tau^2)$ as
\be\label{A4 app}
\begin{split}
A_{4,l}^n&\approx\int_0^\tau 2\theta_l\sin(\theta_l(\tau-s))\fe^{i(t_n+s)/\eps^2}
\left(\cos(s/\eps^2)\widehat{(zr)}_l(t_n)
+\sin(s/\eps^2)\widehat{(zr_p)}_l(t_n)\right)ds\\
&\,\,+\int_0^\tau 2\theta_l\sin(\theta_l(\tau-s))\fe^{-i(t_n+s)/\eps^2}
\left(\cos(s/\eps^2)\widehat{(\overline{z}r)}_l(t_n)
+\sin(s/\eps^2)\widehat{(\overline{z}r_p)}_l(t_n)\right)ds\\
&=\fe^{it_n/\eps^2}\left[\chi_l^1\widehat{(zr)}_l(t_n)
+\chi_l^2\widehat{(z(t_n)r^n_p)}_l\right]+\fe^{-it_n/\eps^2}\left[\overline{\chi_l^1}
\widehat{(\overline{z}r)}_l(t_n)
+\overline{\chi_l^2}\widehat{(\overline{z(t_n)}r^n_p)}_l\right],
\end{split}
\ee
where we denote
\begin{subequations}\label{chi}
\begin{align}
 \chi_l^1&=\int_0^\tau 2\theta_l\sin(\theta_l(\tau-s))\fe^{is/\eps^2}\cos(s/\eps^2)ds\\
 &=1-\cos(\theta_l\tau)+
 \frac{\eps^2\theta_l}{4-\eps^4\theta_l^2}
 \left[2i\sin(\theta_l
 \tau)+\eps^2\theta_l\cos(\theta_l\tau)
 -\eps^2\theta_l\fe^{2i\tau/\eps^2}
 \right],\nonumber\\
 \chi_l^2&=\int_0^\tau 2\theta_l\sin(\theta_l(\tau-s))\fe^{is/\eps^2}\sin(s/\eps^2)ds\\
 &=\frac{1}{4-\eps^4\theta_l^2}\left[2\eps^2\theta_l\sin(\theta_l
 \tau)-4i\cos(\theta_l\tau)+i\left(4
 +\eps^4\theta_l^2(\fe^{2i\tau/\eps^2}-1)\right)
 \right].\nonumber
\end{align}
\end{subequations}

To complete the integration scheme, by the derivative of the Duhamel's formula of $q$, we have
\begin{align*}
\widehat{q}_l'(t_{n+1})=&-\theta_l\sin(\theta_l\tau)\widehat{q}_l(t_{n})
+\cos(\theta_l\tau)\widehat{q}_l'(t_{n})+B_{1,l}^n-B_{2,l}^n-B_{3,l}^n
-B_{4,l}^n,\quad  n\geq0,
\end{align*}
where
\begin{align*}
  &B_{1,l}^n=\int_0^\tau2\cos(\theta_l(\tau-s))\widehat{(|z|^2)}_l''
(t_n+s)ds,\\
& B_{2,l}^n=\int_0^\tau \theta_l^2\cos(\theta_l(\tau-s))\widehat{(r^2)}_l(t_n+s)ds,\\
&B_{3,l}^n=\int_0^\tau\theta_l^2\cos(\theta_l(\tau-s))\left[\fe^{2i(t_n+s)/\eps^2} \widehat{(z^2)}_l(t_n+s)+\fe^{-2i(t_n+s)/\eps^2}\widehat{(\overline{z}^2)}_l(t_n+s)\right]ds,\\
&B_{4,l}^n=\int_0^\tau 2\theta_l^2\cos(\theta_l(\tau-s))\left[\fe^{i(t_n+s)/\eps^2}\widehat{(zr)}_l(t_n+s)
+\fe^{-i(t_n+s)/\eps^2}\widehat{(\overline{z}r)}_l(t_n+s)\right]ds.
\end{align*}
We approximate these integral terms in a similar way as before. For $B_{1,l}^n$, we set $s=\tau$ in the
kernel function to get
\begin{equation}\label{B1 app}B_{1,l}^n\approx 2\widehat{(|z|^2)}_l'(t_{n+1})-2\widehat{(|z|^2)}_l'(t_{n}).
\end{equation}
For $B_{2,l}^n$, we use the left rectangle rule to get
\begin{equation}\label{B2 app}B_{2,l}^n\approx \tau\theta_l^2\cos(\theta_l\tau)\widehat{(r^2)}_l(t_n).
\end{equation}
For $B_{3,l}^n$, we apply the same integration-by-parts and the Gautschi's quadrature as for $A_{3,l}^n$ to get
\begin{align}
B_{3,l}^n
&\approx
\fe^{2it_n/\eps^2}\left[\beta_l(\tau)\widehat{(z^2)}_l(t_{n+1})-\beta_l(0)\widehat{(z^2)}_l(t_{n})
-\rho_l\widehat{(z\partial_tz)}_l(t_{n})\right]\label{B3 app}\\
&\quad+\fe^{-2it_n/\eps^2}\left[\overline{\beta_l}(\tau)\widehat{(\overline{z}^2)}_l(t_{n+1})
-\overline{\beta_l}(0)\widehat{(\overline{z}^2)}_l(t_{n})
-\overline{\rho_l}\widehat{(\overline{z\partial_tz})}_l(t_{n})\right],\nonumber
\end{align}
where
\be\label{beta}
\begin{split}
\beta_l(s)&=\int_0^s\theta_l^2\cos(\theta_l(\tau-\sigma))e^{2i\sigma/\eps^2}
d\sigma=\frac{\eps^2\theta_l^2}{4-\eps^4\theta_l^2}\left[
2i\cos(\theta_l\tau)-\eps^2\theta_l\sin(\theta_l\tau)\right.\\
&\qquad\qquad\qquad\qquad\left.-\fe^{2is/\eps^2}\left(2i\cos(\theta_l(\tau-s))-\eps^2
\theta_l\sin(\theta_l(\tau-s))\right)\right],
\end{split}
\ee
\be\label{rho}
\begin{split}
\rho_l&=\int_0^\tau2\beta_l(s)ds=\frac{2\tau\eps^2\theta_l^2}{4-\eps^4\theta_l^2}\left[
2i\cos(\theta_l\tau)-\eps^2\theta_l\sin(\theta_l\tau)
\right]\\
&\quad\,+\frac{2\eps^4\theta_l^2}{(4-\eps^4\theta_l^2)^2}
\left[(4+\eps^4\theta_l^2)\cos(\theta_l\tau)-(4+\eps^4\theta_l^2)\fe^{2i\tau/\eps^2}
+4i\eps^2\theta_l\sin(\theta_l\tau)\right].
\end{split}
\ee
For $B_{4,l}^n$, we adopt the similar approximation as for $A_{4,l}^n$ to get
\be\label{B4 app}
\begin{split}
B_{4,l}^n&\approx
\fe^{it_n/\eps^2}\left[\dot{\chi}_l^1\widehat{(zr)}_l(t_n)
+\dot{\chi}_l^2\widehat{(z(t_n)r^n_p)}_l\right]\\
&\quad+\fe^{-it_n/\eps^2}\left[\overline{\dot{\chi}_l^1}
\widehat{(\overline{z}r)}_l(t_n)
+\overline{\dot{\chi}_l^2}\widehat{(\overline{z(t_n)}r^n_p)}_l\right],
\end{split}
\ee
where
\begin{subequations}\label{dchi}
\begin{align}
 \dot{\chi}_l^1&=\int_0^\tau 2\theta_l^2\cos(\theta_l(\tau-s))\fe^{is/\eps^2}\cos(s/\eps^2)ds\\
 &=\theta\sin(\theta_l\tau)-
 \frac{\eps^2\theta_l^2}{4-\eps^4\theta_l^2}
 \left[2i\fe^{2i\tau/\eps^2}-2i\cos(\theta_l
 \tau)+\eps^2\theta_l\sin(\theta_l\tau) \right],\nonumber\\
 \dot{\chi}_l^2&=\int_0^\tau 2\theta_l^2\cos(\theta_l(\tau-s))\fe^{is/\eps^2}\sin(s/\eps^2)ds\\
 &=\frac{2\theta_l}{4-\eps^4\theta_l^2}
 \left[2i\sin(\theta_l
 \tau)+\eps^2\theta_l\cos(\theta_l\tau)-
 \eps^2\theta_l\fe^{2i\tau/\eps^2}
 \right].\nonumber
\end{align}\end{subequations}

\begin{remark}
If one analyzes the local truncation error induced by the above approximations to  $\partial_tr$ and $\partial_tq$, the error would be at $O(\tau^2/\eps^2)$ and $O(\tau^2/\gamma)$, respectively. This would not affect the approximation error for $r$ and $q$ since the coefficients involving $\partial_tr$ and $\partial_tq$ for approximating $r$ and $q$ (cf. \eqref{r app} and \eqref{q_app}) are at the order of $O(\eps^2)$ and $O(\gamma)$, respectively.
The rigorous convergence analysis is undergoing.
\end{remark}

\bigskip

\textbf{UA scheme.}
We summarize the proposed approximations (\ref{z1}), (\ref{z2}), (\ref{r app}), (\ref{dr app}), (\ref{A3 app}), (\ref{A4 app}), (\ref{B1 app})-(\ref{B3 app}) and (\ref{B4 app}) above and present the full scheme for solving the decomposed system (\ref{KGZ 1d}) and hence for solving the KGZ system \eqref{KGZ}.
For spatial discretization, we choose an even integer $N\in\bN^+$ to truncate the Fourier series. We denote
$z^n(x)\approx z(x,t_n)$, $r^n(x)\approx r(x,t_n)$, $\dot{r}^n(x)\approx \partial_tr(x,t_n)$, $I^n(x)\approx I(x,t_n)$, $q^n(x)\approx q(x,t_n)$ and $\dot{q}^n(x)\approx \partial_tq(x,t_n)$ as the numerical solutions for the decomposed system (\ref{KGZ 1d}). Choosing
$z^0(x)=z(x,0)$, $r^0(x)=r(x,0)$, $\dot{r}^0(x)=\partial_tr(x,0)$, $I^0(x)=I(x,0)$, $q^0(x)=q(x,0)$ and
$\dot{q}^0(x)=\partial_tq(x,0)$, we update for $n\geq0$ as
\begin{subequations}\label{scheme}
\begin{align}
  &z^{n+1}(x)=\fe^{-\frac{i}{2}\partial_{xx}}
  \fe^{\fl{i}{2}\left[-2\tau|z^n(x)|^2+\tau q^n(x)+J^n(x)\right]}z^n(x),\\
  &\widehat{(q^{n+1})}_l=\cos(\theta_l\tau)\widehat{(q^n)}_l+\frac{\sin(\theta_l\tau)}{\theta_l}
\widehat{(\dot{q}^n)}_l-A_{l}^n,\\
&\widehat{(\dot{q}^{n+1})}_l=-\theta_l\sin(\theta_l\tau)\widehat{(q^n)}_l
+\cos(\theta_l\tau)\widehat{(\dot{q}^n)}_l+\widehat{(g^n)}_l-
\tau\theta_l^2\cos(\theta_l\tau)\widehat{((r^n)^2)}_l\\
&\qquad\quad\quad \,-B_{l}^n,\nonumber\\
  &\widehat{(r^{n+1})}_l=\cos(\omega_l\tau)\widehat{(r^n)}_l+\frac{\sin(\omega_l\tau)}{\omega_l}
\widehat{(\dot{r}^n)}_l-\fe^{it_n/\eps^2}\sigma_l\left[\widehat{(\dot{z}^{n+1})}_l-
\widehat{(\dot{z}^n)}_l\right]\\
&\qquad\quad\quad \,-\fe^{-it_n/\eps^2}\overline{\sigma_l}\left
[\widehat{(\overline{\dot{z}^{n+1}})}_l-
\widehat{(\overline{\dot{z}^n})}_l\right],\nonumber \\
&\widehat{(\dot{r}^{n+1})}_l=-\omega_l\sin(\omega_l\tau)\widehat{(r^n)}_l
+\cos(\omega_l\tau)
\widehat{(\dot{r}^n)}_l-\frac{\tau}{\eps^2}\widehat{(f^{n+1})}_l
\\
&\qquad\qquad\ -\fe^{it_{n}/\eps^2}\dot{\sigma}_l\big[\widehat{(\dot{z}^{n+1})}_l
-\widehat{(\dot{z}^n)}_l\big]-\fe^{-it_{n}/\eps^2}\overline{\dot{\sigma}_l}
\big[\widehat{(\overline{\dot{z}^{n+1}})}_l-
\widehat{(\overline{\dot{z}^n})}_l\big],\nonumber
\end{align}
\end{subequations}
where $l=-N/2,\ldots, N/2-1$, $J^n(x)=\sum\limits_{l=-N/2}^{N/2-1} \widehat{(J^n)}_l e^{i\mu_l(x+L)}$, with
\[\widehat{(J^n)}_l=\frac{\sin(\theta_lt_{n+1})-\sin(\theta_lt_{n})
}{\theta_l}\widehat{I}_l(0)+\frac{\cos(\theta_lt_n)-\cos(\theta_lt_{n+1})}{\theta_l^2} \widehat{I}_l'(0),\]
and
\begin{align*}
&f^n(x)=\left(-2|z^n(x)|^2+q^n(x)+I^n(x)\right)r^n(x),\\
&g^n(x)=4\mathrm{Re}\left[\overline{z^{n+1}(x)}\dot{z}^{n+1}(x)-\overline{z^{n}(x)}\dot{z}^{n}(x)\right],\\
&\dot{z}^n(x)=\frac{i}{2}\left[-\partial_{xx} z^n(x)+(-2|z^n(x)|^2+q^n(x)+I^n(x))z^n(x)\right],\\
&A_{l}^n
=\fe^{2i\frac{t_n}{\eps^2}}\left[\alpha_l(\tau)\widehat{((z^{n+1})^2)}_l
-\kappa_l\widehat{(z^n\dot{z}^n)}_l\right]+\fe^{-2i\frac{t_n}{\eps^2}}\left[
\overline{\alpha_l}(\tau)\widehat{((\overline{z^{n+1}})^2)}_l
-\overline{\kappa_l}\widehat{(\overline{z^n\dot{z}^n})}_l\right]\\
&\qquad+
\fe^{i\frac{t_n}{\eps^2}}\left[\chi_l^1\widehat{(z^nr^n)}_l
+\chi_l^2\widehat{(z^nr_p^n)}_l\right]
+\fe^{-i\frac{t_n}{\eps^2}}\left[\overline{\chi_l^1}
\widehat{(\overline{z^n}r^n)}_l
+\overline{\chi_l^2}\widehat{(\overline{z^n}r_p^n)}_l\right],\\
&B_l^n=\fe^{2i\frac{t_n}{\eps^2}}\left[\beta_l(\tau)\widehat{((z^{n+1})^2)}_l
-\rho_l\widehat{(z^n\dot{z}^n)}_l\right]+\fe^{-2i\frac{t_n}{\eps^2}}\left[
\overline{\beta_l}(\tau)\widehat{((\overline{z^{n+1}})^2)}_l
-\overline{\rho_l}\widehat{(\overline{z^n\dot{z}^n})}_l\right]\\
&\qquad+
\fe^{i\frac{t_n}{\eps^2}}\left[\dot{\chi}_l^1\widehat{(z^nr^n)}_l
+\dot{\chi}_l^2\widehat{(z^nr_p^n)}_l\right]
+\fe^{-i\frac{t_n}{\eps^2}}\left[\overline{\dot{\chi}_l^1}
\widehat{(\overline{z^n}r^n)}_l
+\overline{\dot{\chi}_l^2}\widehat{(\overline{z^n}r_p^n)}_l\right],
\end{align*}
with
\begin{align*}
  & I^n(x)=\sum\limits_{l=-N/2}^{N/2-1}\left(\cos(\theta_l t_n)\widehat{I}_l(0)+\fl{\sin(\theta_l t_n)}{\theta_l} \widehat{I}_l'(0)\right)e^{i\mu_l(x+L)},\\
  &r_p^n(x)=\sum_{l=-N/2}^{N/2-1}\frac{\widehat{(\dot{r}^n)}_l}{\omega_l}e^{i\mu_l(x+L)}.
\end{align*}
The coefficients $\sigma_l,\dot{\sigma}_l,\alpha_l,\dot{\alpha}_l,\kappa_l,\rho_l,\chi_l^1,\chi_l^2,\dot{\chi}_l^1$ and $\dot{\chi}_l^2$ are defined respectively in (\ref{sigma}), (\ref{dsigma})-(\ref{dchi}).
Based on the multiscale expansion (\ref{mfe}), (\ref{dec-phi}) and the numerical solution from the decomposed system, we have the numerical solution for the KGZ system (\ref{KGZ}): $\psi^n(x)\approx \psi(x,t_n)$ and $\phi^n(x)\approx \phi(x,t_n)$ at each time step $n\in\bN$ as
\begin{align}\label{MTI}
 &\psi^n=\fe^{it_n/\eps^2}z^n+\fe^{-it_n/\eps^2}\overline{z^n}+r^n,
 \quad \phi^n=-2|z^n|^2+I^n+q^n,
\end{align}
and we refer to this scheme as \textbf{multiscale time integrator} (MTI) Fourier spectral method.

The proposed MTI scheme (\ref{MTI}) with (\ref{scheme}) is fully explicit. In practice, we would give a discretization to the space variable $x\in[-L,L]$ with mesh size $\Delta x=2L/N$, and the Fourier coefficients in (\ref{scheme}) are computed by the trigonometric interpolation \cite{ST}. The computational cost at each time level is $O(N\log N)$ thanks to the fast Fourier transform. As we explained along the derivation of the scheme, the truncation error of MTI is uniformly bounded for all $0<\eps<\gamma\leq1$, and therefore the MTI scheme is expected to be (verified numerically in the next section) \emph{uniformly accurate (UA)} for solving the KGZ (\ref{KGZ}) with first order and spectral order of convergence in time and space, respectively. Thanks to the UA property, the MTI scheme is \emph{super-resolution} in time for the high frequencies.

\begin{remark}
  A second order UA scheme for the KGZ system (\ref{KGZ}) in the simultaneous limit regime would be very challenging. There are two main difficulties. The first one is the integration of the nonlinear Schr\"{o}dinger equation with a highly oscillatory potential \cite{NLSOP} where standard Strang splitting can not provide uniform accuracy at the second order. Another difficulty is the necessity of a higher order multiscale expansion for $\psi$ and $\phi$.
\end{remark}

\section{Numerical results}\label{sec:4}
In this section, we present numerical results of the proposed MTI scheme (\ref{MTI}) with (\ref{scheme}) for solving the KGZ system (\ref{KGZ}) in the simultaneous high-plasma-frequency and subsonic limit regime $\eps<\gamma\to0^+$.

\subsection{Accuracy tests}
We begin with two one-dimensional examples to test the accuracy of the proposed MTI scheme. The first one is an example with initial localized wave in the whole space. The second example is the plane wave type solution on a periodic box. In both cases, the chosen initial data belongs to the incompatible class, and the reference solutions are obtained by the EI scheme (\ref{EWI}) (in the appendix) with a very small step size, e.g., $\tau=10^{-6}$ and $\Delta x=1/16$ (or $\Delta x=\pi/128$).

\begin{example}\label{ex1}\emph{(Whole space)}
  We take the truncated computational domain as $x\in\Omega=[-2^{m_0+3},$ $2^{m_0+3}]$ when
  $\eps=1/2^{m_0}$ for $m_0\in\bN$ and $\gamma=2\eps$. The expanding size of the domain is to make sure that the waves during the dynamics are always far away from the boundary such that the periodic boundary condition does not introduce a significant truncation error relative to the problem in the whole space.  The initial data of (\ref{KGZ}) in 1D is given as
  \[\psi_0(x)=\mathrm{sech}(x^2),\quad
  \psi_1(x)=\fe^{-x^2}/2,\quad
\phi_0(x)=\sin(x)\fe^{-x^2},
  \quad \phi_1(x)=\mathrm{sech}(x^2)/\sqrt{\pi}.
\]
\end{example}

\begin{example}\label{ex2}\emph{(Torus)}
  We consider the KGZ system (\ref{KGZ}) on an one-dimensional torus $\Omega=[-\pi,\pi]$. For
  $\eps=1/2^{m_0},\ m_0\in\bN$ and $\gamma=e\eps$, the initial data is given as
\[\psi_0(x)=\frac{2\sin(x)}{2-\cos(x)},\ \,\,
  \psi_1(x)=\cos^2(x),\ \,\,\phi_0(x)=\frac{\cos(x)}{2-\sin(x)},\ \,\, \phi_1(x)=\frac{\sin(x)\cos(2x)}{2-\cos(x)}.\]
\end{example}

\begin{figure}[t!]
\begin{minipage}[t]{0.5\linewidth}
\centering
\includegraphics[height=4.5cm,width=6.9cm]{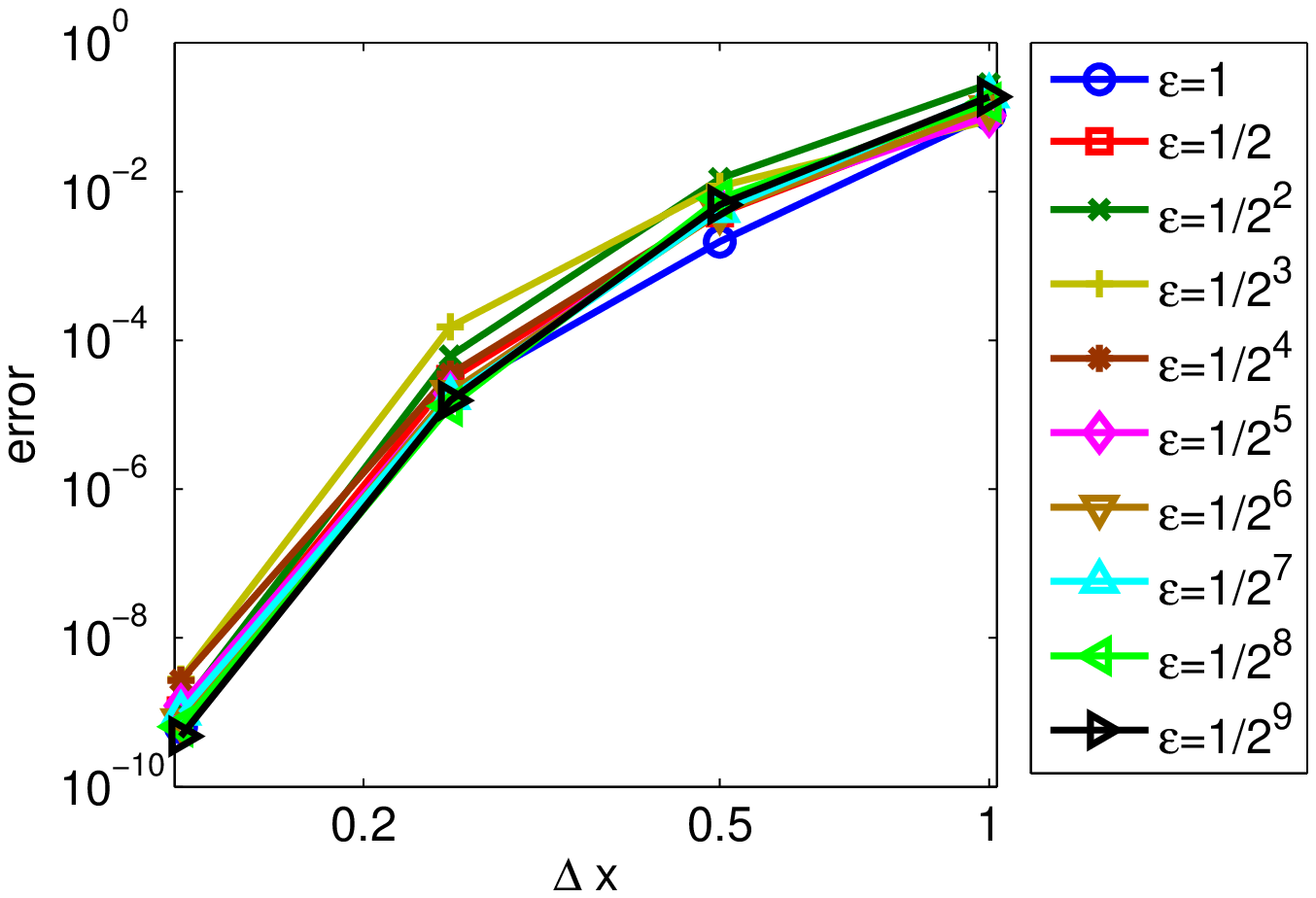}
\end{minipage}%
\hspace{1mm}
\begin{minipage}[t]{0.5\linewidth}
\centering
\includegraphics[height=4.5cm,width=6.9cm]{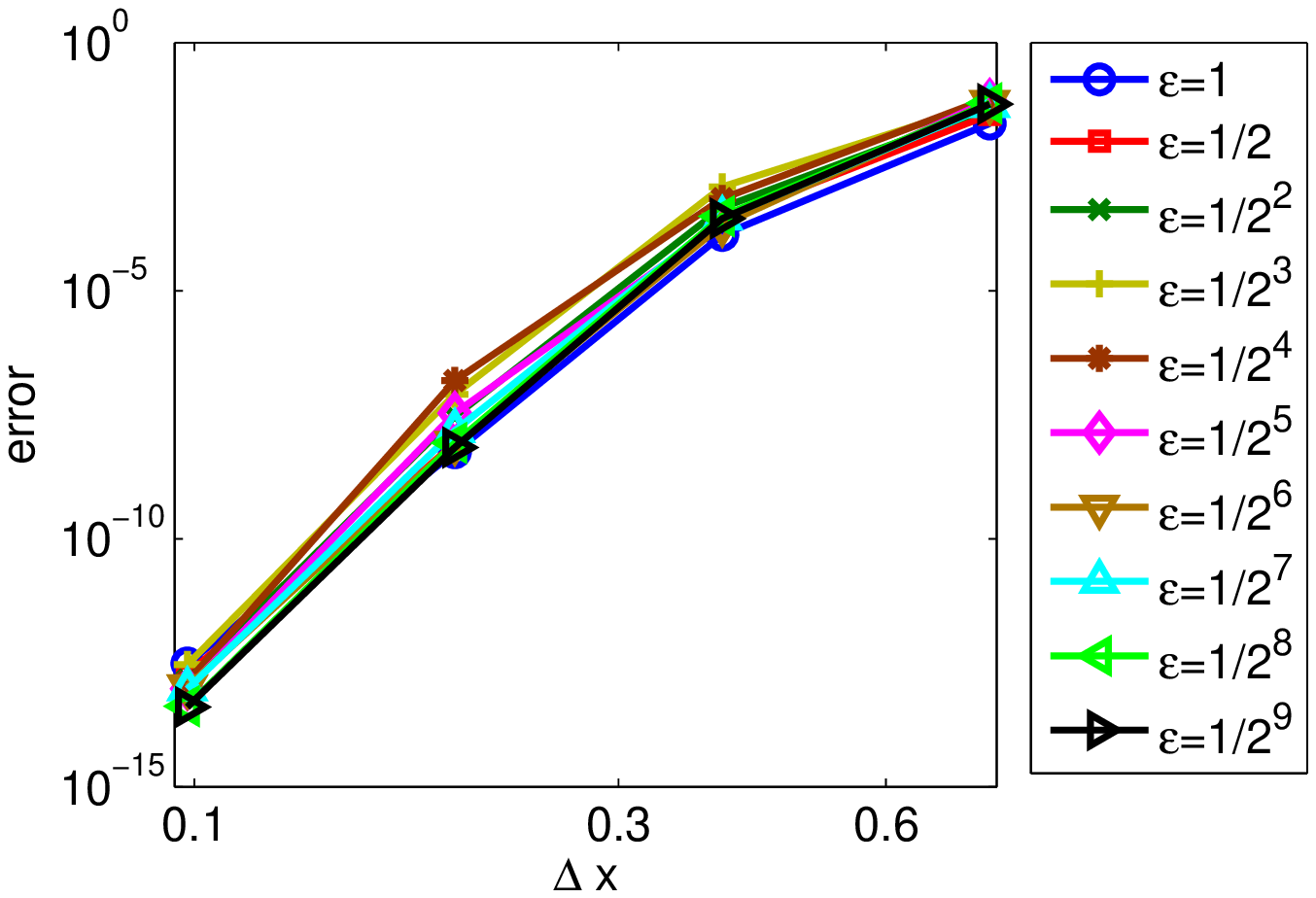}
\end{minipage}
\caption{Spatial errors of MTI at $t=0.5$ for Example \ref{ex1} (left) and \ref{ex2} (right).}\label{fig4}
\end{figure}
For both examples we solve the KGZ system until $t=0.5$ for a wide range of $\eps\in(0,1]$. To quantify the numerical method, we compute the error in maximum norm, i.e.,
\[\mathrm{error}=\|\psi^n-\psi(\cdot,t_n)\|_{L^\infty}+\|\phi^n-\phi(\cdot,t_n)\|_{L^\infty}.\]
The spatial discretization error of MTI under different $\Delta x=|\Omega|/N$ and $\eps$ but fixed $\tau=10^{-7}$
is shown in Figure \ref{fig4}. To observe the temporal approximation error, we fix $\Delta x=1/16$ for Example \ref{ex1} and $\Delta x=\pi/128$ for Example \ref{ex2}, respectively, so that the spatial discretization error is negligible. The error of the MTI scheme under different $\tau$ and $\eps$ is shown in Figures \ref{fig1} and \ref{fig2}, respectively for Examples \ref{ex1} and \ref{ex2}. To make a comparison, we show
the performance of the EI scheme (\ref{EWI}) for Example \ref{ex2} in Figure \ref{fig3}.

\begin{figure}[t!]
\begin{minipage}[t]{0.5\linewidth}
\centering
\includegraphics[height=4.5cm,width=6.9cm]{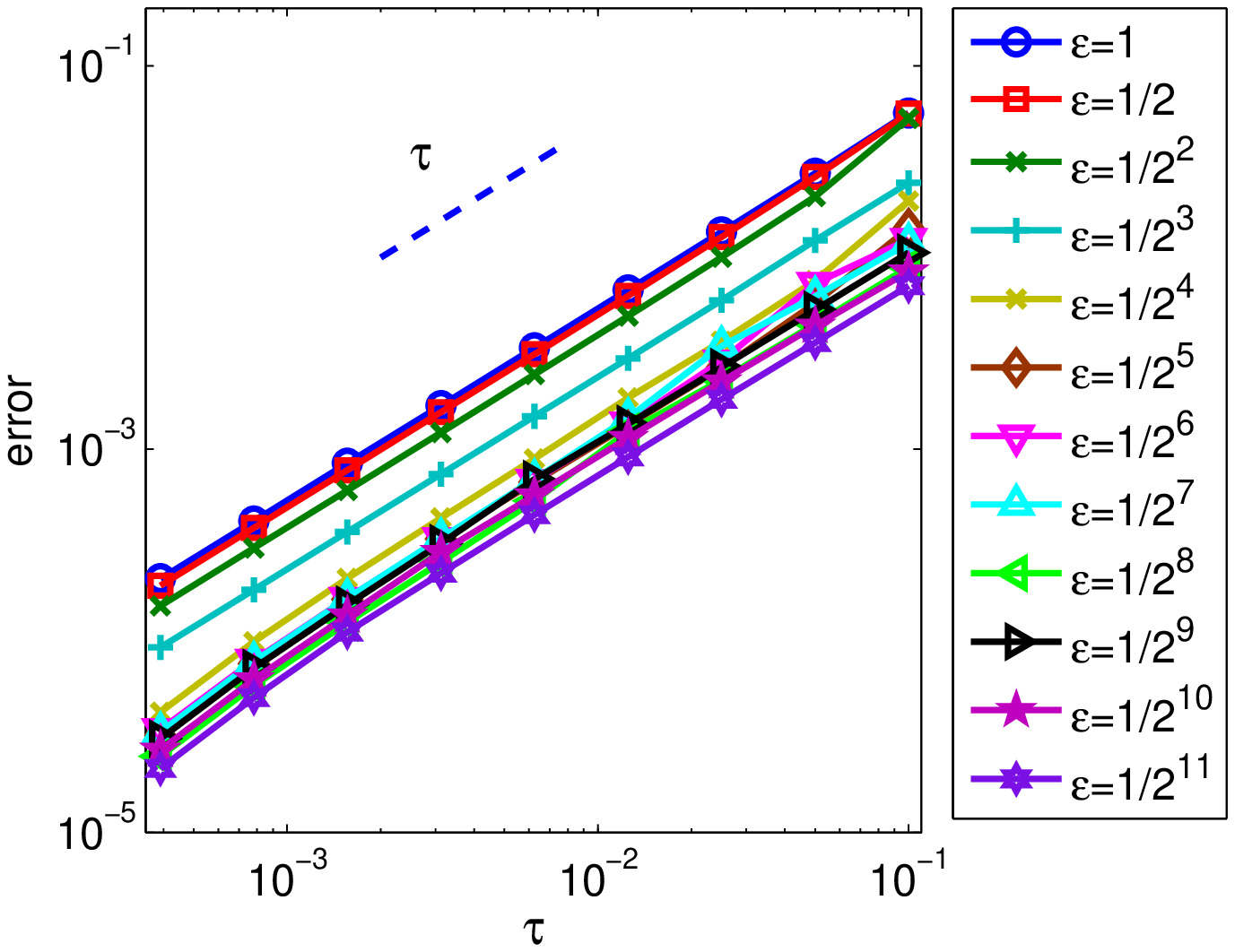}
\end{minipage}%
\hspace{1mm}
\begin{minipage}[t]{0.5\linewidth}
\centering
\includegraphics[height=4.5cm,width=6.9cm]{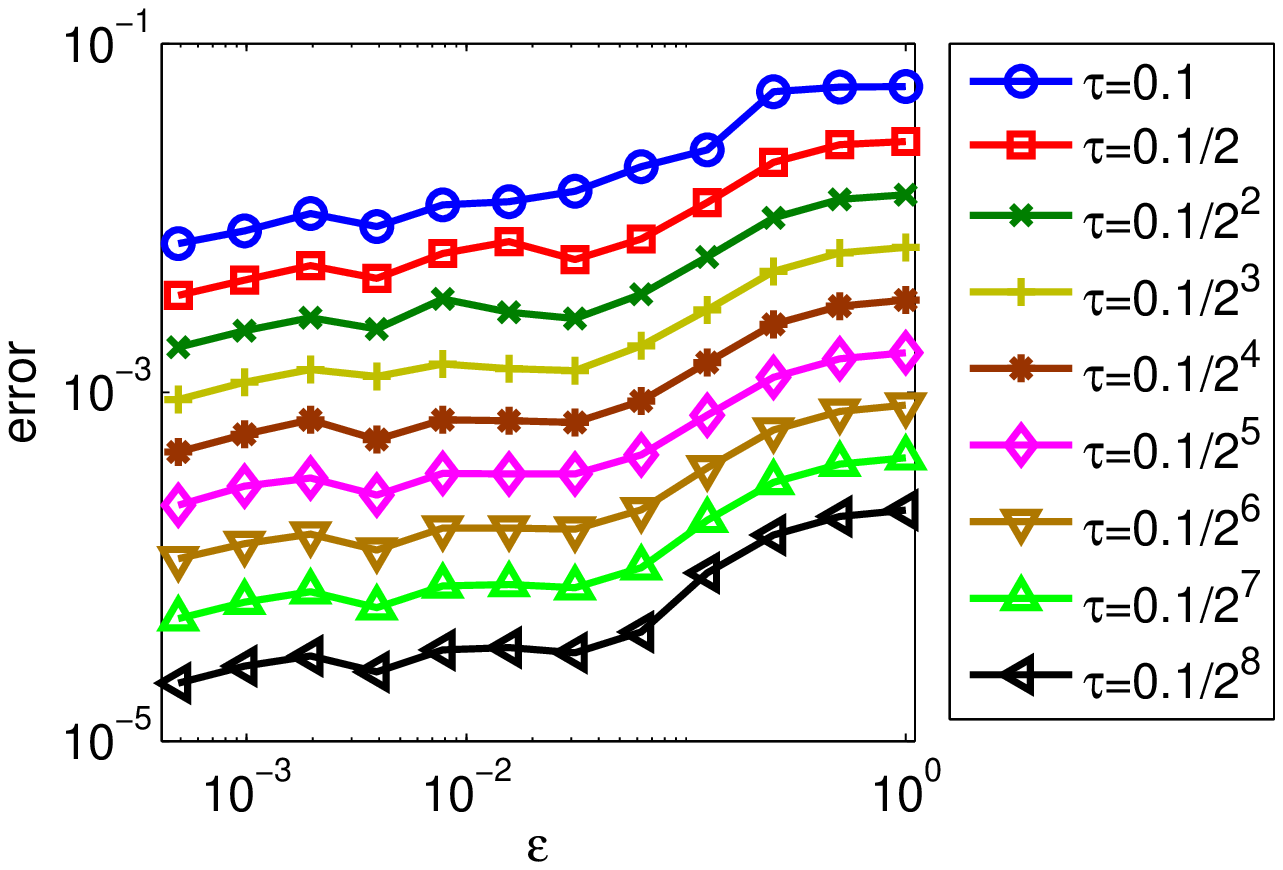}
\end{minipage}
\caption{Temporal errors of MTI at $t=0.5$ for Example \ref{ex1} under different $\eps$ and $\tau$.}\label{fig1}
\end{figure}

\begin{figure}[t!]
\begin{minipage}[t]{0.5\linewidth}
\centering
\includegraphics[height=4.5cm,width=6.9cm]{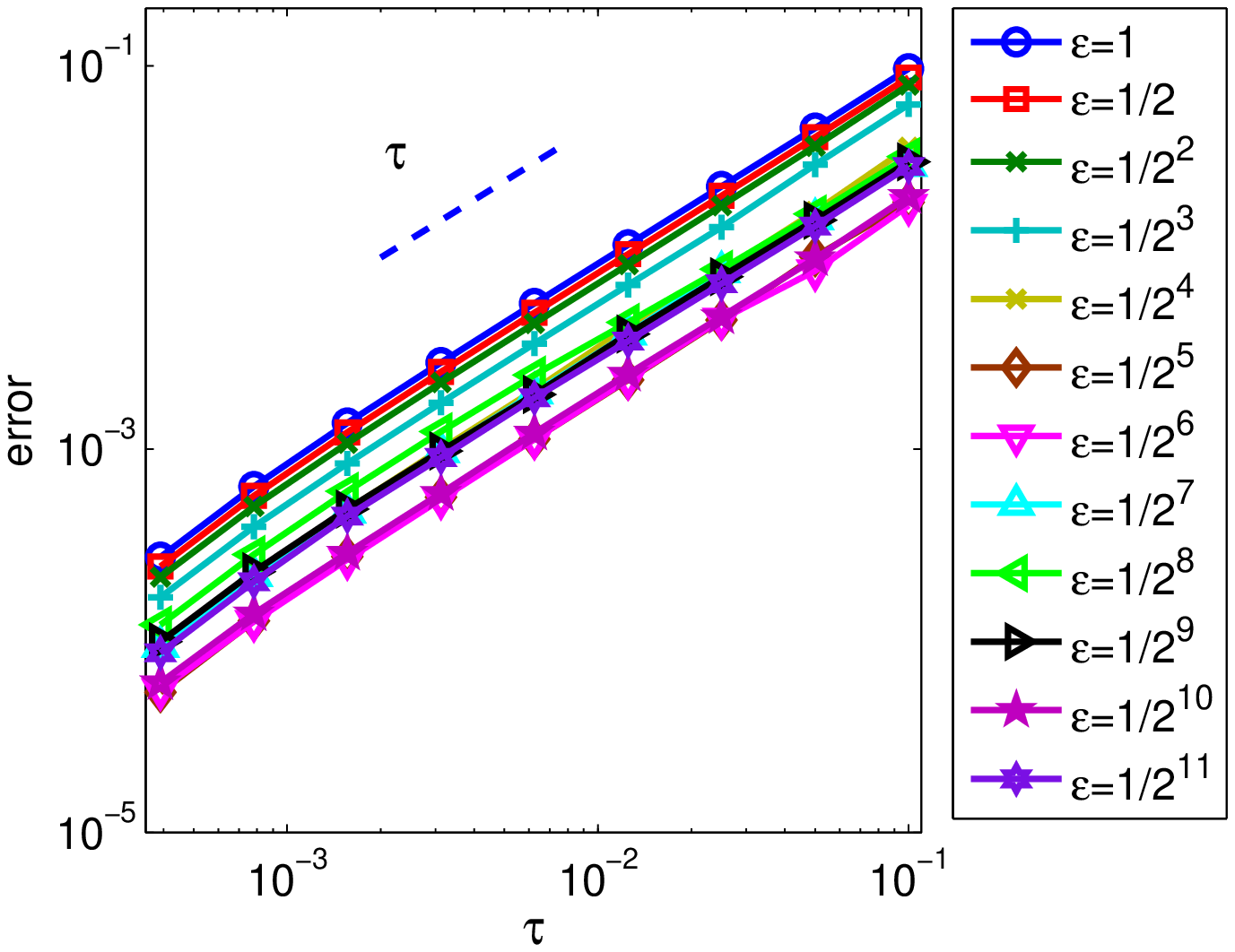}
\end{minipage}%
\hspace{1mm}
\begin{minipage}[t]{0.5\linewidth}
\centering
\includegraphics[height=4.5cm,width=6.9cm]{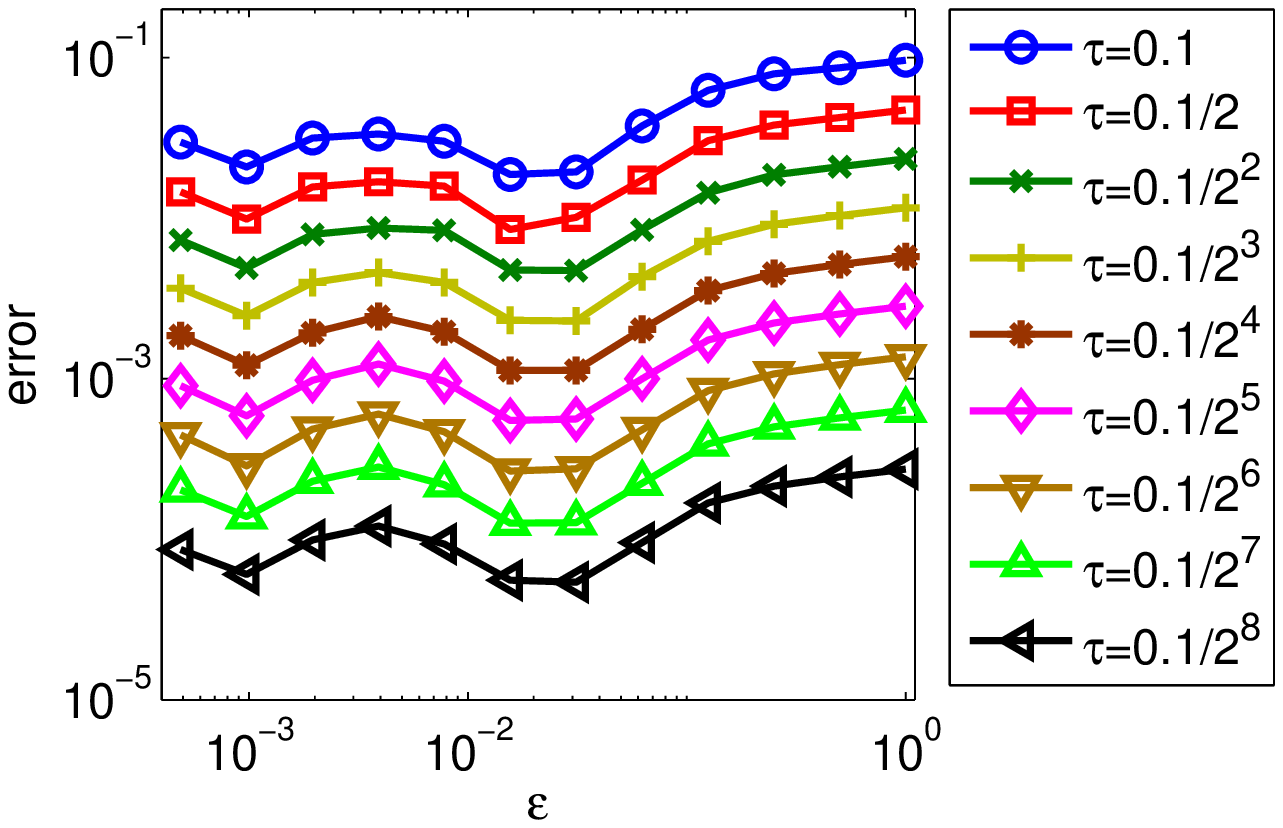}
\end{minipage}
\caption{Temporal errors of MTI at $t=0.5$ for Example \ref{ex2} under different $\eps$ and $\tau$.}\label{fig2}
\end{figure}
\begin{figure}[t!]
\hspace{8mm}
\begin{minipage}[t]{0.8\linewidth}
\centering
\includegraphics[height=4.5cm,width=11cm]{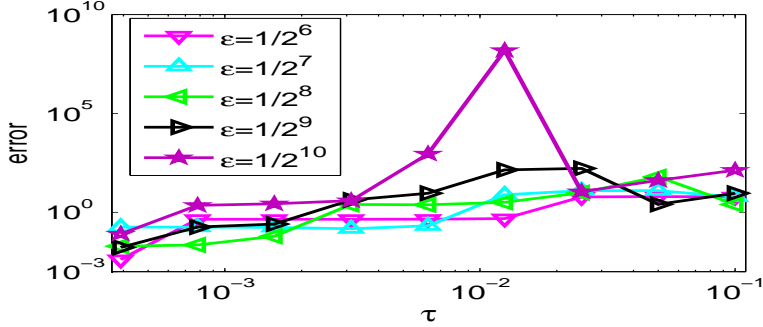}
\end{minipage}
\caption{Comparison: errors of EI at $t=0.5$ for Example \ref{ex2} under different $\eps$ and $\tau$.}\label{fig3}
\end{figure}

Based on the numerical results from Figures \ref{fig4}-\ref{fig3}, it is safe to draw the following conclusions:

1) The MTI scheme (\ref{scheme}) is uniformly accurate for solving the KGZ system (\ref{KGZ}) for all $0<\eps<\gamma\leq1$, where the temporal convergence rate is uniformly linear and the spatial accuracy is uniformly spectral when the solution is smooth in space. In view of the order of the introduced truncation/quadrature errors, the MTI scheme reaches its optimal convergence rate for all fixed $0<\eps<\gamma\leq1$. Thus, we say that the MTI scheme (\ref{MTI}) with (\ref{scheme}) is uniformly and optimally accurate.

2) When $\eps$ becomes small, the EI method (\ref{EWI}) has no accuracy or convergence at all for a wide range of  time step $\tau$ which is a common problem shared by all standard numerical methods, while in such regime the MTI scheme is much more accurate and therefore more efficient.

\subsection{Convergence rates of KGZ to its limit models} We apply the MTI scheme to solve the KGZ system and study the dynamics of the solution in the simultaneous limit $\eps<\gamma\to0^+$. We take the illustrative example from  Section \ref{sec1}.

 \begin{example}\label{ex3}
   We consider the 1D example from the Section \ref{sec1}: i.e., $\gamma=2\eps$ with incompatible initial data \eqref{psi0} and \eqref{E1}.
 \end{example}

Firstly, we study the behavior of each component of the decomposition (\ref{KGZ decomp}) in the limit, by which we illustrate how the decomposition captures the oscillation of the solutions of the KGZ equations. To do so, we solve (\ref{KGZ decomp}) by using the MTI scheme (\ref{scheme}) with a fine mesh on a large domain $[-64,64]$ till $T=1$. The profiles of each component for different $\eps$ are shown in Figure \ref{fig5}, where their combinations through (\ref{mfe}) and (\ref{dec-phi}) give the profiles of $\psi$ and $\phi$ in Figure \ref{fig:incomp}. The fluctuation of the numerical energy:
$$\mathrm{error}=|E^n-E(0)|/|E(0)|$$
during the computation is shown in Figure \ref{fig6}, where $E^n$ denotes the energy (\ref{energy}) of the KGZ at $t_n$ with the numerical solutions from the MTI scheme (\ref{MTI}).
 To verify the order of $r$ and $q$ in Proposition \ref{prop}, we plot  $\|q(\cdot,t)\|_{L^2}/\eps$ and $\|r(\cdot,t)\|_{L^2}/\eps^2$ as functions of time under different $\eps$ in Figure \ref{fig55}.
\begin{figure}[t!]
\begin{minipage}[t]{0.5\linewidth}
\centering
\includegraphics[height=4cm,width=6.9cm]{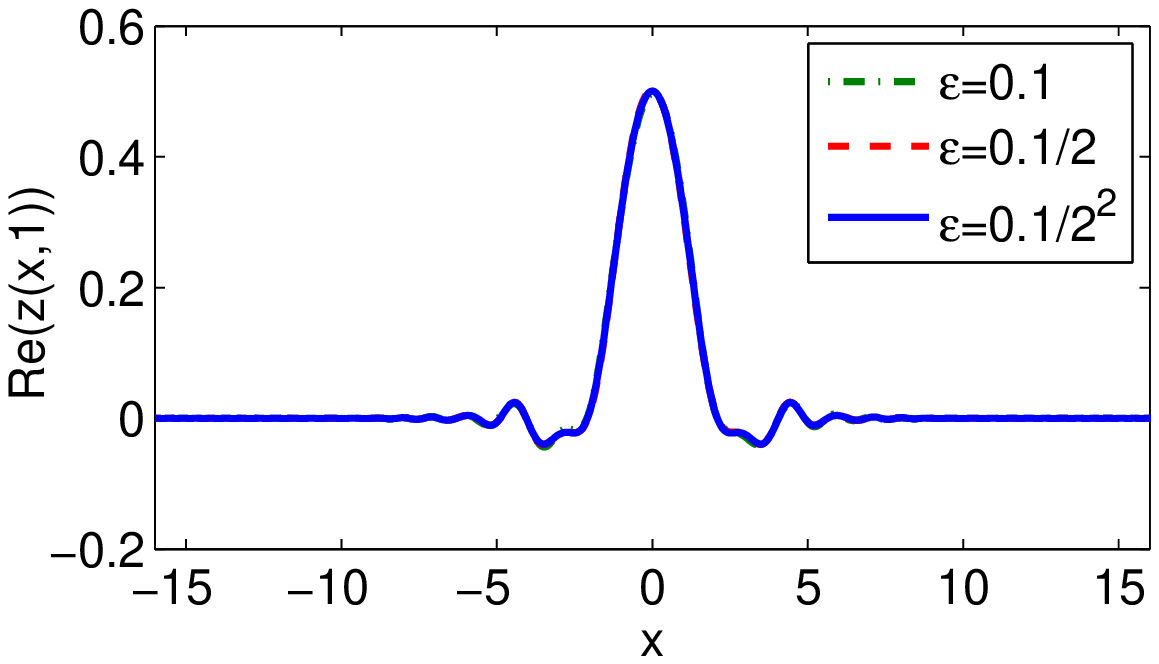}
\end{minipage}%
\hspace{1mm}
\begin{minipage}[t]{0.5\linewidth}
\centering
\includegraphics[height=4cm,width=6.9cm]{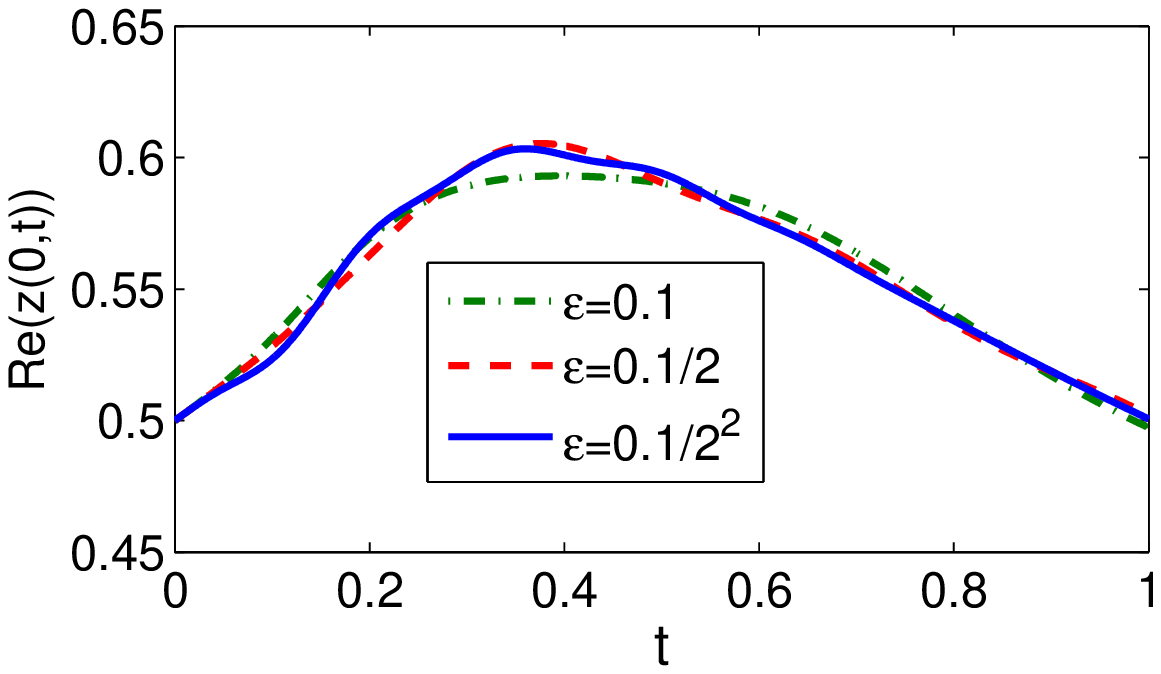}
\end{minipage}
\begin{minipage}[t]{0.5\linewidth}
\centering
\includegraphics[height=4cm,width=6.9cm]{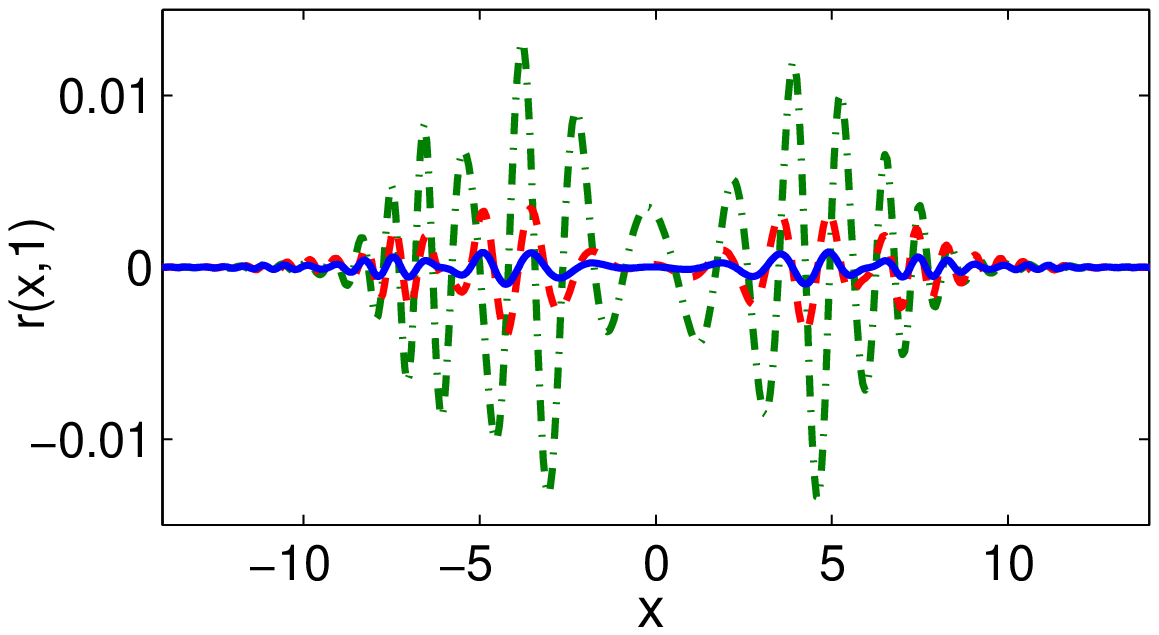}
\end{minipage}%
\hspace{1mm}
\begin{minipage}[t]{0.5\linewidth}
\centering
\includegraphics[height=4cm,width=6.9cm]{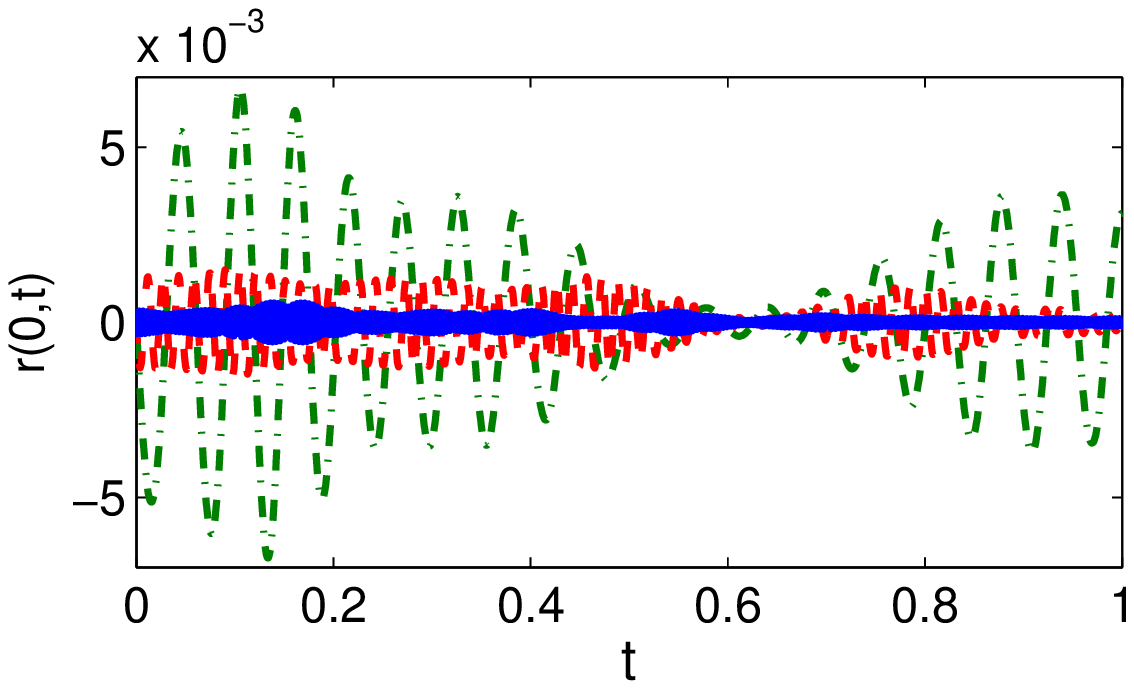}
\end{minipage}
\begin{minipage}[t]{0.5\linewidth}
\centering
\includegraphics[height=4cm,width=6.9cm]{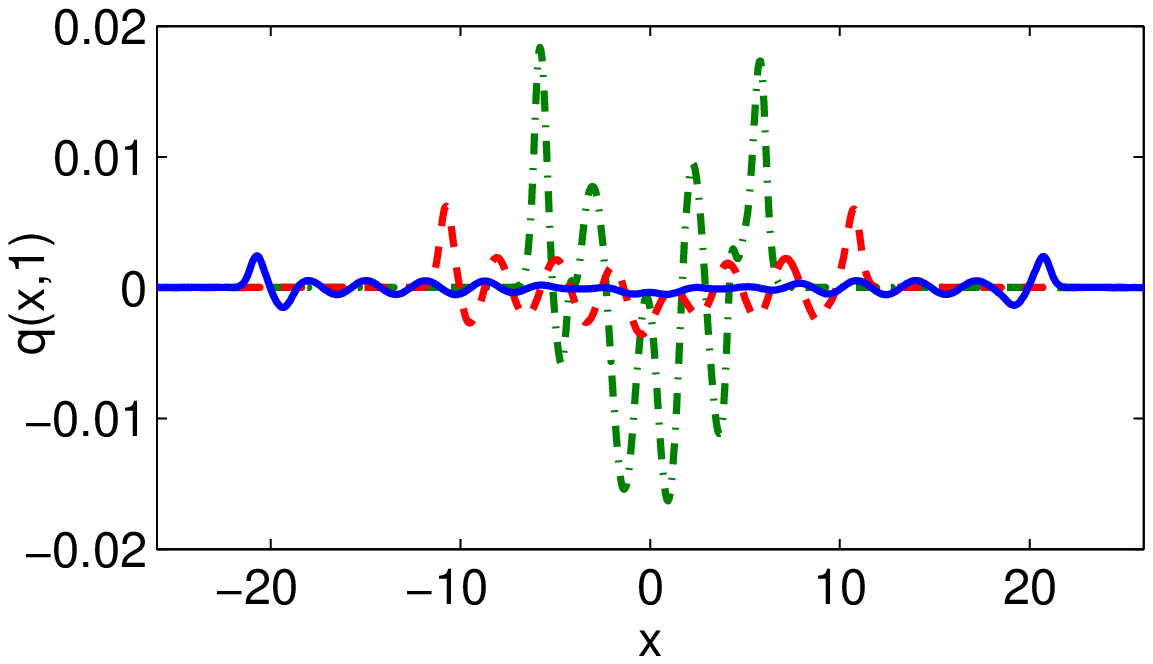}
\end{minipage}%
\hspace{1mm}
\begin{minipage}[t]{0.5\linewidth}
\centering
\includegraphics[height=4cm,width=6.9cm]{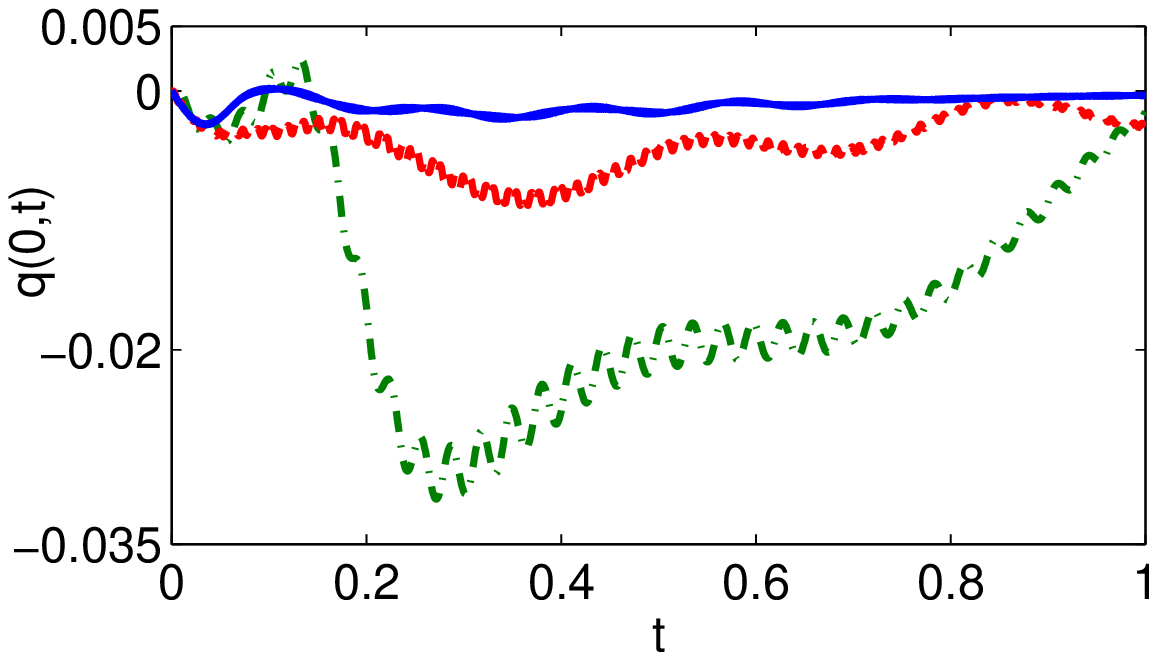}
\end{minipage}
\begin{minipage}[t]{0.5\linewidth}
\centering
\includegraphics[height=4cm,width=6.9cm]{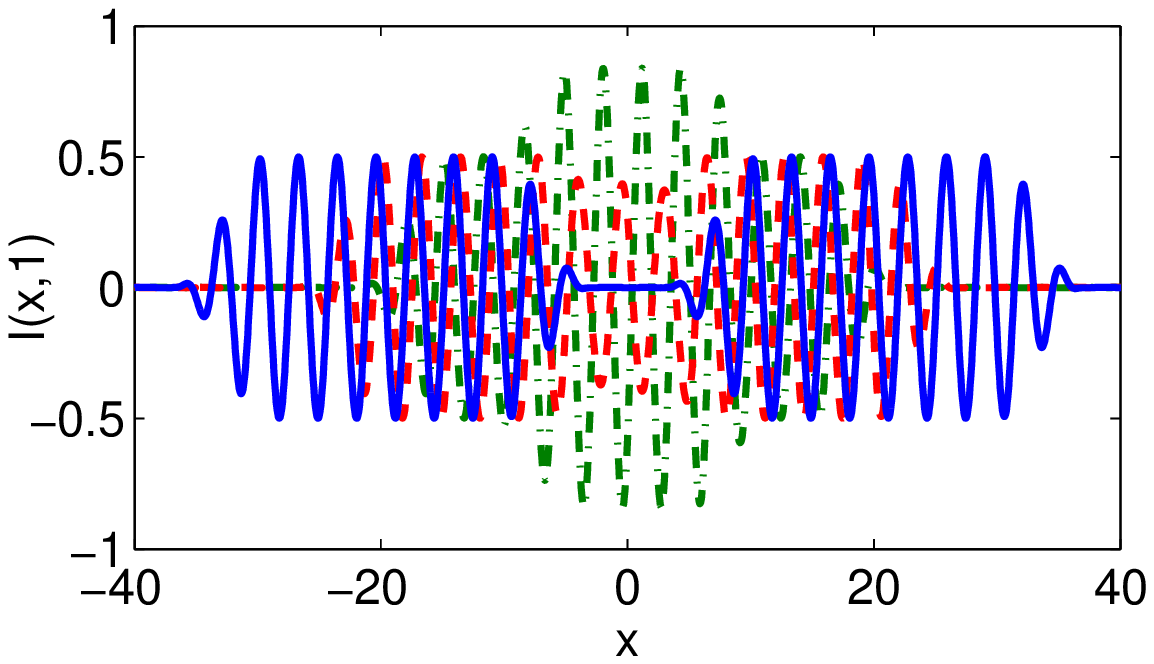}
\end{minipage}%
\hspace{1mm}
\begin{minipage}[t]{0.5\linewidth}
\centering
\includegraphics[height=4cm,width=6.9cm]{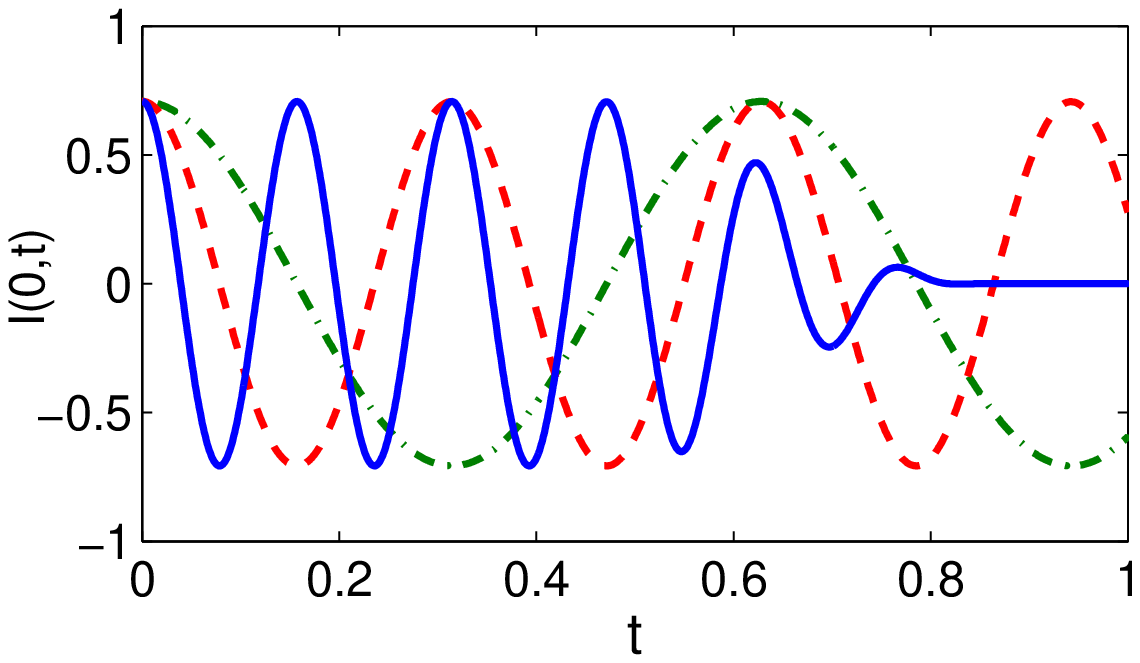}
\end{minipage}
\caption{Profiles of each component in (\ref{KGZ decomp}) for Example \ref{ex3} under
different $\eps$.}\label{fig5}
\end{figure}

\begin{figure}[t!]
\begin{minipage}[t]{0.32\linewidth}
\centering
\includegraphics[height=4cm,width=4.4cm]{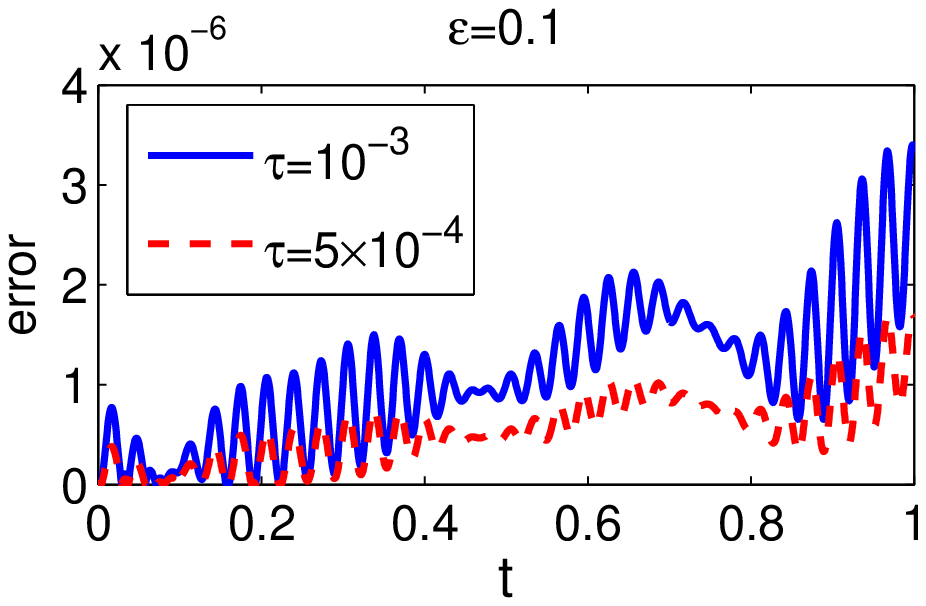}
\end{minipage}%
\hspace{1mm}
\begin{minipage}[t]{0.32\linewidth}
\centering
\includegraphics[height=4cm,width=4.4cm]{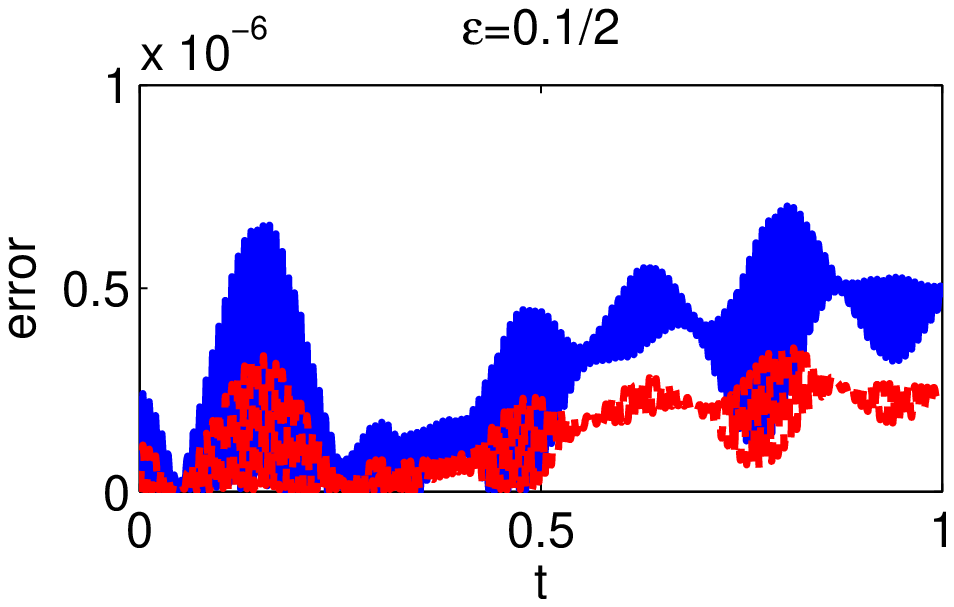}
\end{minipage}
\hspace{1mm}
\begin{minipage}[t]{0.32\linewidth}
\centering
\includegraphics[height=4cm,width=4.4cm]{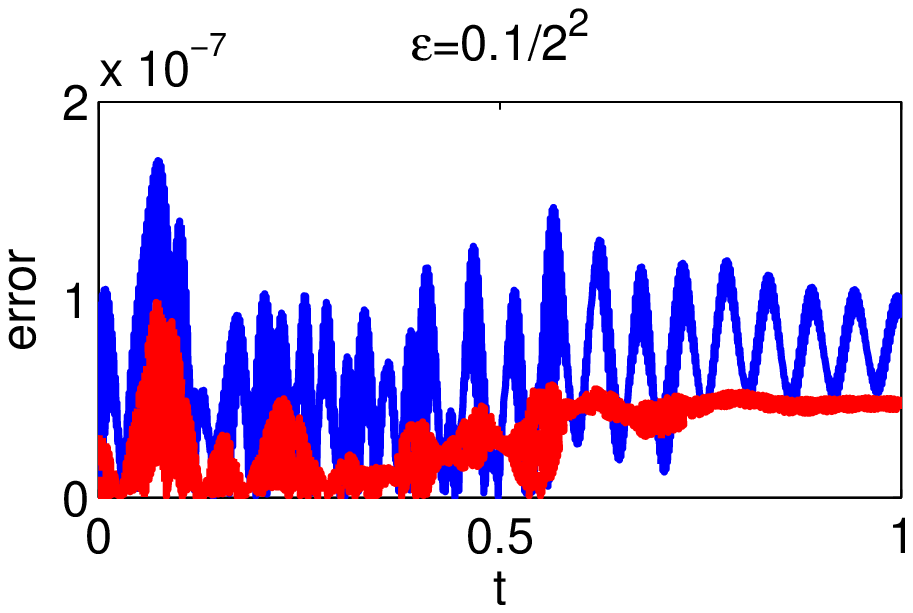}
\end{minipage}
\caption{Energy error of the MTI scheme for KGZ in Example \ref{ex3} under different $\eps$.}\label{fig6}
\end{figure}

\begin{figure}[t!]
\begin{minipage}[t]{0.5\linewidth}
\centering
\includegraphics[height=4cm,width=6.9cm]{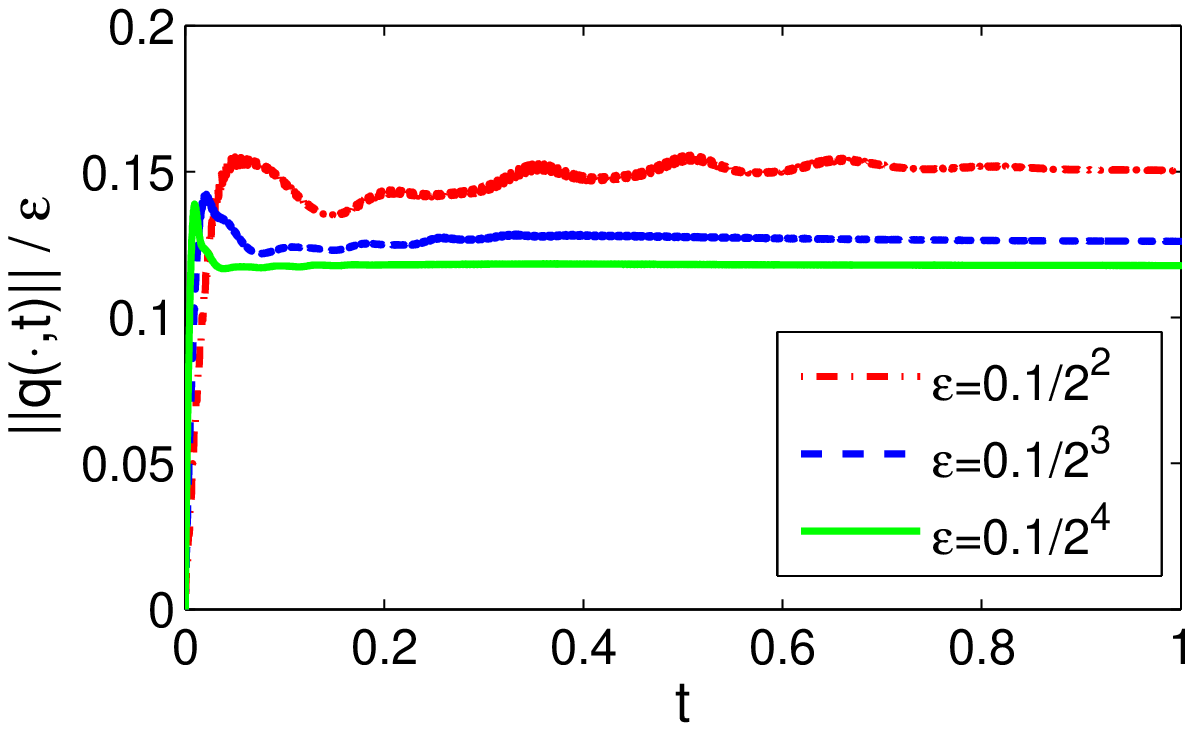}
\end{minipage}%
\hspace{1mm}
\begin{minipage}[t]{0.5\linewidth}
\centering
\includegraphics[height=4cm,width=6.9cm]{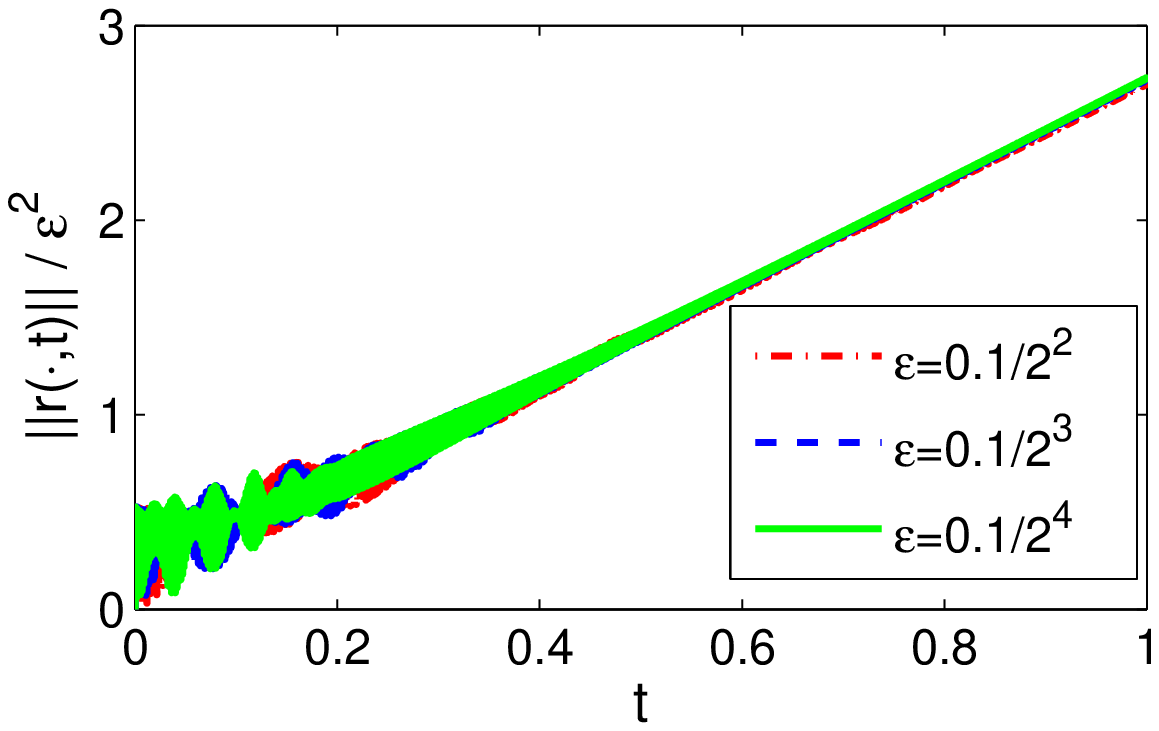}
\end{minipage}
\caption{Quantity $\|q(\cdot,t)\|_{L^2}/\eps$ and $\|r(\cdot,t)\|_{L^2}/\eps^2$ in Example \ref{ex3} under
different $\eps$.}\label{fig55}
\end{figure}
Then we study the convergence rate of the KGZ system to its limit model (\ref{Sch}) or (\ref{Sch2}) as $\eps<\gamma\to0^+$. By (\ref{app}) and (\ref{app2}), we denote
 $\psi_{\rm nls}=\fe^{it/\eps^2}z_{\rm nls}+\fe^{-it/\eps^2}
 \overline{z_{\rm nls}}$, $\phi_{\rm nls}=-2|z_{\rm nls}|^2+I_{\rm nls}$,
  $\psi_{\rm op}=\fe^{it/\eps^2}z_{\rm op}+\fe^{-it/\eps^2}
 \overline{z_{\rm op}}$ and $\phi_{\rm op}=-2|z_{\rm op}|^2+I$,
 and we define
\begin{align*}
&\eta_{\rm nls}^\phi(t):=\|\phi(\cdot,t)-\phi_{\rm nls}(\cdot,t)\|_{L^2},\quad \eta_{\rm nls}^\psi(t):=\|\psi(\cdot,t)-\psi_{\rm nls}(\cdot,t)\|_{L^2},\\
&\eta_{\rm op}^\phi(t):=\|\phi(\cdot,t)-\phi_{\rm op}(\cdot,t)\|_{L^2},\quad \eta_{\rm op}^\psi(t):=\|\psi(\cdot,t)-\psi_{\rm op}(\cdot,t)\|_{L^2}.
 \end{align*}
 Figure \ref{figlimit} shows $\eta_{\rm nls}^\phi(t)/\eps$, $\eta_{\rm nls}^\psi(t)/\eps$, $\eta_{\rm op}^\phi(t)/\eps$ and $\eta_{\rm op}^\psi(t)/\eps^2$ under different $\eps$.
 Finally, to further illustrate the efficiency of the MTI scheme and the super-resolution, we show in Figure \ref{fig7} the numerical solutions obtained by MTI under a fixed large time step $\tau=0.1$ for decreasing $\eps$.

\begin{figure}[t!]
\begin{minipage}[t]{0.5\linewidth}
\centering
\includegraphics[height=4cm,width=6.9cm]{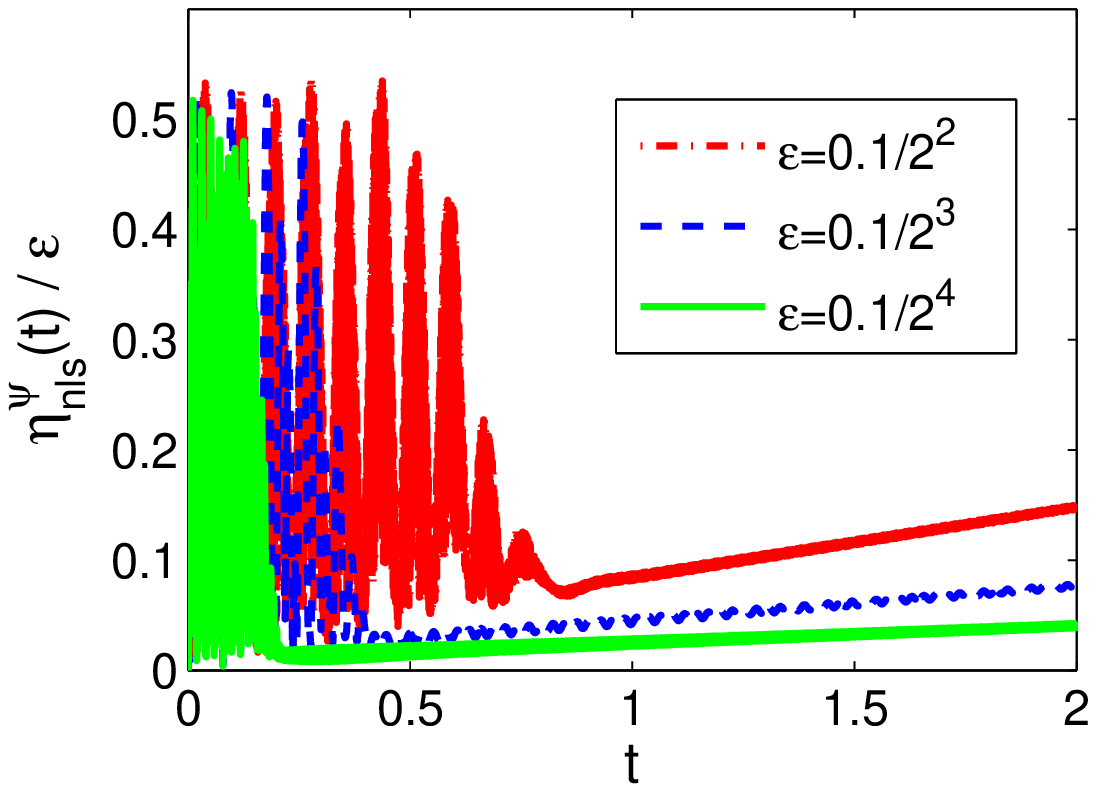}
\end{minipage}%
\hspace{1mm}
\begin{minipage}[t]{0.5\linewidth}
\centering
\includegraphics[height=4cm,width=6.9cm]{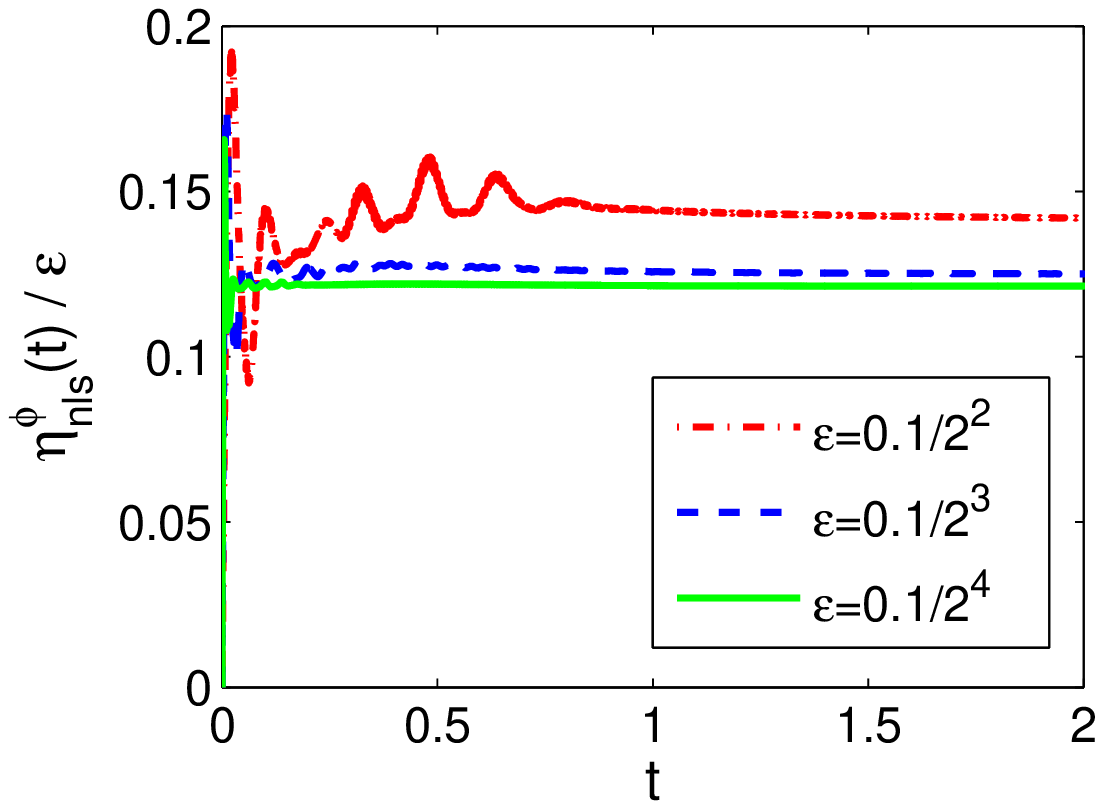}
\end{minipage}
\begin{minipage}[t]{0.5\linewidth}
\centering
\includegraphics[height=4cm,width=6.9cm]{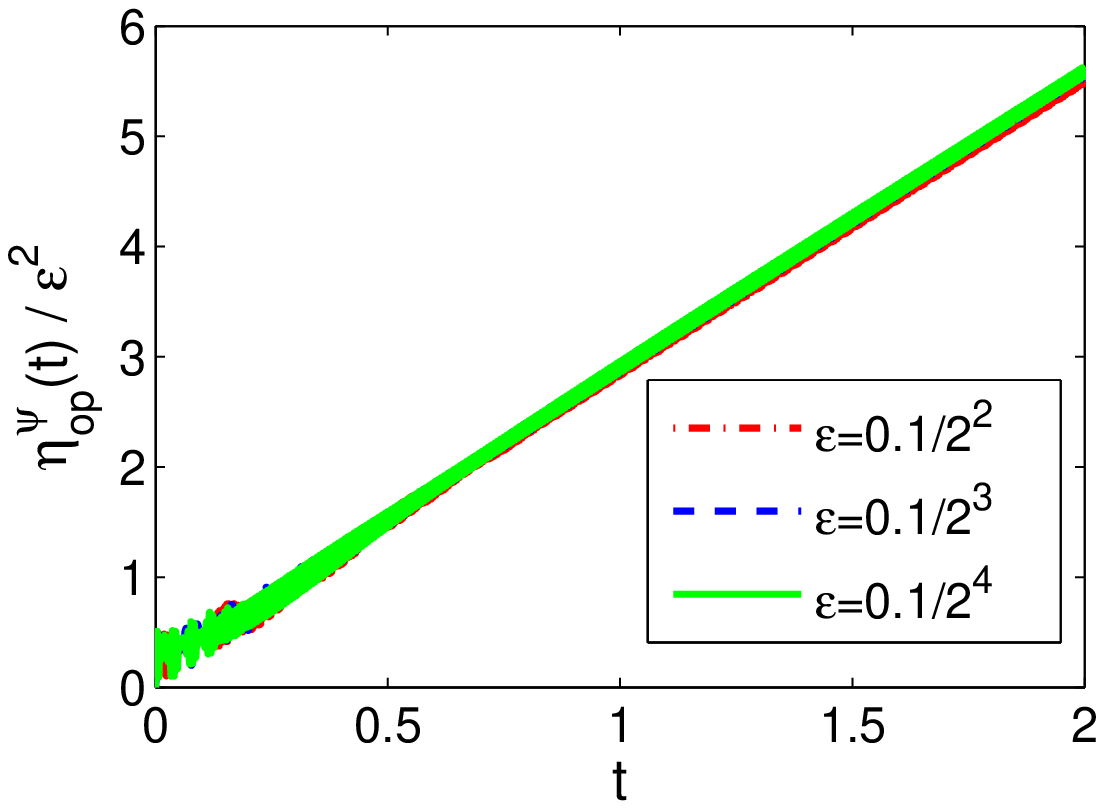}
\end{minipage}%
\hspace{1mm}
\begin{minipage}[t]{0.5\linewidth}
\centering
\includegraphics[height=4cm,width=6.9cm]{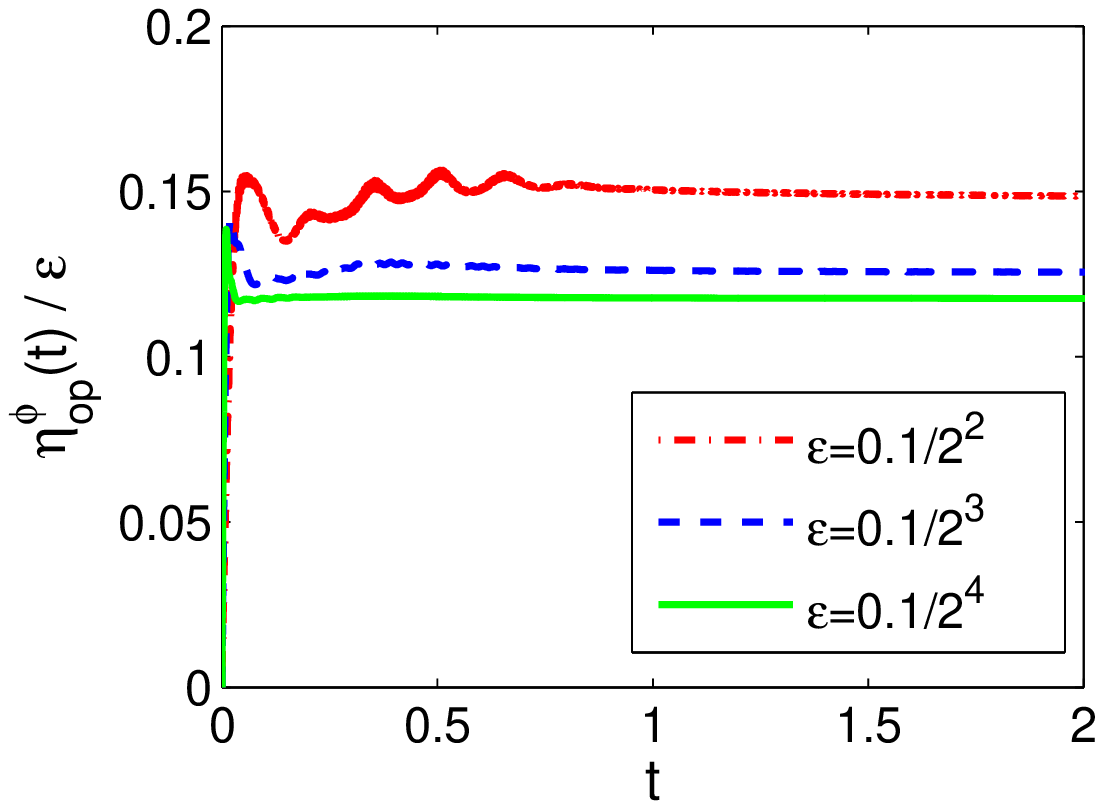}
\end{minipage}
\caption{Convergence from KGZ (\ref{KGZ}) to (\ref{Sch}) or (\ref{Sch2}) in Example \ref{ex3}: the quantities $\eta_{\rm nls}^\psi(t)/\eps,\,\eta_{\rm nls}^\phi(t)/\eps,\,\eta_{\rm op}^\psi(t)/\eps^2$ and $\eta_{\rm op}^\phi(t)/\eps$.}\label{figlimit}
\end{figure}
\begin{figure}[t!]
\begin{minipage}[t]{0.5\linewidth}
\centering
\includegraphics[height=4cm,width=6.9cm]{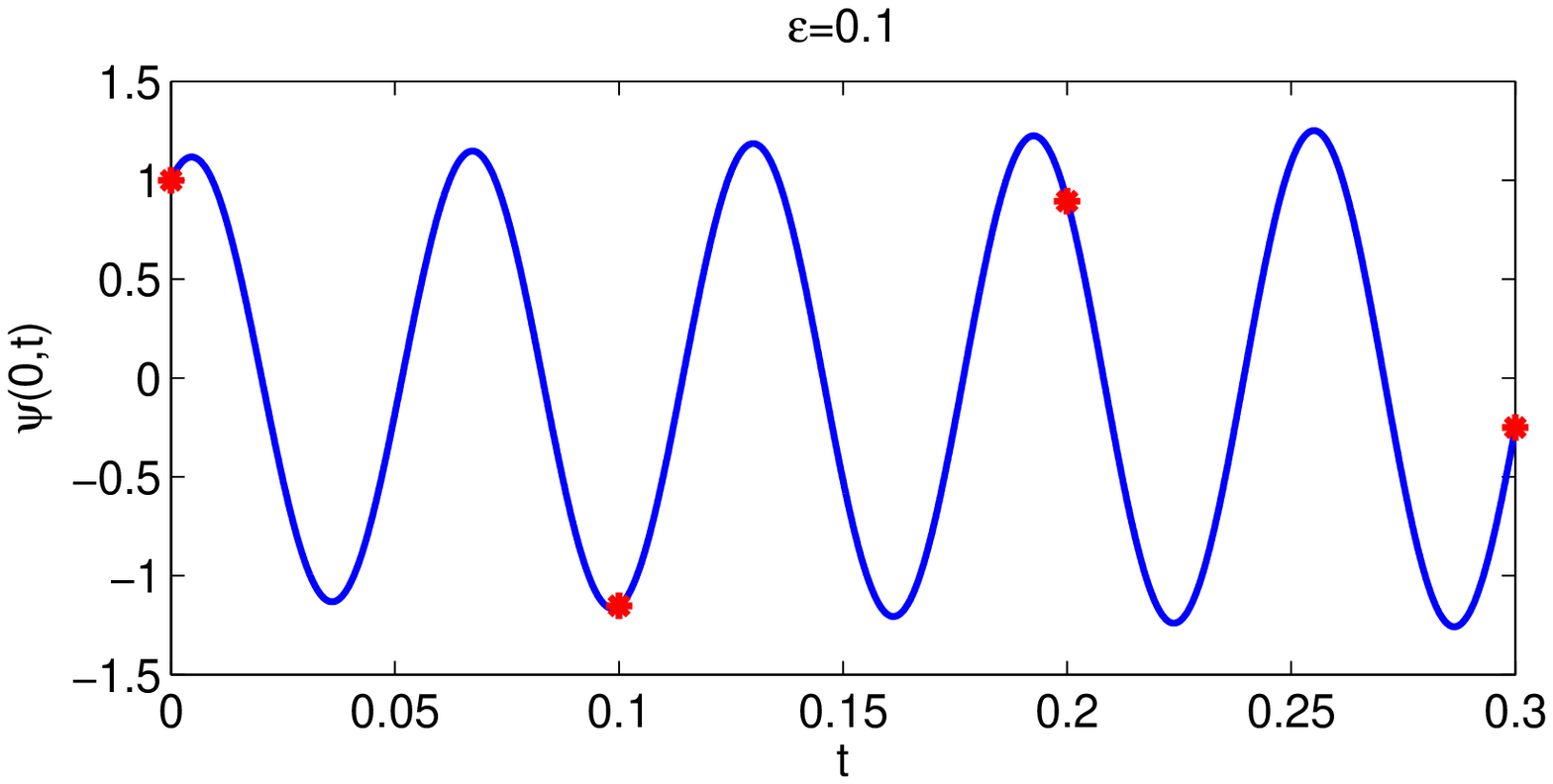}
\end{minipage}%
\hspace{1mm}
\begin{minipage}[t]{0.5\linewidth}
\centering
\includegraphics[height=4cm,width=6.9cm]{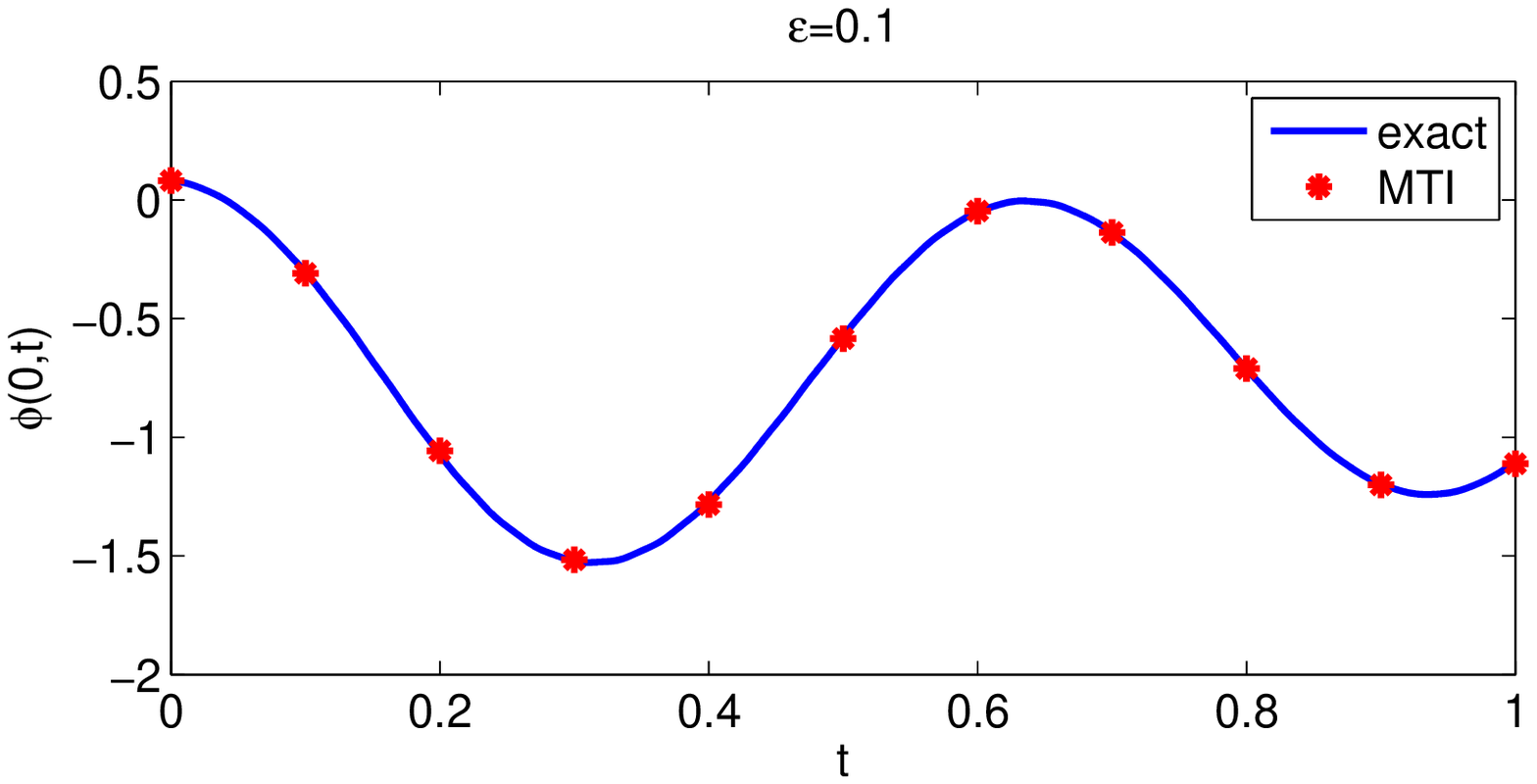}
\end{minipage}
\begin{minipage}[t]{0.5\linewidth}
\centering
\includegraphics[height=4cm,width=6.9cm]{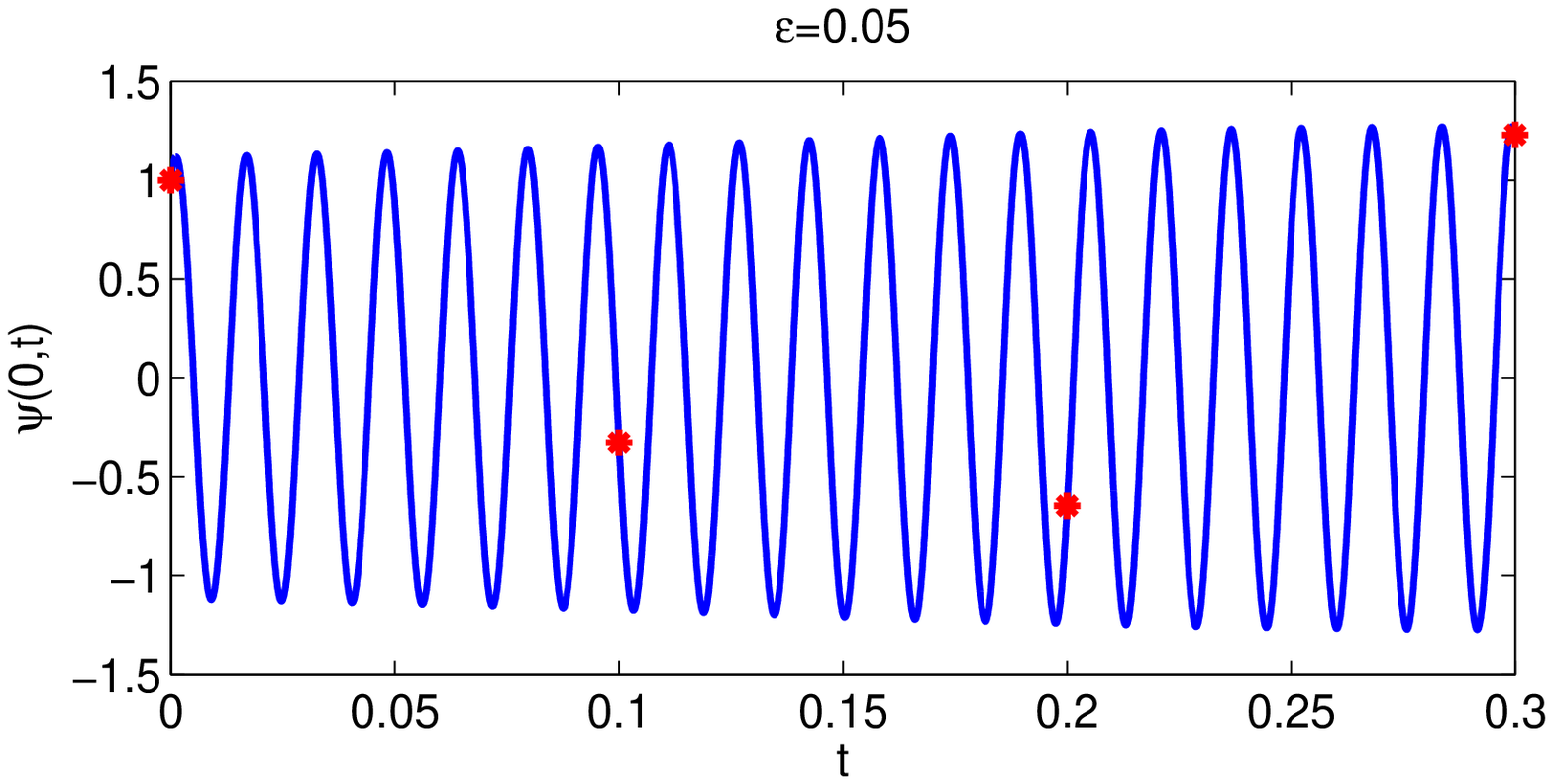}
\end{minipage}%
\hspace{1mm}
\begin{minipage}[t]{0.5\linewidth}
\centering
\includegraphics[height=4cm,width=6.9cm]{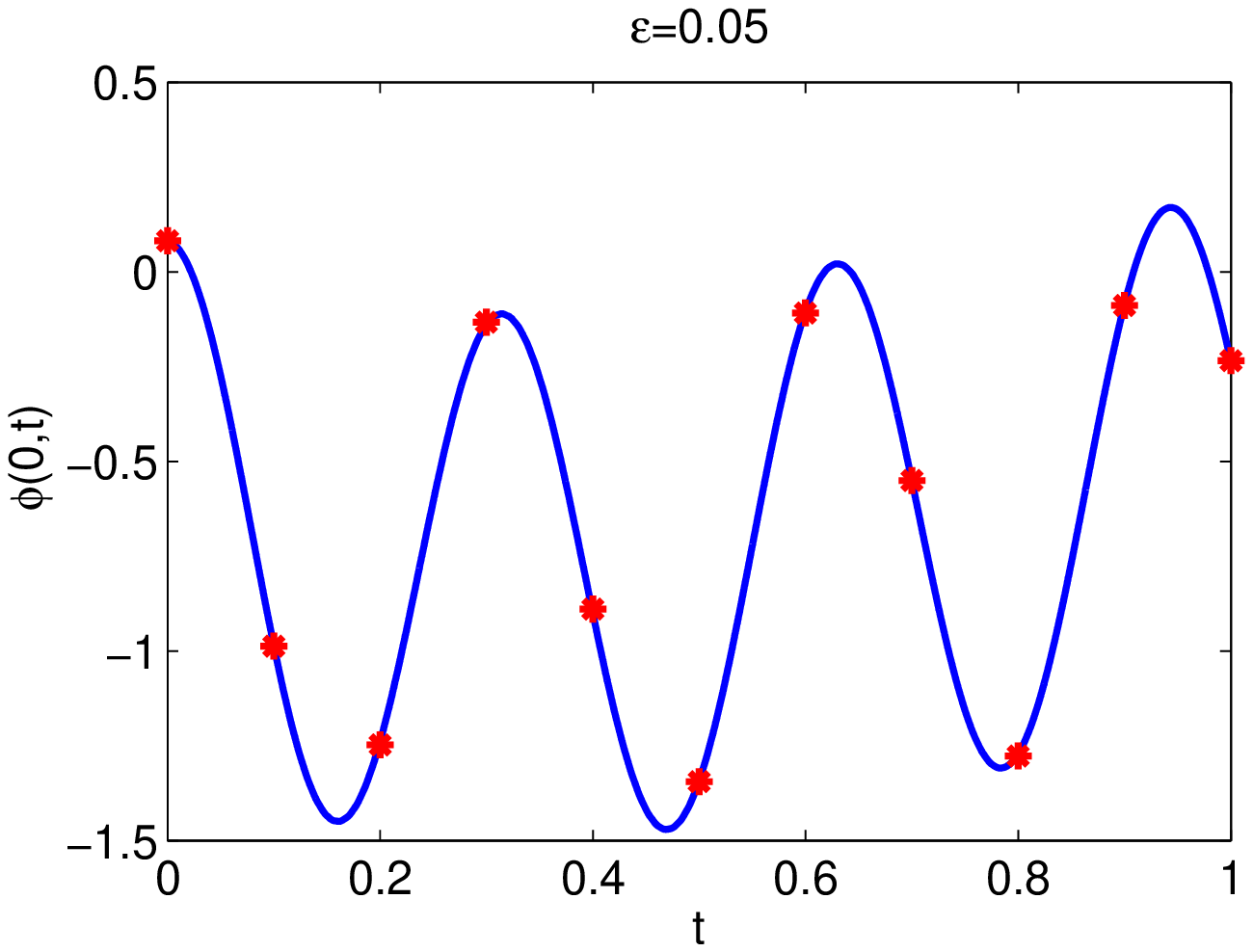}
\end{minipage}
\begin{minipage}[t]{0.5\linewidth}
\centering
\includegraphics[height=4cm,width=6.9cm]{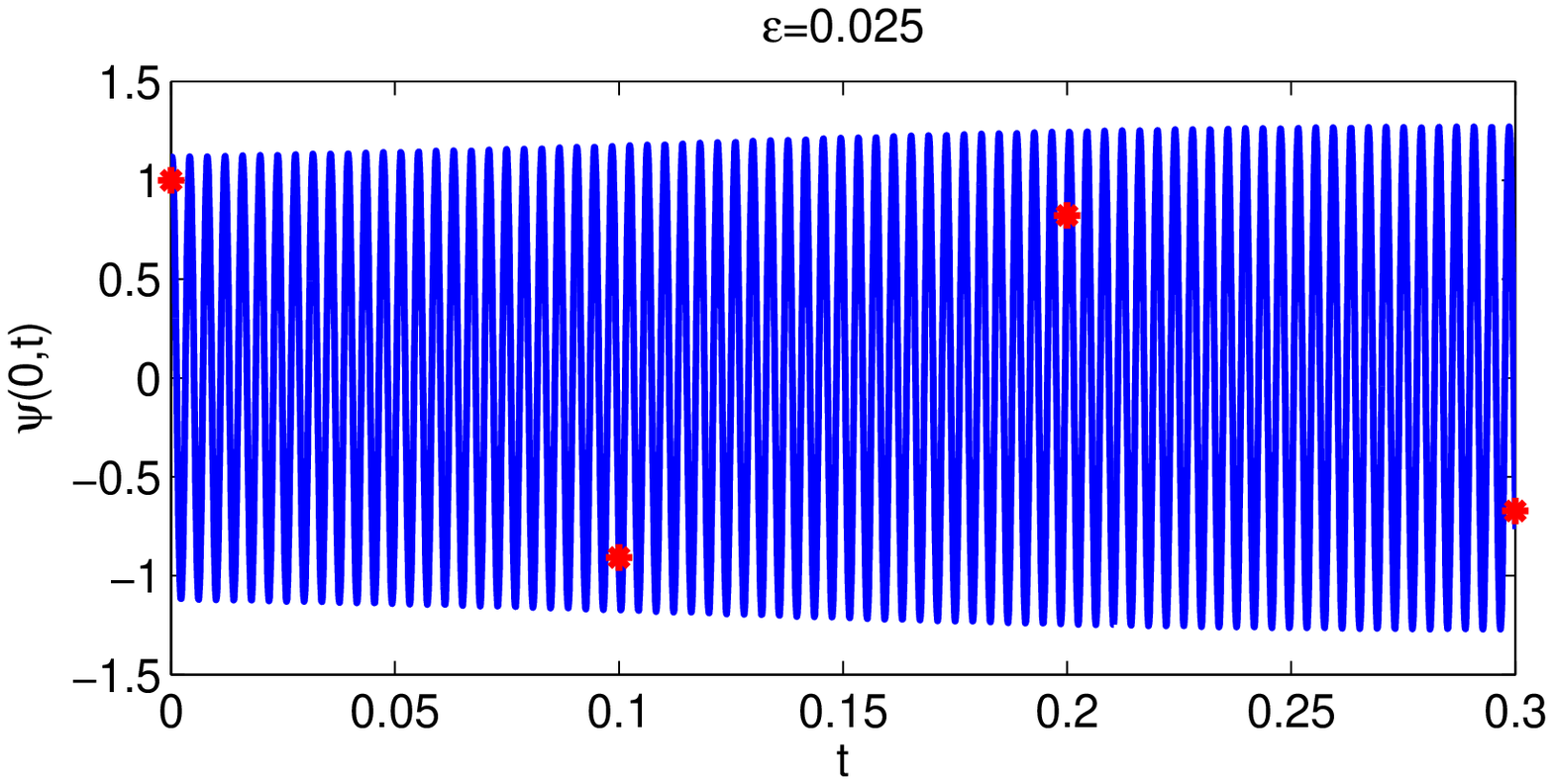}
\end{minipage}%
\hspace{1mm}
\begin{minipage}[t]{0.5\linewidth}
\centering
\includegraphics[height=4cm,width=6.9cm]{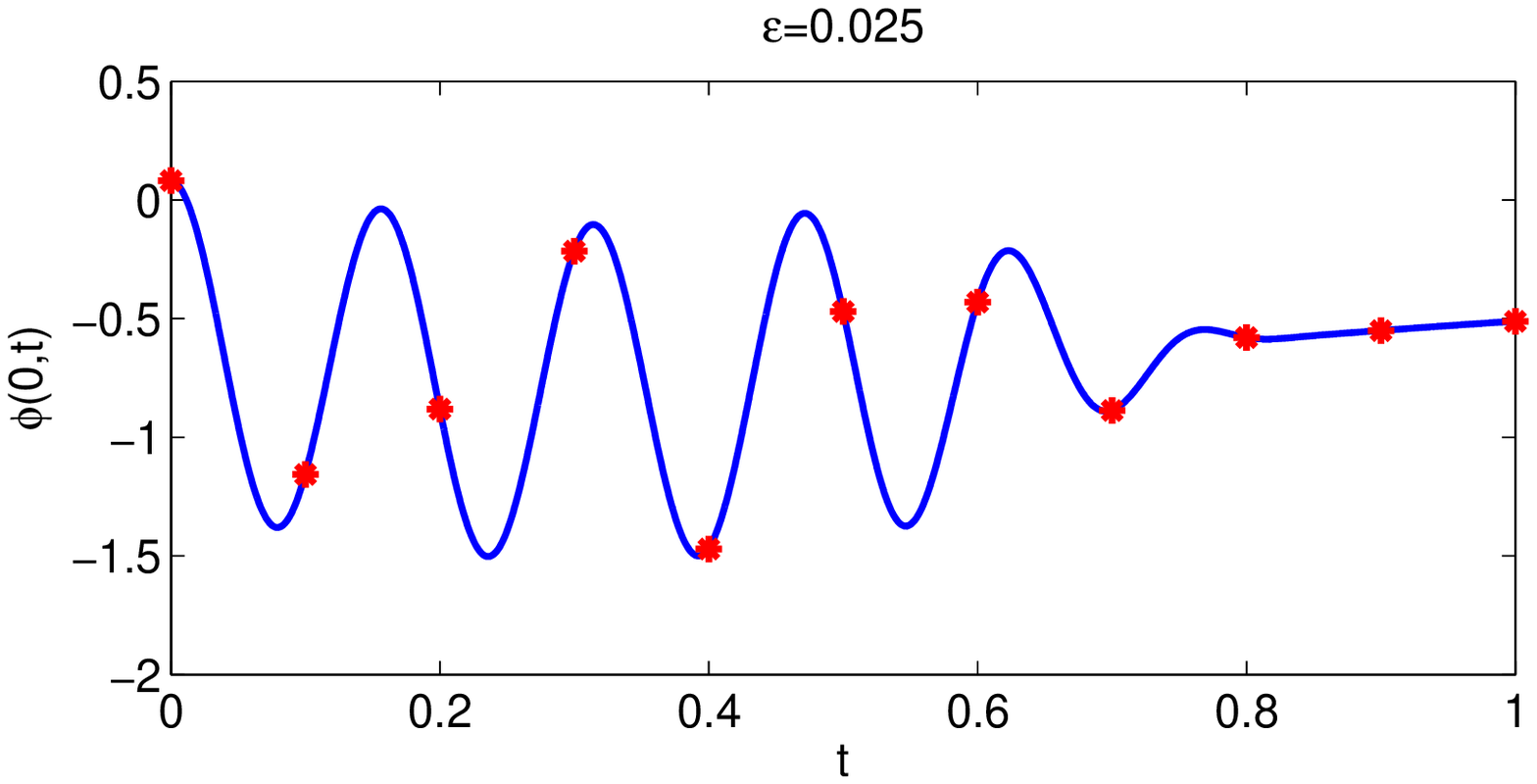}
\end{minipage}
\caption{The solutions $\psi(0,t)$ and $\phi(0,t)$ in Example \ref{ex3} under
different $\eps$: exact profiles and numerical solutions from MTI with fixed $\tau=0.1$.}\label{fig7}
\end{figure}

Based on the numerical results in Figures \ref{fig5}-\ref{fig7}, we have the following observations:

1) The dynamics of the KGZ system (\ref{KGZ}) is captured individually through the components $z,r,q,I$ in the decomposition (\ref{KGZ decomp}). Among them, $I$ and $q$ carry the fast outing initial layer caused by the incompatible initial data and the wave operator, respectively, while $z$ and $r$ remain rather localized (cf. Figure \ref{fig5}). To avoid the expanding domain for computation, one could consider an absorbing boundary condition for the equation of $q$ to gain more efficiency in practical simulation.

2) The energy error of the MTI scheme converges linearly in time (see Figure \ref{fig6}). The error is not only uniformly bounded for $\eps\in(0,1]$, but it also seems to have a super-convergence in $\eps$ in the limit $\eps\to0$ (see Figure \ref{fig6}).

3) The components $q$ and $r$ are highly oscillatory in time (see Figure \ref{fig5}), but they vanish at
$O(\gamma)$ and $O(\eps^2)$ (see Figure \ref{fig55}), respectively, in the limit $\eps<\gamma\to0^+$. This verifies our estimates in Proposition \ref{prop}.

4) The KGZ system (\ref{KGZ}) converges to the limit model (\ref{Sch}) at the first order rate (see Figure \ref{figlimit}), i.e.,
$\eta_{\rm nls}^\psi(t)=O(\eps)$ and $\eta_{\rm nls}^\phi(t)=O(\eps)$ as $ \eps\to0,$
while its convergence rate to the semi-limit model (\ref{Sch2}) is improved to be quadratic in $\psi$, i.e., $\eta_{\rm op}^\psi(t)=O(\eps^2)$.

5) The MTI scheme has super-resolution to the temporal oscillations. It can correctly capture the oscillation with a fixed time step, no matter how strong the oscillation becomes (see Figure \ref{fig7}). This significantly improves the efficiency of computation compared to standard numerical methods that need to fully resolve the oscillations.

\section{Conclusion}\label{sec:5}
We considered the numerical solution of the Klein-Gordon-Zakharov (KGZ) system in the simultaneous high-plasma-frequency and subsonic limit regime, where two independent small parameters $0<\eps,\gamma\leq1$ are involved. When $\eps,\gamma\to0$, the solution of the KGZ equations exhibits complicated highly oscillatory behaviour including fast temporal oscillations and rapid out-going initial layers, which makes standard numerical methods suffer. By applying a multiscale expansion to the solution in the critical case $\eps<\gamma$, we decomposed KGZ into a consistent formulation with milder oscillations and an explicit description of the initial layer. Formal estimates were established for the decomposed system to explain the advantage of the formulation. Based on the decomposed formulation, we proposed a multiscale time integrator Fourier spectral/pseudospectral method for solving KGZ, which is uniformly and optimally accurate for all $0<\eps<\gamma\leq1$. Various numerical experiments were conducted to illustrate the efficiency and accuracy of the proposed scheme over existing methods. Convergence rates of the KGZ system to its limit/semi-limit model as $\eps<\gamma\to0^+$ were studied numerically.

\appendix
\section{A benchmark algorithm} As a benchmark for reference solution and comparisons, we briefly present the exponential integrator Fourier spectral method \cite{KGZ,XZ} in 1D which is a classical scheme \cite{Deuflhard,Lubichbook} for solving the KGZ system.

Taking the Fourier transform of the KGZ system (\ref{KGZ}) in 1D and using
the Duhamel's formula, one gets
\begin{align*}
&\widehat{\psi}_l(t_{n+1})=\cos(\omega_l\tau)\widehat{\psi}_l(t_n)+\frac{\sin(\omega_l\tau)}{\omega_l}\widehat{\psi}_l'(t_n)
-\int_0^\tau\frac{\sin(\omega_l(\tau-s))}{\eps^2\omega_l}\widehat{(\psi\phi)}_l(t_n+s)ds,\\
&\widehat{\phi}_l(t_{n+1})=\cos(\theta_l\tau)\widehat{\phi}_l(t_n)+\frac{\sin(\theta_l\tau)}{\theta_l}\widehat{\phi}_l'(t_n)
-\theta_l\int_0^\tau\sin(\theta_l(\tau-s))\widehat{(\psi^2)}_l(t_n+s)ds,\\
&\widehat{\psi}_l'(t_{n+1})=-\omega_l\sin(\omega_l\tau)\widehat{\psi}_l(t_n)+\cos(\omega_l\tau)\widehat{\psi}_l'(t_n)
-\int_0^\tau\frac{\cos(\omega_l(\tau-s))}{\eps^2}\widehat{(\psi\phi)}_l(t_n+s)ds,\\
&\widehat{\phi}_l'(t_{n+1})=-\theta_l\sin(\theta_l\tau)\widehat{\phi}_l(t_n)+\cos(\theta_l\tau)\widehat{\phi}_l'(t_n)
-\theta_l^2\int_0^\tau\cos(\theta_l(\tau-s))\widehat{(\psi^2)}_l(t_n+s)ds.
\end{align*}
By applying the trapezoidal rule to approximate the integrals, the explicit Deuflhard-type exponential integrator (EI) Fourier spectral method reads: $\psi^n(x)\approx\psi(t_n,x)$, $\dot{\psi}^n(x)\approx\partial_t\psi(t_n,x)$,
$\phi^n(x)\approx\phi(t_n,x)$, $\dot{\phi}^n(x)\approx\partial_t\phi(t_n,x)$, where for $n\geq0$,
$l=-N/2,\ldots,N/2-1$,
\begin{equation}\tag{A.1}\label{EWI}
\begin{split}
&\widehat{(\psi^{n+1})}_l=\cos(\omega_l\tau)\widehat{(\psi^n)}_l
+\frac{\sin(\omega_l\tau)}{\omega_l}\widehat{(\dot{\psi}^n)}_l
-\frac{\tau\sin(\omega_l\tau)}{2\eps^2\omega_l}\widehat{(\psi^n\phi^n)}_l,\\
&\widehat{(\phi^{n+1})}_l=\cos(\theta_l\tau)\widehat{(\phi^n)}_l
+\frac{\sin(\theta_l\tau)}{\theta_l}\widehat{(\dot{\phi}^n)}_l
-\frac{\tau\theta_l}{2}\sin(\theta_l\tau)\widehat{((\psi^n)^2)}_l, \\
&\widehat{(\dot{\psi}^{n+1})}_l=-\omega_l\sin(\omega_l\tau)\widehat{(\psi^n)}_l
+\cos(\omega_l\tau)\widehat{(\dot{\psi}^n)}_l
-\frac{\tau}{2\eps^2}\left[\cos(\omega_l\tau)\widehat{(\psi^n\phi^n)}_l
+\widehat{(\psi^{n+1}\phi^{n+1})}_l\right],\\
&\widehat{(\dot{\phi}^{n+1})}_l=-\theta_l\sin(\theta_l\tau)\widehat{\phi}_l^n
+\cos(\theta_l\tau)\widehat{(\dot{\phi}^n)}_l
-\frac{\theta_l^2\tau}{2}\left[\cos(\theta_l\tau)\widehat{((\psi^n)^2)}_l+\widehat{((\psi^{n+1})^2)}_l\right].
\end{split}
\end{equation}

\bibliographystyle{SIAM}

\end{document}